\documentclass[reqno, 12pt]{amsart}

\usepackage{lpform}
\usepackage{tikz}
\usepackage{adjustbox}
\usepackage{amssymb}
\usepackage{amsmath}
\usepackage{mathtools}
\usepackage{xfrac}
\usepackage[T1]{fontenc}
\usepackage{caption}
\usepackage{booktabs}
\usepackage{nicefrac}
\usepackage{color}
\usepackage{colortbl}
\usepackage[hyphens]{url}
\usepackage{hyperref}
\usepackage[hyphenbreaks]{breakurl}
\usepackage[12pt]{moresize}
\usepackage{amsaddr}

\usepackage{placeins}
\usepackage{float}

\newtheorem{thm}{Theorem}[section]
\newtheorem{prop}[thm]{Proposition}

\renewcommand{\lpsubjectto}{\normalfont s.t.}

\begin{document}

\title[Exact algorithms for BPCCSPs]{Exact algorithms for budgeted prize-collecting covering subgraph problems}



\author[N. Morandi, R. Leus, and H. Yaman]{Nicola Morandi, Roel Leus, and Hande Yaman}
\email{nicola.morandi@kuleuven.be}
\email{roel.leus@kuleuven.be}
\address{ORSTAT, Faculty of Economics and Business, KU Leuven, Belgium}
\email{hande.yaman@kuleuven.be}

\begin{abstract}
We introduce a class of budgeted prize-collecting covering subgraph problems.  For an input graph with prizes on the vertices and costs on the edges, the aim of these problems is to find a connected subgraph such that the cost of its edges does not exceed a given budget and its collected prize is maximum. A vertex prize is collected when the vertex is visited, but the price can also be partially collected if the vertex is covered, where an unvisited vertex is covered by a visited one if the latter belongs to the former's neighbourhood. A capacity limit is imposed on the  number of vertices that can be covered by the same visited vertex. Potential application areas include network design and intermodal transportation.  We develop a branch-and-cut framework and a Benders decomposition for the exact solution of the problems in this class.  We observe that the former algorithm results in shorter computational times on average, but also that the latter can outperform the former for specific instance settings.
Finally, we validate our algorithmic frameworks for the cases where the subgraph is a tour and a tree, and for these two cases we also identify novel symmetry-breaking inequalities.
\end{abstract}

\maketitle

\section{Introduction}
According to a recent report by Capgemini~\cite{lastmiledelivery}, last-mile delivery accounts for up to $41\%$ of overall supply chain costs, and this proportion is only expected to rise, as the volume of retail e-commerce sales is forecast to see a yearly increase of between $14\%$ and $18\%$ worldwide \cite{ecommerce}. In the ordering process, customers are often asked to choose between a more expensive home delivery and a cheaper delivery to a nearby pick-up location such as a postal office or a supermarket. The latter scenario leads to a two-layered ring-star network, where the upper-layer ring represents the trajectory of the delivery truck, and all customers' paths to a collection point form a disjoint union of stars as the lower layer. A relevant objective in the design of such a supply network is the cost of the network, and the corresponding minimization problem is known as the Ring-Star problem (RSP) \cite{Labbe04}.

Popular fast delivery options impose a strict schedule on the truck drivers, and put a significant strain on the logistics providers.  In such a context, visiting all the customers that desire home delivery at the minimum cost might not always be feasible. In Belgium, for instance, the leading postal-service operator BPost announced in November 2020, following a surge in online purchases, that it might be unable to deliver all packages to the addressee at home, forcing them instead to pick up the package at one of the 2300 pick-up points across the country \cite{brustimes}.  The profit associated with a customer that is referred to a pick-up point rather than served directly, can be considered to be only a certain fraction of the profit in case of home delivery, mainly due to a loss of goodwill from the part of the customer. Specifically for this setting, a more accurate objective is the total collected profit; the resulting maximization problem is known as the Time Constrained Maximum Covering Salesman problem  (TCMCSP) \cite{Najiazimi14, Ozbaygin2016}.

A closely related scenario arises in telecommunication services when high-speed broadband connection is brought to the customers via a tree-star network, whose upper layer is designed with a fiber-to-the-curb architecture. When the infrastructure's cost is the main concern for the providers, a number of problems arise that are known in the literature as tree-star problems \cite{Leitner17}. In particular, when further constraints are imposed on the size of the stars in the lower layer, the latter problems reduce  to the capacitated connected facility location problem \cite{Leitner11}.

In this article, we introduce the class of Budgeted Prize-Collecting Covering Subgraph problems (BPCCSPs), which generalize the TCMCSP by the inclusion of constraints on the size of the lower-layer stars, and which also include a budgeted prize-collecting variant of the tree-star problems.
Furthermore, in the class of BPCCSPs we allow for any type of connected upper-layer subgraph, with a tree and a tour as special cases.
The contribution of this paper consists in generalizing existing MIP (mixed-integer programming) models in order to solve the BPCCSPs and in designing a branch-and-cut framework for finding optimal solutions.  We validate our methods on realistic-sized instances for the most common cases where the subgraph is a tour and a tree.  We also develop a Benders decomposition as a viable alternative to the branch-and-cut for specific 
instance settings.  Finally, we introduce novel symmetry-breaking inequalities for the tree and the tour case.

The remainder of this paper is structured as follows.  In Section 2, we first provide an overview of the related literature. In Section 3, we describe our branch-and-cut and Benders decomposition algorithms. A detailed discussion of the tree and the tour case is given in Sections 4 and 5, respectively. We present and discuss our computational results in Section~6, and we conclude in Section 7.

\section{Related literature}
\subsection{Budgeted prize-collecting problems}
\emph{Prize-collecting} problems on a graph arise when visiting all the vertices of the graph is not desirable or feasible.  Prizes are then associated to the vertices, and the optimization needs to strike a balance between the routing costs and the collected prizes. The first example dates back to the work by Segev \cite{Segev87}, who introduced the node-weighted Steiner tree problem, in which a tree is to be found that minimizes the total cost of the edges minus the total prize of the vertices. Shortly after, Balas \cite{Balas89} introduced the Prize-collecting TSP, aiming at finding a tour that minimizes the total cost of its edges plus the prize of the unvisited vertices. Following the insights provided by the survey by Feillet et al.~\cite{Feillet05}, we distinguish between four approaches in the literature to model this trade-off between costs and prizes.
\begin{enumerate}
\item[i.] In a prize-collecting problem, a linear combination of routing costs and penalties for not visiting some vertices is usually minimized. An early example of this approach is the Profitable Tour problem by Dell'Amico et al.~\cite{Dell'amico95}. Equivalently, uncollected prizes can be minimized in the objective instead of penalties. We find this approach, among other studies, in Ljubi\'c et al.~\cite{Ljubic06} for the prize-collecting Steiner tree problem (PCSTP).
\item[ii.] Another natural approach resorts to \emph{bi-objective} models. Early examples are Keller \cite{Keller85} and Keller and Goodchild~\cite{Goodchild88}. More recently, Leitner et al.~\cite{Leitner13, Leitner15} proposed several methods to find Pareto-optimal solutions to the PCSTP\@. Specifically for rooted tree graphs, Coene et al.~\cite{Coene13} provided a FPTAS to find all the efficient solutions to an analogous bi-objective problem.
\item[iii.] A \emph{quota} constraint requires a minimum total prize to be collected. This type of constraint was investigated, among others, by Haouari et al.~\cite{Haouari10} while assessing the strength of MIP formulations for the PCSTP, and by Ljubi\'c et al.~\cite{Ljubic14} while solving a prize-collecting problem in network design.
\item[iv.] Prize-collecting problems can also feature a \emph{budgeted} variant, where the collected prize of the visited vertices is maximized, while a budget on the total routing cost is imposed. The first example of this variant is the Orienteering problem (OP) by Tsiligirides \cite{Tsiligirides84} and by Laporte and Martello \cite{Laporte90}, which was effectively solved via a branch-and-cut algorithm in Fischetti et al.~\cite{Fischetti98}. Johnson et al.~\cite{Johnson00} introduced the budgeted variant of the PCSTP, while more recently Paul et al.\ \cite{Paul20} described a $2$-PTAS for the case with unit profits.
Costa et al.~\cite{Costa09} studied a budgeted prize-collecting tree problem with additional hop constraints, in  which a solution vertex should not be more than a given number of solution edges away from the depot. Sinnl and Ljubi\'c \cite{Sinnl16} solved this problem using a MIP formulation involving layered graphs.
\end{enumerate}
Finally, for relevant studies on prize-collecting problems with disconnected, directed or generalized spanning subgraphs, which do not belong to the scope of this work, we refer the reader to \cite{Alvarez13,Bateni12,chimani07, Golden08,Prodon10}.
\subsection{Coverage in network design, routing, and facility location}
The idea of \emph{coverage} appears in problems at the intersection of network design and facility location with routing. 
For an overview of both capacitated and uncapacitated facility location problems, we refer the reader to the work by Verter \cite{Verter11} and references therein. We identify six different groups of related literature that approach this intersection 
from different and equally interesting angles.
\begin{enumerate}
\item[i.] An early example of coverage in routing was introduced as a variant of the TSP by Current \cite{Current89}, who named it the Covering Salesman problem (CSP) and proposed a heuristic algorithm for its solution. In the CSP, all the vertices have to be either visited by the tour or be such that their distance from a visited vertex is not greater than a given value. Shortly after, Gendreau et al.\ \cite{Gendreau97} generalized the CSP by introducing the Covering Tour problem, where some vertices have to be visited and some other ones have to be covered.
\item[ii.] The Maximum Covering Cycle problem (MCCP) aims at finding a simple cycle in a non-complete undirected graph such that it maximizes the sum of the number of its vertices and of those adjacent to it. It was introduced by Grosso et al.~\cite{Grosso16}; in this latter work, it was also shown that the MCCP is NP-complete, and an exact constraint generation approach was developed and tested both on instances from the Hamiltonian Cycle problem and on randomly generated ones. More recently, \'Alvarez-Miranda and Sinnl\ \cite{Alvarez20} identified valid and lifted inequalities for the MCCP, and designed an effective branch-and-cut framework for its solution, reporting exact solutions for large instances in just a few seconds.
\item[iii.] When a facility is extensive, i.e., it cannot be modelled as a single vertex, its coverage capability is captured by its \emph{accessibility}, defined as the sum of the distances from each customer to the facility, with inverted sign. Labb\'e et al.\ \cite{Labbe98} proposed a bi-objective study in which the network cost is minimized and its accessibility maximized simultaneously. A bi-objective approach was also proposed by Current \cite{Current94} while introducing the Median Tour problem, which aims at finding the Pareto-efficient tours made by $p$ vertices out of the total $n$. A single-objective minimum cost approach has instead been studied by Labb\'e et al.\ \cite{Labbe05} for median cycles with a lower bound on the accessibility.
\item[iv.] Two-layer network design problems, where the lower layer is a star, also provide a way to model coverage. The RSP seeks to find an optimal ring such that its edge cost plus the total cost of assigning all the vertices outside of the ring to the cheapest vertex on the ring is minimum. Labb\'e et al.\ \cite{Labbe04} solved the RSP exactly and identified facet-defining inequalities for the convex hull of the feasible solutions. Lee et al.\ \cite{Lee98} developed a branch-and-cut algorithm to solve a Steiner variant of the RSP, where terminal nodes are considered. Baldacci et al.\ \cite{Baldacci07} proposed a branch-and-cut approach to the $m$-rings capacitated variant of the RSP, where a capacity is assigned to every ring. This latter variant is solved via a branch-and-cut-and-price algorithm by Hoshino and de Souza in \cite{Hoshino12}. The upper ring layer of the RSP can be relaxed to a $2$-edge-connected subgraph, and facet-defining inequalities for the resulting problem were identified by Fouilhoux et al.\ \cite{Fouilhoux12}, in a work where furthermore a branch-and-cut algorithm was devised. On the other hand, the Tree Star analogous problem, where the upper layer of the network is a tree, was effectively solved by Leitner et al.\ \cite{Leitner17} via an algorithmic framework that is also capable of solving closely related problems such as the Rent-or-Buy and the Connected Facility Location problem. Lee et al.\ \cite{Lee96} solved the corresponding Steiner variant of the problem, where a subset of terminal vertices needs to be spanned by the upper-layer Steiner tree.
\item[v.] The Time-Constrained  Maximal Covering Salesman problem (TCMCSP) was introduced by Naji-Azimi et al.\ \cite{Najiazimi14}. It distinguishes between customer vertices and facility vertices, and seeks to find a tour through the facility vertices such that its routing cost is not greater than a given budget, and it maximizes the number of customers who are not further away from the closest visited facility than a given value. The TCMCSP in its original form does not involve prizes, but it was extended to a budgeted prize-collecting problem by Ozbaygin et al.\ \cite{Ozbaygin2016}, who introduced the TCMCSP with partial cover: no distinction is made between customers and facilities, and every vertex has its own prize, to be collected entirely if the vertex is in the tour, or instead partially if the vertex lies within a given distance from a vertex in the tour. The TCMCSP with partial cover is closely related to our work; our model is a generalization of the one studied in \cite{Ozbaygin2016}, in which we incorporate capacities into the lower coverage layer and a non-tour subgraph as the upper visiting one.
\item[vi.] Lastly, the Connected Facility Location (conFL) problem arises when a graph consists of facility vertices and customer vertices, and costs are associated to opening facilities, assigning customers to facilities and connecting the facilities. Customers need to be assigned to facilities, and facilities in turn need to be connected via a Steiner tree whose cost is minimum. Karget and Minkoff \cite{Karget00} first introduced the problem, and Gupta et al.\ \cite{Gupta01} coined the name conFL. A hybrid metaheuristic for conFL, together with an exact method to assess its quality, was provided by Ljubi\'c \cite{Ljubic07} and validated on instances obtained by combining facility location and Steiner tree instances. Gollowitzer and Ljubi\'c \cite{Gollowitzer11} compared different models for conFL and proposed two branch-and-cut algorithms. A branch-and-cut and a branch-and-cut-and-price algorithm were devised by Leitner and Raidl \cite{Leitner11} to solve the Capacitated conFL, with capacities on the facilities and prizes on the customers. Customers can potentially be unassigned and their corresponding uncollected prize is minimized in the objective function. Gollowitzer et al.~\cite{Gollowitzer13} proposed a cutting-plane algorithm for the case where capacities are also imposed on the edges connecting the facilities, and not only on the facilities themselves.
\end{enumerate}

\section{
Problem statement and solution methods}
\subsection{Problem statement}
Let $G=\left( V, E\right)$ be a graph, where $V$ denotes the set of vertices and $E$ the set of edges. The vertex $0\in V$ is the \emph{depot}. We are given a budget $L>0$ and a cost $\ell_e\geq 0$ for each edge $e\in E$. For any vertex $v\in V$, the \emph{neighbourhood}  $N_{v}$ is the subset of vertices in $V\smallsetminus\left\lbrace v\right\rbrace$ 
that can cover vertex $v$. We denote by $c_v$ the \emph{coverage capacity} of vertex $v$.

Our aim is to find a connected subgraph of $G$ such that the total cost of its edges is not greater than $L$. Additional requirements such as degree constraints can be imposed on this subgraph. The vertices of the subgraph are called  \emph{visited} vertices; the depot must be visited. Each unvisited vertex can be covered by at most one visited vertex in its neighbourhood, while the number of vertices covered by a given visited vertex~$v$ cannot exceed its coverage capacity~$c_v$.

The prize earned by visiting vertex $v\in V$ is denoted by $p_v$ and the prize resulting from covering vertex $v \in V\smallsetminus\left\lbrace 0\right\rbrace$ via vertex $w\in N_{v}$ is $q_{vw}$. We assume that all prizes are non-negative. A BPCCSP seeks to find a feasible connected subgraph that maximizes the collected prize.  This problem statement is generic and leads to a class of different BPCCSPs, dependent on the constraints imposed on the constructed subgraph. 
\subsection{
MIP model}
We use the following binary variables to model a BPCCSP: 
\begin{align*}
\phantom{z_{vw}=}
&\begin{aligned}
\mathllap{x_{e}=} & \begin{cases} 1 & $if $e\in E$ is used in the subgraph$, \\ 0 & $otherwise,$\end{cases}
\end{aligned}\\
&\begin{aligned}
\mathllap{y_{v}=} &\begin{cases} 1 & $if $v\in V\smallsetminus\left\lbrace 0\right\rbrace $ is visited$, \\ 0 & $otherwise,$\end{cases}
\end{aligned}\\
&\begin{aligned}
\mathllap{z_{vw}=} &\begin{cases} 1 & $if $v\in V\smallsetminus\left\lbrace 0\right\rbrace$ is covered by $w\in N_{v},\\ 0 & $otherwise.$\end{cases}
\end{aligned}
\end{align*}
For any set $S\subset V$, let $\delta\left( S\right)=\left\lbrace e \in E :|e\cap S|=1\right\rbrace$. A BPCCSP can be modeled as follows:
\begin{lpformulation}[]
\lpobj{\normalfont max}{\sum_{v\in V\smallsetminus\left\lbrace 0\right\rbrace}\left( p_{v}\; y_{v}+\sum_{w\in N_{v}}q_{vw}\; z_{vw}\right)}
\lpeq{\label{glp:length_threshold}\sum_{e\in E}\ell_{e}x_{e}\leq L}{}
\lpeq{\label{glp:visorcov}y_{v}+\sum_{w\in N_{v}}z_{vw}\leq 1&&}{v\in V\smallsetminus\left\lbrace 0\right\rbrace}
\lpeq{\label{glp:coverage}\sum_{v\in V\smallsetminus\left\lbrace 0\right\rbrace\; :\; w\in N_{v}}z_{vw}\leq \begin{cases}c_{w}y_{w}& w\neq 0\\ c_{0}& w=0\end{cases}&&}{w\in V}
\lpeq{\label{glp:liftednconn}
\sum_{e\in\delta\left( S\right)}x_{e}\geq y_{v}+\sum_{w\in N_{v}\cap S}z_{vw} &&}{S\subseteq V\smallsetminus\left\lbrace 0\right\rbrace, \; v\in S}
\lpeq{\label{tree:xy}x_{\left\lbrace v,w\right\rbrace}\leq y_{v} &&}{\left\lbrace v,w\right\rbrace\in E}
\lpeq{\label{glp:subgraph}  \left(\mbox{additional constraints on }x \mbox{ and } y \right)&&}{ }
\lpeq{x_{e}\in\left\lbrace 0,1\right\rbrace&&}{e\in E}
\lpeq{y_{v}\in\left\lbrace 0,1\right\rbrace&&}{v\in V\smallsetminus\left\lbrace 0\right\rbrace}
\lpeq{\label{glp:end}z_{vw}\geq 0&&}{v\in V\smallsetminus\left\lbrace 0\right\rbrace,\; w\in N_{v}}
\end{lpformulation}

Constraint \eqref{glp:length_threshold} imposes the budget $L$ on the total cost of the edges in the subgraph. Constraints \eqref{glp:visorcov} ensure that a vertex cannot be visited and covered at the same time and if it is covered, then it can be covered only by a single vertex. Constraints \eqref{glp:coverage} make sure that an unvisited vertex does not cover any other vertex and that a visited vertex $w$ covers at most $c_{w}$ vertices. Constraints \eqref{glp:liftednconn} impose connectivity; this requirement follows from the assumption that there is a single depot.  These constraints are analogous to those that 
appear, among other studies, in Ozbaygin et al.\ \cite{Ozbaygin2016}, Gollowitzer and Ljubi\'c \cite{Gollowitzer11}, and Fouilhoux et al.\ \cite{Fouilhoux12}. Constraints \eqref{tree:xy} prevent an edge from being selected if one of its endpoints is not visited. Additional constraints \eqref{glp:subgraph} can be imposed on the visited vertices and the edges of the subgraph. The remaining constraints are variable restrictions. 

Note that we do not constrain the $z$ variables to be binary. Indeed, observe that for $S\subseteq V\smallsetminus \{0\}$ and $v\in S$, if $\sum_{e\in \delta\left( S\right)}x_{e} - y_{v} =0\:$ then the corresponding constraint \eqref{glp:liftednconn} forces $z_{vw}$ to be zero for any $w\in N_{v} \cap S$. On the other hand, if $\sum_{e\in \delta\left( S\right)}x_{e} - y_{v} \geq 1\:$ then constraint \eqref{glp:liftednconn} is redundant. Consequently, when $x$ and $y$ are given, the remaining problem of finding the best $z$ is a transportation problem. Hence, there exists an optimal solution to our model with binary $z$.

This MIP model generalizes the model of \cite{Ozbaygin2016} by replacing its single-indexed coverage variables $z_v$  with the double-indexed $z_{vw}$ defined above and by allowing non-tour visiting subgraphs via the generic constraints \eqref{glp:subgraph}. We believe that this generalization is certainly worth investigating, as it opens the door to the study of more balanced coverage layers and of non-tour solution subgraphs, e.g., the trees, that pose new challenges in terms of the solutions' symmetry. Indeed, depending on the subgraph type and on the specific constraints \eqref{glp:subgraph} used to impose this type of requirement, this MIP can have symmetric optimal solutions, which only differ in the values of the $x$ variables. In particular, there might be multiple feasible subgraphs whose cost does not exceed the budget and that contain the same vertices. With a slight abuse of terminology, we refer to this as ``symmetry''. We identify symmetry-breaking inequalities for subgraphs that are trees and tours in Sections 4 and 5, respectively.
\subsection{Branch-and-cut framework}
In this subsection, we propose a branch-and-cut framework based on the separation of inequalities \eqref{glp:liftednconn}, analogously to the schemes proposed by Fischetti et al.\ \cite{Fischetti98} and Ozbaygin et al.\ \cite{Ozbaygin2016}. The framework can be described as follows:
\begin{enumerate}
\item[i.] We initialize the formulation with only a subset of the inequalities \eqref{glp:liftednconn}, namely for every $v\in V$ and with $S = N_{v}\cup\left\lbrace v\right\rbrace\smallsetminus\left\lbrace 0\right\rbrace$.
\item[ii.] At any given node of the branch-and-cut tree, with either fractional or integral solutions, we define the support graph $\Gamma$ that consists of all the vertices $v$ and the edges $e$ such that, respectively, $y_{v}>0$ and $x_{e}>0$. If $\Gamma$ is disconnected, then for every connected component $S$ that does not contain the depot and for every $v\in S$, we add the corresponding inequality \eqref{glp:liftednconn}. Let $S_{0}$ be the connected component containing the depot. We then also add the connectivity constraints corresponding to the set $S'=V\smallsetminus S_{0}$ and the vertices $v\in S'$.
\item[iii.] At the root node, if the solution is fractional and if the support graph $\Gamma$ is connected, we set the weight of each edge $e$ in $\Gamma$ as its solution value $x_{e}$ and we then compute a global min-cut. We choose the side $S$ of the min-cut that does not contain the depot. We then check, for every vertex $v\in S$, whether the corresponding inequality \eqref{glp:liftednconn} is violated and if so, we add it.
\end{enumerate}
\subsection{A Benders decomposition for the case with independent prizes}
For any vertex pair $v\in V\smallsetminus\left\lbrace 0\right\rbrace$ and $w\in N_{v}$, we say that the coverage prize $q_{vw}$ is \emph{independent} if it does not depend on the covering vertex $w$, i.e., if $q_{vw}=q_{v}$ 
and only depends on $v$.
For this special case, we propose a Benders decomposition whose master problem is obtained from model \eqref{glp:length_threshold}-\eqref{glp:end} by projecting out the $z$ variables.

We first add new continuous variables $\theta_{v}\geq 0$ for any $v\in V\smallsetminus\left\lbrace 0\right\rbrace$, and $\eta_{v}\geq 0$ for any $v\in V\smallsetminus\left\lbrace 0\right\rbrace$ such that $0\in N_{v}$ to the model \eqref{glp:length_threshold}-\eqref{glp:end}. These variables variables have the following interpretation: $\theta_{v}$ indicates whether vertex $v\in V\smallsetminus\left\lbrace 0\right\rbrace$ is covered by some non-depot vertex or not, as in Ozbaygin et al. \cite{Ozbaygin2016}, while $\eta_{v}$ indicates whether vertex $v$ is covered by the depot or not. We can then relate these variables to the former $z$ variables as follows:
\begin{align}
\theta_{v}&=\sum_{w\in N_{v}\smallsetminus\left\lbrace 0\right\rbrace} z_{vw} &\forall v\in V\smallsetminus\left\lbrace 0\right\rbrace\\
\eta_{v}&=z_{v0} &\forall v\in V\smallsetminus\left\lbrace 0\right\rbrace\; :\; 0\in N_{v}
\end{align}
These relations allow us to refrain from imposing integrality constraints on vectors $\theta$ and $\eta$. Then, we impose the coverage capacity on the depot by making use of the $\eta$ variables:
\begin{align}
\sum_{v\in V\smallsetminus\left\lbrace 0\right\rbrace\; :\; 0\in N_{v}}&\eta_{v}\leq c_{0}
\end{align}
and we update constraints \eqref{glp:visorcov} as follows:
\begin{align}
y_{v}+\theta_{v} & +\eta_{v}\leq 1 & \forall v\in V\smallsetminus\left\lbrace 0\right\rbrace\; :\; 0\in N_{v}\\
y_{v}+\theta_{v} & \leq 1 & \forall v\in V\smallsetminus\left\lbrace 0\right\rbrace\; :\; 0\in V\smallsetminus N_{v}
\end{align}
In the master problem of the branch-and-cut framework above, the connectivity constraints \eqref{glp:liftednconn} identified in step~i. are restated as follows:
\begin{align}
&\sum_{e\in\delta\left( N_{v}\cup\left\lbrace v\right\rbrace\smallsetminus\left\lbrace 0\right\rbrace\right)}x_{e}\geq  y_{v}+\theta_{v} &\forall  v\in V\smallsetminus\left\lbrace 0\right\rbrace
\end{align}
We project the variables $z$ out of the model, and for every optimal solution $\left( y^{*}, \theta^{*}\right)$ we get a Benders feasibility subproblem as follows:
\begin{align}
\label{substart}&\sum_{w\in N_{v}\smallsetminus\left\lbrace 0\right\rbrace}z_{vw}= \theta^{*}_{v}&\forall v\in V\smallsetminus\left\lbrace 0\right\rbrace\\
&\sum_{v\in V\smallsetminus\left\lbrace 0\right\rbrace\; :\; w\in N_{v}}z_{vw}\leq c_{w}y^{*}_{w}& \forall w\in V\smallsetminus\left\lbrace 0\right\rbrace\\
\label{subend}&z_{vw}\geq 0&\forall v\in V\smallsetminus\left\lbrace 0\right\rbrace,\;w\in N_{v}\smallsetminus\left\lbrace 0\right\rbrace
\end{align}
We then update the objective of the resulting master problem to:
\begin{align}
\sum_{v\in V\smallsetminus\left\lbrace 0\right\rbrace}\left( p_{v}y_{v} + q_{v}\theta_{v}\right) + \sum_{v\in V\smallsetminus\left\lbrace 0\right\rbrace\; :\; 0\in N_{v}} q_{v}\eta_{v}
\end{align}
Finally, we solve the master problem in the branch-and-cut framework above, where at each node we also solve the Benders feasibility subproblem \eqref{substart}-\eqref{subend}. If this subproblem is infeasible at a node, we retrieve an improving direction $\left( u^{\theta}, u^{y}\right)$ for its dual and add a feasibility cut of the form:
\begin{align}
\label{glp:bendersfeas}
\sum_{v\in V\smallsetminus\left\lbrace 0\right\rbrace} \left( u^{\theta}_{v}\cdot\theta_{v}+u^{y}_{v}\cdot y_{v}\right)\geq 0
\end{align}

\section{The tree BPCCSP}
In this section, we discuss how to adapt the generic model \eqref{glp:length_threshold}-\eqref{glp:end} to the tree subgraph case and we introduce ad-hoc symmetry-breaking inequalities. Two optimal BPCCSP solution trees on the same instance are shown in Figure \ref{figure_Belgium_tree}, when all the coverage capacities $c_{v}$ are set to zero or two, respectively. This small example shows that the subgraph can change significantly when coverage is taken into account.

\begin{figure}[t]
\centering
\begin{adjustbox}{valign=t,minipage={.45\textwidth}}
\centering
\begin{tikzpicture}[scale=0.45]
\tikzstyle{every node}=[draw,circle,fill=white,minimum size=4pt,
                            inner sep=0pt]

\draw (4.22,	2.67) node (1) [overlay] {};
\draw (5.24,	2.91) node (2) [overlay] {};
\draw (8.00,	1.43) node (3) [overlay] {};
\draw (4.36,	5.00) node (4) [overlay] {};
\draw (7.16,	3.16) node (5) [overlay] {};
\draw (0.00,	4.97) node (6) [overlay] {};
\draw (0.98,	2.53) node (7) [overlay] {};
\draw (2.37,	4.06) node (8) [overlay] {};
\draw (0.88,	4.93) node (9) [overlay] {};
\draw (5.70,	0.30) node (10) [overlay] {};
\draw (4.49,	0.00) node (11) [overlay] {};
\draw (3.05,	0.27) node (12) [overlay] {};

\draw (9) -- (6) [very thick, color=red];
\draw (9) -- (7) [very thick, color=red];
\draw (7) -- (12) [very thick, color=red];
\draw (12) -- (1) [very thick, color=red];
\end{tikzpicture}
\begin{center}BPCCSP with $c_{v}=0\;\;\;\forall v\in V$\end{center}
\end{adjustbox}\hfill
\begin{adjustbox}{valign=t,minipage={.45\textwidth}}
\centering
\begin{tikzpicture}[scale=0.45]
\tikzstyle{every node}=[draw,circle,fill=white,minimum size=4pt,
                            inner sep=0pt]

\draw (4.22,	2.67) node (1) [overlay] {};
\draw (5.24,	2.91) node (2) [overlay] {};
\draw (8.00,	1.43) node (3) [overlay] {};
\draw (4.36,	5.00) node (4) [overlay] {};
\draw (7.16,	3.16) node (5) [overlay] {};
\draw (0.00,	4.97) node (6) [overlay] {};
\draw (0.98,	2.53) node (7) [overlay] {};
\draw (2.37,	4.06) node (8) [overlay] {};
\draw (0.88,	4.93) node (9) [overlay] {};
\draw (5.70,	0.30) node (10) [overlay] {};
\draw (4.49,	0.00) node (11) [overlay] {};
\draw (3.05,	0.27) node (12) [overlay] {};

\draw (1) -- (11) [very thick, color=red];
\draw (11) -- (10) [very thick, color=red];
\draw (11) -- (12) [very thick, color=red];
\draw (12) -- (7) [very thick, color=red];

\draw [dashed, very thin, color=teal] (7) -- (6);
\draw [dashed, very thin, color=teal] (7) -- (9);
\draw [dashed, very thin, color=teal] (1) -- (4);

\draw [dashed, very thin, color=teal] (1) -- (8);
\draw [dashed, very thin, color=teal] (10) -- (2);
\draw [dashed, very thin, color=teal] (10) -- (3);

\end{tikzpicture}
\begin{center}BPCCSP with $c_{v}=2\;\;\;\forall v\in V$\end{center}
\end{adjustbox}
\caption{Two solutions for the tree BPCCSP on the same instance are shown, with coverage capacities set to zero or two, respectively. The instance is made by $12$ Belgian cities with actual Euclidean distances as edge costs and a $200$ km budget. Prizes were arbitrarily generated.}
\label{figure_Belgium_tree}
\end{figure}
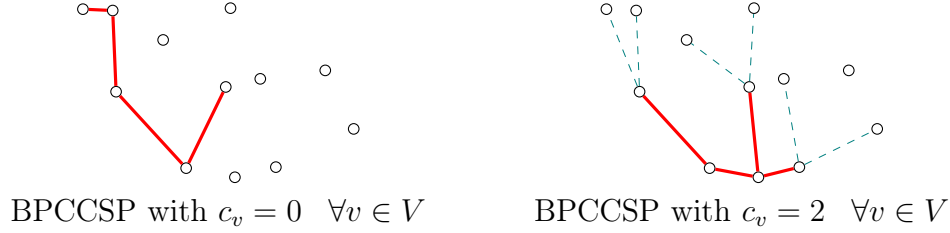

\subsection{MIP model}
Preliminary computational experiments show that introducing directed arc variables is more effective than formulating the problem with only edge variables, in line with the hierarchy of formulations for the spanning tree problem proposed by Magnanti and Wolsey \cite{trees_Wolsey}. Accordingly, we add the following continuous variables to the model \eqref{glp:length_threshold}-\eqref{glp:end}, and we relate them to the former edge variables $x$:
\begin{align}
&u_{vw}\geq 0 & \forall v\in V,\; w\in V\smallsetminus\left\lbrace v\right\rbrace \: :\:\left\lbrace v,w\right\rbrace\in E\\
&x_{e}=u_{vw}+u_{wv} & \forall e=\left\lbrace v,w\right\rbrace\in E
\end{align}
Instead of inequalities \eqref{glp:liftednconn}, we derive the analogous connectivity inequalities with the $u$ variables:
\begin{align}
\sum_{v\in V\smallsetminus S}\;\sum_{w\in S: \left\lbrace v,w\right\rbrace\in E}\;u_{vw}\;\geq y_{k} + \sum_{w\in S\cap N_{k}}z_{kw} && \forall S\subseteq V\smallsetminus\left\lbrace 0\right\rbrace,\; k\in S
\end{align}
Moreover, we constrain the number of solution edges and vertices as follows:
\begin{align}
\label{n_edges_vertices}
\sum_{e\in E}x_{e}\; =\; \sum_{v\in V\smallsetminus\left\lbrace 0\right\rbrace}y_{v}
\end{align}
Finally, as proposed by Labb\'e et al.\ \cite{Labbe04}, we strengthen our model by lifting constraints \eqref{tree:xy}, when applicable, with the $z$ variables:
\begin{align}
\label{tree:xylifted}
x_{\left\lbrace v,w\right\rbrace}+z_{wv}\leq y_{v} && \forall w\in V\smallsetminus\left\lbrace 0\right\rbrace ,\; v\in N_{w}\smallsetminus\left\lbrace 0\right\rbrace \: :\:\left\lbrace v,w\right\rbrace\in E
\end{align}
Indeed, $x_{\left\lbrace v,w\right\rbrace}$ and $z_{wv}$ cannot be $1$ at the same time for such vertices $v$ and $w$, and they are both constrained to take the value $0$ if $y_{v}=0$.
\subsection{Symmetry-breaking inequalities}
The budgeted prize-col\-lect\-ing character of this problem might give rise to symmetries. In particular, there might be a number of optimal solutions that only differ in their values of the $x$ variables. Indeed, the objective depends on the choice of vertices visited by the tree but is indifferent between alternative edge choices, as long as the tree spans the visited vertices without exceeding the budget $L$.
\begin{prop}
\label{prop_triangles}
Let $T\subseteq E$ be an optimal tree and $v,w,k\in V$ be distinct visited vertices by $T$. Suppose that the three edges $\left\lbrace v,w\right\rbrace$, $\left\lbrace v,k\right\rbrace$, $\left\lbrace w,k\right\rbrace$ are in $E$ and that $\left\lbrace v,w\right\rbrace$ is the most expensive one among them. Then, either $\left\lbrace v,w\right\rbrace\notin T$ or there exists another optimal tree $T'$ such that $\left\lbrace v,w\right\rbrace\notin T'$ and
\begin{equation}
\label{eqn_prop_tria}
T\smallsetminus \left\lbrace\left\lbrace v,w\right\rbrace, \left\lbrace v,k\right\rbrace, \left\lbrace w,k\right\rbrace\right\rbrace = T'\smallsetminus \left\lbrace\left\lbrace v,w\right\rbrace, \left\lbrace v,k\right\rbrace, \left\lbrace w,k\right\rbrace\right\rbrace.
\end{equation}
\end{prop}
\begin{proof}
Suppose that $T$ contains $\left\lbrace v,w\right\rbrace$. $T$ is connected, thus there exists a path $P\subseteq T\smallsetminus\left\lbrace\left\lbrace v,w\right\rbrace\right\rbrace$ from $k$ to either $v$ or $w$. Suppose without loss of generality that $P$ connects $k$ to $v$ and that, consequently, $T$ does not contain $\left\lbrace w,k\right\rbrace$. Then, $T'=T\smallsetminus\left\lbrace\left\lbrace v,w\right\rbrace\right\rbrace\cup\left\lbrace\left\lbrace w,k\right\rbrace\right\rbrace$ is the sought optimal tree. First, notice that condition \eqref{eqn_prop_tria} holds by construction. Because of $\ell_{vw}\geq\ell_{wk}$, $T'$ is not more expensive than $T$. Moreover, $T'$ visits the same vertices as $T$, and hence collects the same prize. Since $T$ is connected, $T'$ is connected if and only if $k$ and $v$ are endpoints of a path in $T'$. But $P$ is contained in $T'$, so $T'$ is connected. Finally, $T'$ contains the same number of edges as $T$, hence it satisfies constraint \eqref{n_edges_vertices} and therefore cannot contain cycles.
\end{proof}
%
Notice that condition \eqref{eqn_prop_tria} allows us to replace the most expensive edge of any visited $3$-clique independently, and thus we can translate the latter result into symmetry-breaking inequalities involving the $x$ and the $y$ variables, for any three distinct vertices $v,w,k\in V$ such that their connecting edges $\left\lbrace v,w\right\rbrace , \left\lbrace v,k \right\rbrace , \left\lbrace w,k \right\rbrace$ belong to $E$:
\begin{align}
\label{triangle_ineq}
x_{vw}+y_{k}\leq 1 && \text{if }\ell_{vw}> \ell_{vk}\text{ and } \ell_{vw}>\ell_{wk}
\end{align}
If needed, we can break ties between the edge costs based on  endpoint indices. Figure \ref{figure_symmetry} shows a case where symmetry occurs and which can be broken by inequalities \eqref{triangle_ineq}.

\begin{figure}[t]
\centering
\begin{adjustbox}{valign=t,minipage={.45\textwidth}}
\centering
\begin{tikzpicture}[scale=0.45]
\tikzstyle{every node}=[draw,circle,fill=white,minimum size=4pt,
                            inner sep=0pt]

\draw (4.22,	2.67) node (1) [overlay, fill=black] {};
\draw (5.24,	2.91) node (2) [overlay] {};
\draw (8.00,	1.43) node (3) [overlay] {};
\draw (4.36,	5.00) node (4) [overlay] {};
\draw (7.16,	3.16) node (5) [overlay] {};
\draw (0.00,	4.97) node (6) [overlay] {};
\draw (0.98,	2.53) node (7) [overlay] {};
\draw (2.37,	4.06) node (8) [overlay] {};
\draw (0.88,	4.93) node (9) [overlay] {};
\draw (5.70,	0.30) node (10) [overlay] {};
\draw (4.49,	0.00) node (11) [overlay, fill=black] {};
\draw (3.05,	0.27) node (12) [overlay, fill=black] {};

\draw (1) -- (12) [very thick, color=red];
\draw (11) -- (10) [very thick, color=red];
\draw (11) -- (12) [very thick, color=red];
\draw (12) -- (7) [very thick, color=red];

\draw [dashed, very thin, color=teal] (7) -- (6);
\draw [dashed, very thin, color=teal] (7) -- (9);
\draw [dashed, very thin, color=teal] (1) -- (4);

\draw [dashed, very thin, color=teal] (1) -- (8);
\draw [dashed, very thin, color=teal] (10) -- (2);
\draw [dashed, very thin, color=teal] (10) -- (3);
\end{tikzpicture}
\end{adjustbox}\hfill
\begin{adjustbox}{valign=t,minipage={.45\textwidth}}
\centering
\begin{tikzpicture}[scale=0.45]
\tikzstyle{every node}=[draw,circle,fill=white,minimum size=4pt,
                            inner sep=0pt]

\draw (4.22,	2.67) node (1) [overlay, fill=black] {};
\draw (5.24,	2.91) node (2) [overlay] {};
\draw (8.00,	1.43) node (3) [overlay] {};
\draw (4.36,	5.00) node (4) [overlay] {};
\draw (7.16,	3.16) node (5) [overlay] {};
\draw (0.00,	4.97) node (6) [overlay] {};
\draw (0.98,	2.53) node (7) [overlay] {};
\draw (2.37,	4.06) node (8) [overlay] {};
\draw (0.88,	4.93) node (9) [overlay] {};
\draw (5.70,	0.30) node (10) [overlay, fill=black] {};
\draw (4.49,	0.00) node (11) [overlay, fill=black] {};
\draw (3.05,	0.27) node (12) [overlay] {};

\draw (1) -- (10) [very thick, color=red];
\draw (11) -- (10) [very thick, color=red];
\draw (11) -- (12) [very thick, color=red];
\draw (12) -- (7) [very thick, color=red];

\draw [dashed, very thin, color=teal] (7) -- (6);
\draw [dashed, very thin, color=teal] (7) -- (9);
\draw [dashed, very thin, color=teal] (1) -- (4);

\draw [dashed, very thin, color=teal] (1) -- (8);
\draw [dashed, very thin, color=teal] (10) -- (2);
\draw [dashed, very thin, color=teal] (10) -- (3);

\end{tikzpicture}
\end{adjustbox}
\caption{Two symmetric tree subgraph BPCCSP solutions on the same instance as in Figure \ref{figure_Belgium_tree} are shown, with full black vertices corresponding to those inequalities \eqref{triangle_ineq} that can break the symmetry.}
\label{figure_symmetry}
\end{figure}
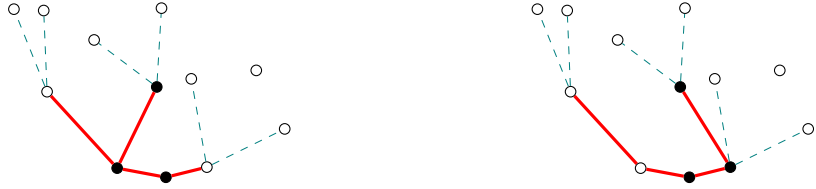

Other symmetries may occur for $n$-tuples of vertices for $n>3$, even if we add all the constraints \eqref{triangle_ineq} to the model. Already for $n = 4$, however, adding to the model their corresponding symmetry-breaking inequalities does not 
compensate the effect of the larger size of the model.

The impact of inequalities \eqref{triangle_ineq} on the solution process is discussed in Section \ref{comp_sym}, based on the computational results reported in Table \ref{table_triangles}.

\section{The tour BPCCSP}

\begin{figure}[t]
\centering
\begin{adjustbox}{valign=t,minipage={.45\textwidth}}
\hspace{33pt}
\begin{tikzpicture}[scale=0.45]
\tikzstyle{every node}=[draw,circle,fill=white,minimum size=4pt,
                            inner sep=0pt]

\draw (4.22,	2.67) node (1) [overlay] {};
\draw (5.24,	2.91) node (2) [overlay] {};
\draw (8.00,	1.43) node (3) [overlay] {};
\draw (4.36,	5.00) node (4) [overlay] {};
\draw (7.16,	3.16) node (5) [overlay] {};
\draw (0.00,	4.97) node (6) [overlay] {};
\draw (0.98,	2.53) node (7) [overlay] {};
\draw (2.37,	4.06) node (8) [overlay] {};
\draw (0.88,	4.93) node (9) [overlay] {};
\draw (5.70,	0.30) node (10) [overlay] {};
\draw (4.49,	0.00) node (11) [overlay] {};
\draw (3.05,	0.27) node (12) [overlay] {};

\draw (1) -- (2) [very thick, color=red];
\draw (2) -- (4) [very thick, color=red];
\draw (4) -- (8) [very thick, color=red];
\draw (8) -- (1) [very thick, color=red];

\draw [dashed, very thin, color=teal] (8) -- (9);
\draw [dashed, very thin, color=teal] (8) -- (6);
\draw [dashed, very thin, color=teal] (8) -- (7);

\draw [dashed, very thin, color=teal] (1) -- (10);
\draw [dashed, very thin, color=teal] (1) -- (11);
\draw [dashed, very thin, color=teal] (1) -- (12);

\draw [dashed, very thin, color=teal] (2) -- (5);
\draw [dashed, very thin, color=teal] (2) -- (3);

\end{tikzpicture}
\begin{center}TCMCSP\end{center}
\end{adjustbox}\hfill
\begin{adjustbox}{valign=t,minipage={.45\textwidth}}
\hspace{13pt}
\begin{tikzpicture}[scale=0.45]
\tikzstyle{every node}=[draw,circle,fill=white,minimum size=4pt,
                            inner sep=0pt]

\draw (4.22,	2.67) node (1) [overlay] {};
\draw (5.24,	2.91) node (2) [overlay] {};
\draw (8.00,	1.43) node (3) [overlay] {};
\draw (4.36,	5.00) node (4) [overlay] {};
\draw (7.16,	3.16) node (5) [overlay] {};
\draw (0.00,	4.97) node (6) [overlay] {};
\draw (0.98,	2.53) node (7) [overlay] {};
\draw (2.37,	4.06) node (8) [overlay] {};
\draw (0.88,	4.93) node (9) [overlay] {};
\draw (5.70,	0.30) node (10) [overlay] {};
\draw (4.49,	0.00) node (11) [overlay] {};
\draw (3.05,	0.27) node (12) [overlay] {};

\draw [very thick, color=red] (1) -- (12);
\draw [very thick, color=red] (12) -- (11);
\draw [very thick, color=red] (11) -- (10);
\draw [very thick, color=red] (10) -- (1);

\draw [dashed, very thin, color=teal] (12) -- (7);
\draw [dashed, very thin, color=teal] (11) -- (2);
\draw [dashed, very thin, color=teal] (10) -- (5);
\draw [dashed, very thin, color=teal] (10) -- (3);
\draw [dashed, very thin, color=teal] (1) -- (4);
\draw [dashed, very thin, color=teal] (1) -- (8);

\end{tikzpicture}
\begin{center}BPCCSP with $c_{v}=2\;\;\;\forall v\in V$\end{center}
\end{adjustbox}
\caption{Optimal solutions of the TCMCSP (left) and the tour subgraph BPCCSP with $c_{v}=2$ for every vertex $v$ (right) are shown, on the same instance as in Figure \ref{figure_Belgium_tree}.}
\label{figure_Belgium}
\end{figure}
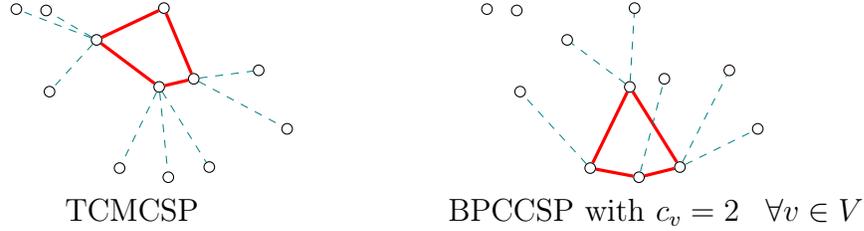
The tour subgraph BPCCSP extends the definition of the TCMCSP by considering covering capacities. Figure \ref{figure_Belgium} compares their optimal solutions for an instance with $12$ Belgian cities.
A formulation for the tour subgraph BPCCSP can be obtained from the generic model \eqref{glp:length_threshold}-\eqref{glp:end} by replacing the generic constraints \eqref{glp:subgraph} with $2$-degree constraints:
\begin{align}
\label{2degreetour}
\sum_{w\in V\smallsetminus\left\lbrace v\right\rbrace\; :\;\left\lbrace v,w\right\rbrace\in E}x_{\left\lbrace v,w\right\rbrace}\; =\; \begin{cases} 2y_{v} & \;\; v\neq 0 \\ 2 & \;\; v=0 \end{cases}&& \forall v\in V
\end{align}
Since any tour is $2$-edge connected, i.e., it cannot be disconnected by removing one edge, we can also replace constraints \eqref{glp:liftednconn} with $2$-edge connectivity constraints, which imply the former, as already proposed by Ozbaygin et al.\ \cite{Ozbaygin2016} for the TCMCSP:
\begin{align}
\sum_{e\in\delta\left( S\right)}x_{e}\geq 2\left( y_{v}+\sum_{w\in N_{v}\cap S}z_{vw} \right)&& \forall S\subseteq V\smallsetminus\left\lbrace 0\right\rbrace,\; v\in S
\end{align}
The latter constraints, given the $2$-degree constraints \eqref{2degreetour}, can be separated with the same separation scheme as \eqref{glp:liftednconn}.


Furthermore, symmetry-breaking inequalities can be added to the model for the tour case BPCCSP, to limit the number of optimal solutions that have in common the same visited vertices and coverage layer, but not the selected edges.
\begin{prop}
Let $v_{1},v_{2},v_{3},v_{4}$ be four distinct visited vertices in $V$, with $e_{12}$, $e_{13}$, $e_{14}$, $e_{23}$, $e_{24}$, $e_{34}\in E$ the edges connecting them. Suppose that $\ell_{e_{12}}+\ell_{e_{34}}>\ell_{e_{13}}+\ell_{e_{24}}$ and $\ell_{e_{12}}+\ell_{e_{34}}>\ell_{e_{23}}+\ell_{e_{14}}$. For any optimal tour $T$, there exists an optimal tour $T'$ that does not contain both the edges $e_{12}$ and $e_{34}$, and such that $T\smallsetminus\left\lbrace e_{12}, e_{13}, e_{14}, e_{23}, e_{24}, e_{34}\right\rbrace = T'\smallsetminus\left\lbrace e_{12}, e_{13}, e_{14}, e_{23}, e_{24}, e_{34}\right\rbrace$.
\end{prop}
\begin{proof}
If one of the edges $e_{12}$ and $e_{34}$ is not in $T$, there is nothing to prove. Suppose now that both $e_{12}$ and $e_{34}$ are in $T$. Then, consider the other two couples of opposite edges in the $4$-clique determined by $v_{1},v_{2},v_{3},v_{4}$, namely $\left\lbrace e_{13}, e_{24}\right\rbrace$ and $\left\lbrace e_{14}, e_{23}\right\rbrace$. By removing $e_{12}$ and $e_{34}$ from $T$ and reconnecting it with the other two couples of opposite edges, we obtain $T'=T\smallsetminus\left\lbrace e_{12},e_{34}\right\rbrace\cup\left\lbrace e_{13}, e_{24}\right\rbrace$ and $T''=T\smallsetminus\left\lbrace e_{12},e_{34}\right\rbrace\cup\left\lbrace e_{14}, e_{23}\right\rbrace$, respectively. At least one of the subgraphs $T'$ and $T''$ is $2$-edge connected and does not contain subtours, say $T'$ without loss of generality. Moreover, the total cost of $T'$ is not greater than $T$, as $\ell_{e_{12}}+\ell_{e_{34}}>\ell_{e_{13}}+\ell_{e_{24}}$. Finally, $T'$ visits the same vertices as $T$, hence it is optimal.
\end{proof}
This proposition allows us to add the following symmetry-breaking inequalities to our model, for every such non-incident edges $e_{12}$ and $e_{34}$:
\begin{flalign}
\label{tour_quad}
x_{e_{12}}+x_{e_{34}}\leq 1
\end{flalign}
If needed, we can break ties on the costs of different couples of opposite edges by using the involved vertices' indices. Notice that these inequalities are the constraints counterpart of a $2$-OPT move. Further pursuing this idea, we can generalize the result in a $k$-OPT fashion for $k\geq 3$. Even though the number of the resulting constraints is polynomial in $k$, their benefits are more than offset by the higher runtimes due to increased model size, already for $k=3$.

The impact of inequalities \eqref{tour_quad} on the computational times is discussed in Section \ref{comp_sym}.

\section{Computational results}
We test our branch-and-cut framework and the Benders decomposition for the BPCCSPs with tree and with tour subgraphs. We implement our branch-and-cut framework using Gurobi 8.1, and we run our computational experiments on an Intel i7-8850H CPU at $2.60\:\mathrm{GHz}$, 12 cores and $31.2$ GB RAM machine. For every instance, we set the time limit to one hour. With this experimental setup, we are able to solve instances with up to $200$ vertices. Subroutines such as counting the number of connected components of a graph or the Stoer-Wagner max-flow algorithm have been called from the Boost library for c++ \cite{boost}.
\subsection{Test instances}
We test our algorithms on six groups of instances. Every group is generated from a single \emph{base instance}, i.e., a publicly available instance of the vehicle routing problem (VRP, four instances) or from the TSPLIB (two instances). These base instances are Euclidean and symmetric, which are reasonable properties for practical real-world applications. The details of  the base instances can be found in Table \ref{tab:instances}. We always use the edge length as edge cost.

\begin{table}[t]
\centering
\small
\begin{tabular}{ccccccc}
Name & $\left| V\right|$ & TSP & AVG & MST & Type & Source\\\cmidrule[.1em]{1-7}
p4 & $151$ & $707.86$ & $33.47$ & $635.28$ & VRP & \cite{instances_Neumann}\\\hline
p5 & $200$ & $776.44$ & $32.91$ & $691.97$ & VRP & \cite{instances_Neumann}\\\hline
X-n162-k11 & $162$ & $9174.95$ & $491.56$ & $7903.35$ & CVRP & \cite{instances_Uchoa}\\\hline
X-n195-k51 & $195$ & $10221.68$ & $496.46$ & $8882.14$ & CVRP & \cite{instances_Uchoa}\\\hline
ch150 & $150$ & $6530.90$ & $359.31$ & $5880.96$ & TSP & TSPLIB, \cite{Rein91}\\\hline
kroa200 & $200$ & $29369.41$ & $1701.17$ & $25932.60$ & TSP & TSPLIB, \cite{Rein91}
\end{tabular}
\caption{The six base instances used to generate each group of test instances. For each base instance, we show its size, its TSP objective, its average edge cost, its minimum spanning tree (MST) objective, the problem type, and its source.}
\label{tab:instances}
\end{table}

For every VRP instance, we use the vertex demands as their prize, as also proposed by Ozbaygin et al.\ \cite{Ozbaygin2016}. For the TSPLIB instances, we choose the first vertex as depot and generate the prizes pseudo-randomly according to the formula $p_{v}=1+\left( 7141 v+ 73\right) \mathrm{mod} \left( 100\right)$ for every vertex $v\in V$, as already proposed for the OP by Fischetti et al.\ \cite{Fischetti98}. Finally, for each of the six base instances, we compute the TSP and minimum spanning tree (MST) optimal solutions, and the average distance between two vertices (AVG). In every instance group we combine different parameter values:
\begin{enumerate}
\item[i.] the budget $L$ takes the values of $25\%$, $50\%$ and $75\%$ of the base instance's TSP (for tour subgraphs) or MST (tree subgraphs) optimal objective value;
\item[ii.] the neighbourhood of a vertex $v$ is chosen as
\begin{equation*}
N_{v}=\left\lbrace w\in V\smallsetminus \left\lbrace v\right\rbrace\text{ such that }\ell_{vw}\leq r_{v}\right\rbrace
\end{equation*}
for some given \emph{covering radius} $r_{v}\geq 0$ of $v$. For every $v\in V$, the radius $r_{v}$ takes  the values of $0.5$, $1$ and $2$ times the value of AVG;
\item[iii.] the coverage capacity $c_{v}$ is chosen as $0.5\%$, $1\%$, $2\%$ and $5\%$ of the base instance's size, for all $v\in V$, rounded to the closest integer.
\end{enumerate}
The covering prizes are \emph{independent}, i.e., each $q_{vw}$ is set to either $50\%$ or $75\%$ of the prize $p_{v}$ of the covered vertex. For both the tour and the tree subgraph case then, each of the six groups contains $72$ different instances, which are solved both by the branch-and-cut and by the Benders procedure. Table \ref{table_param} summarizes all the parameter values that were tested.

\begin{table}[t]
\centering
\resizebox{\textwidth}{!}{
\begin{tabular}{cccccccccccccccc}
& \multicolumn{3}{c}{$L$} && \multicolumn{3}{c}{$r_{v}$} && \multicolumn{4}{c}{$c_{v}$} && \multicolumn{2}{c}{$\nicefrac{q_{vw}}{p_{v}}$} \\\cmidrule[.1em]{2-4}\cmidrule[.1em]{6-8}\cmidrule[.1em]{10-13}\cmidrule[.1em]{15-16}
p4 & \begin{tabular}{@{}c@{}}  176.97\\  158.82\end{tabular} & \begin{tabular}{@{}c@{}}  353.93\\  317.64\end{tabular} & \begin{tabular}{@{}c@{}}  530.90\\  476.50\end{tabular} && 16.74 & 33.47 & 66.94 && 1 & 2 & 3 & 8 && 0.5 & 0.75\\\hline
p5 & \begin{tabular}{@{}c@{}}  194.11\\  173.00\end{tabular} & \begin{tabular}{@{}c@{}}  388.22\\  345.99\end{tabular} & \begin{tabular}{@{}c@{}}  582.33\\  518.98\end{tabular} && 16.46 & 32.91 & 65.82 && 1 & 2 & 4 & 10 && 0.5 & 0.75\\\hline
X-n162-k11 & \begin{tabular}{@{}c@{}}  2293.74\\  1975.84\end{tabular} & \begin{tabular}{@{}c@{}}  4587.48\\  3951.68\end{tabular} & \begin{tabular}{@{}c@{}}  6881.21\\  5927.51\end{tabular} && 245.78 & 491.56 & 983.12 && 1 & 2 & 3 & 8 && 0.5 & 0.75\\\hline
X-n195-k51 & \begin{tabular}{@{}c@{}}  2555.42\\  2220.53\end{tabular} & \begin{tabular}{@{}c@{}}  5110.84\\  4441.07\end{tabular} & \begin{tabular}{@{}c@{}}  7666.26\\  6661.61\end{tabular} && 248.23 & 496.46 & 992.92 && 1 & 2 & 4 & 10 && 0.5 & 0.75\\\hline
ch150 & \begin{tabular}{@{}c@{}}  1632.73\\  1470.24\end{tabular} & \begin{tabular}{@{}c@{}}  3265.45\\  2940.48\end{tabular} & \begin{tabular}{@{}c@{}}  4898.18\\  4410.72\end{tabular} && 179.66 & 359.31 & 718.62 && 1 & 2 & 3 & 8 && 0.5 & 0.75\\\hline
kroa200 & \begin{tabular}{@{}c@{}}  7342.35\\  6483.15\end{tabular} & \begin{tabular}{@{}c@{}} 14684.70\\  12966.30\end{tabular} & \begin{tabular}{@{}c@{}}  22027.10\\  19449.50\end{tabular} && 850.59 & 1701.17 & 3402.34 && 1 & 2 & 4 & 10 && 0.5 & 0.75\\
\end{tabular}
}
\caption{Overview of the tested parameter values. In the columns corresponding to $L$, the first line refers to the tour subgraph case, and the second line to the tree subgraph case.}
\label{table_param}
\end{table}

\subsection{Comparison between algorithms}
In this subsection and in the remainder of the paper, we report the computational results of both the branch-and-cut and the Benders decomposition while they incorporate either symmetry-breaking inequalities \eqref{triangle_ineq} or \eqref{tour_quad} (except when otherwise specified).

Out of the $864$ instances tested in total,  $782$ can be solved to optimality with the branch-and-cut, and $732$ with the Benders decomposition. We cluster the instances by group and budget $L$, and summarize the average computational times in Table \ref{computational_results}.

\begin{table}[!p]
\centering
\ssmall
\begin{tabular}{cccccccccc}
&\multicolumn{4}{c}{tour subgraph}&&\multicolumn{4}{c}{tree subgraph}\\\cmidrule[.1em]{2-5}\cmidrule[.1em]{7-10}
group& $L$  & solved  &  CPU & nodes && $L$ & solved & CPU & nodes\\\cmidrule[.1em]{1-10}
p4       & 176.97 & \begin{tabular}{@{}c@{}}24  \\ 24 \end{tabular}  & \begin{tabular}{@{}c@{}}131.97 \\ \textbf{127.27}\end{tabular} & \begin{tabular}{@{}c@{}}\textbf{1262.38} \\ 1557.29\end{tabular} && 158.82 & \begin{tabular}{@{}c@{}}24  \\ 24 \end{tabular}  & \begin{tabular}{@{}c@{}}\textbf{308.57} \\ 309.59\end{tabular}  & \begin{tabular}{@{}c@{}}\textbf{688.29} \\ 1200.08\end{tabular}  \\\cmidrule[.05em]{2-10}
         & 353.93 & \begin{tabular}{@{}c@{}}24  \\ 24 \end{tabular}  & \begin{tabular}{@{}c@{}}\textbf{4.61} \\ 9.83\end{tabular}   & \begin{tabular}{@{}c@{}}\textbf{65.50} \\ 97.25\end{tabular}    && 317.64 & \begin{tabular}{@{}c@{}}24  \\ 24 \end{tabular}  & \begin{tabular}{@{}c@{}}\textbf{200.20} \\ 209.40\end{tabular}   & \begin{tabular}{@{}c@{}}\textbf{614.67} \\ 639.63\end{tabular}  \\\cmidrule[.05em]{2-10}
         & 530.90  & \begin{tabular}{@{}c@{}}24  \\ 24 \end{tabular}  & \begin{tabular}{@{}c@{}}\textbf{16.66} \\ 55.31\end{tabular}  & \begin{tabular}{@{}c@{}}911.17 \\ \textbf{835.88}\end{tabular}  && 476.50  & \begin{tabular}{@{}c@{}}24  \\ 24 \end{tabular}  & \begin{tabular}{@{}c@{}}\textbf{179.98} \\ 272.87\end{tabular}  & \begin{tabular}{@{}c@{}}\textbf{1586.17} \\ 1633.38\end{tabular} \\\cmidrule[.05em]{1-10}
p5       & 194.11 & \begin{tabular}{@{}c@{}}24  \\ 24 \end{tabular}  & \begin{tabular}{@{}c@{}}102.17 \\ \textbf{98.58}\end{tabular} & \begin{tabular}{@{}c@{}}\textbf{648.50} \\ 679.50\end{tabular}   && 173.00    & \begin{tabular}{@{}c@{}}23  \\ \textbf{24} \end{tabular}  & \begin{tabular}{@{}c@{}}1233.35 \\ \textbf{1134.45}\end{tabular} & \begin{tabular}{@{}c@{}}\textbf{501.61} \\ 983.96\end{tabular}  \\\cmidrule[.05em]{2-10}
         & 388.22 & \begin{tabular}{@{}c@{}}\textbf{24}  \\ 20 \end{tabular}  & \begin{tabular}{@{}c@{}}\textbf{319.29} \\ 501.94\end{tabular} & \begin{tabular}{@{}c@{}}2493.38 \\ \textbf{2085.90}\end{tabular} && 345.99 & \begin{tabular}{@{}c@{}}\textbf{23}  \\ 19\end{tabular}  & \begin{tabular}{@{}c@{}}\textbf{1151.97} \\ 1768.57\end{tabular} & \begin{tabular}{@{}c@{}}1240.48 \\ \textbf{1212.26}\end{tabular} \\\cmidrule[.05em]{2-10}
         & 582.33 & \begin{tabular}{@{}c@{}}\textbf{22}  \\ 8\end{tabular}  & \begin{tabular}{@{}c@{}}\textbf{988.60} \\ 1949.67\end{tabular}  & \begin{tabular}{@{}c@{}}4845.00 \\ \textbf{4254.75}\end{tabular}    && 518.98 & \begin{tabular}{@{}c@{}}24  \\ 24\end{tabular}  & \begin{tabular}{@{}c@{}}\textbf{880.93} \\ 1243.81\end{tabular}  & \begin{tabular}{@{}c@{}}938.75 \\ \textbf{894.52}\end{tabular}  \\\cmidrule[.05em]{1-10}
X-n162-k11       & 2293.74 & \begin{tabular}{@{}c@{}}22  \\ \textbf{23} \end{tabular}  & \begin{tabular}{@{}c@{}}\textbf{349.77} \\ 527.70\end{tabular} & \begin{tabular}{@{}c@{}}\textbf{2462.50} \\ 6747.83\end{tabular} && 1975.84 & \begin{tabular}{@{}c@{}}20  \\ 20 \end{tabular}  & \begin{tabular}{@{}c@{}}1262.75 \\ \textbf{1011.27}\end{tabular}  & \begin{tabular}{@{}c@{}}959.80 \\ \textbf{898.70}\end{tabular}  \\\cmidrule[.05em]{2-10}
         & 4587.48 & \begin{tabular}{@{}c@{}}24  \\ 24 \end{tabular}  & \begin{tabular}{@{}c@{}}328.85 \\ \textbf{274.60}\end{tabular}   & \begin{tabular}{@{}c@{}}761.67 \\ \textbf{708.21}\end{tabular}    && 3951.68 & \begin{tabular}{@{}c@{}}24  \\ 24 \end{tabular}  & \begin{tabular}{@{}c@{}}\textbf{584.23} \\ 828.46\end{tabular}   & \begin{tabular}{@{}c@{}}\textbf{599.88} \\ 670.50\end{tabular}  \\\cmidrule[.05em]{2-10}
         & 6881.21  & \begin{tabular}{@{}c@{}}24  \\ 24 \end{tabular}  & \begin{tabular}{@{}c@{}}\textbf{307.58} \\ 346.86\end{tabular}  & \begin{tabular}{@{}c@{}}2292.38 \\ \textbf{1938.71}\end{tabular}  && 5927.51  & \begin{tabular}{@{}c@{}}24  \\ 24 \end{tabular}  & \begin{tabular}{@{}c@{}}\textbf{342.15} \\ 495.99\end{tabular}  & \begin{tabular}{@{}c@{}}\textbf{617.71} \\ 705.92\end{tabular} \\\cmidrule[.05em]{1-10}
X-n195-k51       & 2555.42 & \begin{tabular}{@{}c@{}}\textbf{24}  \\ 22 \end{tabular}  & \begin{tabular}{@{}c@{}}\textbf{627.73} \\ 975.62\end{tabular} & \begin{tabular}{@{}c@{}}\textbf{1848.54} \\ 2634.09\end{tabular}   && 2220.53    & \begin{tabular}{@{}c@{}}20  \\ \textbf{23} \end{tabular}  & \begin{tabular}{@{}c@{}}1262.75 \\ \textbf{1025.97}\end{tabular} & \begin{tabular}{@{}c@{}}\textbf{649.95} \\ 949.48\end{tabular}  \\\cmidrule[.05em]{2-10}
         & 5110.84 & \begin{tabular}{@{}c@{}}\textbf{21}  \\ 10 \end{tabular}  & \begin{tabular}{@{}c@{}}1431.02 \\ \textbf{1321.87}\end{tabular} & \begin{tabular}{@{}c@{}}2503.76 \\ \textbf{2232.20}\end{tabular} && 4441.07 & \begin{tabular}{@{}c@{}}24  \\ 24\end{tabular}  & \begin{tabular}{@{}c@{}}\textbf{625.55} \\ 692.72\end{tabular} & \begin{tabular}{@{}c@{}}467.17 \\ \textbf{300.33}\end{tabular} \\\cmidrule[.05em]{2-10}
         & 7666.26 & \begin{tabular}{@{}c@{}}\textbf{22}  \\ 13\end{tabular}  & \begin{tabular}{@{}c@{}}\textbf{637.96} \\ 1324.38\end{tabular}  & \begin{tabular}{@{}c@{}}2667.32 \\ \textbf{2004.38}\end{tabular}    && 6661.61 & \begin{tabular}{@{}c@{}}24  \\ 24\end{tabular}  & \begin{tabular}{@{}c@{}}\textbf{523.98} \\ 779.72\end{tabular}  & \begin{tabular}{@{}c@{}}\textbf{429.54} \\ 537.29\end{tabular}  \\\cmidrule[.05em]{1-10}
ch150       & 1632.73 & \begin{tabular}{@{}c@{}}\textbf{21}  \\ 18 \end{tabular}  & \begin{tabular}{@{}c@{}}438.67 \\ \textbf{201.57}\end{tabular} & \begin{tabular}{@{}c@{}}3209.38 \\ \textbf{2154.22}\end{tabular} && 1470.24 & \begin{tabular}{@{}c@{}}19  \\ \textbf{20} \end{tabular}  & \begin{tabular}{@{}c@{}}841.81 \\ \textbf{538.32}\end{tabular}  & \begin{tabular}{@{}c@{}}\textbf{488.16} \\ 1149.55\end{tabular}  \\\cmidrule[.05em]{2-10}
         & 3265.45 & \begin{tabular}{@{}c@{}}18  \\ \textbf{23} \end{tabular}  & \begin{tabular}{@{}c@{}}\textbf{1049.28} \\ 1443.10\end{tabular}   & \begin{tabular}{@{}c@{}}\textbf{1960.11} \\ 2328.57\end{tabular}    && 2940.48 & \begin{tabular}{@{}c@{}}24  \\ 24 \end{tabular}  & \begin{tabular}{@{}c@{}}707.69 \\ \textbf{375.84}\end{tabular}   & \begin{tabular}{@{}c@{}}22.42 \\ \textbf{1.25}\end{tabular}  \\\cmidrule[.05em]{2-10}
         & 4898.18  & \begin{tabular}{@{}c@{}}24  \\ 24 \end{tabular}  & \begin{tabular}{@{}c@{}}\textbf{235.91} \\ 623.69\end{tabular}  & \begin{tabular}{@{}c@{}}12369.29 \\ \textbf{6573.42}\end{tabular}  && 4410.72  & \begin{tabular}{@{}c@{}}24  \\ 24 \end{tabular}  & \begin{tabular}{@{}c@{}}762.54 \\ \textbf{379.31}\end{tabular}  & \begin{tabular}{@{}c@{}}114.42 \\ \textbf{34.38}\end{tabular} \\\cmidrule[.05em]{1-10}
kroa200       & 7342.35 & \begin{tabular}{@{}c@{}}\textbf{16}  \\ 15 \end{tabular}  & \begin{tabular}{@{}c@{}}\textbf{801.16} \\ 1223.17\end{tabular} & \begin{tabular}{@{}c@{}}\textbf{2764.88} \\ 4987.73\end{tabular}   && 6483.15    & \begin{tabular}{@{}c@{}}18  \\ \textbf{22} \end{tabular}  & \begin{tabular}{@{}c@{}}2128.64 \\ \textbf{1613.38}\end{tabular} & \begin{tabular}{@{}c@{}}\textbf{384.33} \\ 494.45\end{tabular}  \\\cmidrule[.05em]{2-10}
         & 14684.70 & \begin{tabular}{@{}c@{}}24  \\ 24 \end{tabular}  & \begin{tabular}{@{}c@{}}\textbf{66.60} \\ 68.07\end{tabular} & \begin{tabular}{@{}c@{}}405.04 \\ \textbf{352.71}\end{tabular} && 12966.30 & \begin{tabular}{@{}c@{}}\textbf{14}  \\ 9\end{tabular}  & \begin{tabular}{@{}c@{}}3001.87 \\ \textbf{2713.58}\end{tabular} & \begin{tabular}{@{}c@{}}1185.57 \\ \textbf{1062.44}\end{tabular} \\\cmidrule[.05em]{2-10}
         & 22027.10 & \begin{tabular}{@{}c@{}}\textbf{21}  \\ 13\end{tabular}  & \begin{tabular}{@{}c@{}}\textbf{593.44} \\ 885.50\end{tabular}  & \begin{tabular}{@{}c@{}}8095.10 \\ \textbf{5333.23}\end{tabular}    && 19449.50 & \begin{tabular}{@{}c@{}}\textbf{2}  \\ 0\end{tabular}  & \begin{tabular}{@{}c@{}}\textbf{2694.92} \\ n/a\end{tabular}  & \begin{tabular}{@{}c@{}}\textbf{2721.00} \\ n/a\end{tabular}  
\end{tabular}
\caption{For every value of $L$ considered, we show, from left to right, the number of instances solved to optimality (out of $24$), the average CPU time, and the average number of explored nodes, where averages are computed only for the solved instances.  We report both for the tour as well as for the tree case. The first line in each cell refers to the branch-and-cut, while the second line pertains to the Benders decomposition.}
\label{computational_results}
\end{table}
 
 $21$ instances have been solved via the Benders decomposition but have not been solved within the time limit via the branch-and-cut. Conversely, $77$ instances have been solved by the branch-and-cut but not by the Benders decomposition. For every instance that can be solved by both the algorithms, we compute the ratio $\varphi$ as the CPU time of the Benders decomposition divided by that for the branch-and-cut. This ratio is less than $0.9$ in $258$ cases, and greater than $1.1$ in $366$ cases. We artificially set $\varphi=0$ if the Benders decomposition can solve the instance but the branch-and-cut cannot and $\varphi=+\infty$ for the opposite case. Table \ref{table_comparison} reports on the number of times such cases have occurred. We observe that most of the instances where the Benders decomposition outperforms the branch-and-cut have a budget $L$ corresponding to $25\%$ of the TSP or the MST optimal solution, while the performance of the branch-and-cut is clearly better as $L$ approaches the TSP or the MST solution.
These observations suggest that the two algorithms are complementary to each other, and that the Benders should be used on instances with low $L$, while the branch-and-cut on all the other ones.

\begin{table}[!p]
\centering
\scriptsize
\begin{tabular}{ccccccccc}
group & $L$ & $\varphi$ & $\varphi<1$ & $\varphi\geq 1$ & $\varphi<0.9$ & $\varphi>1.1$ & $\varphi=0$ & $\varphi=+\infty$ \\\cmidrule[.1em]{1-9}

p4 & \begin{tabular}{@{}c@{}} 176.97\\ 158.82\end{tabular} & \begin{tabular}{@{}c@{}} 1.44\\ 1.02\end{tabular} & \begin{tabular}{@{}c@{}} 11\\ 13\end{tabular} & \begin{tabular}{@{}c@{}} 13\\ 11\end{tabular} & \begin{tabular}{@{}c@{}} 9\\ 12\end{tabular} & \begin{tabular}{@{}c@{}} 12\\ 8\end{tabular} & \begin{tabular}{@{}c@{}} 0\\0 \end{tabular} & \begin{tabular}{@{}c@{}} 0\\ 0\end{tabular} \\\cmidrule[.05em]{2-9}

& \begin{tabular}{@{}c@{}} 353.93\\ 317.64\end{tabular} & \begin{tabular}{@{}c@{}} 2.14\\ 1.13\end{tabular} & \begin{tabular}{@{}c@{}} 11\\ 9\end{tabular} & \begin{tabular}{@{}c@{}} 13\\ 15\end{tabular} & \begin{tabular}{@{}c@{}} 10\\ 8\end{tabular} & \begin{tabular}{@{}c@{}} 9\\ 12\end{tabular} &\begin{tabular}{@{}c@{}} 0\\ 0\end{tabular} & \begin{tabular}{@{}c@{}} 0\\ 0\end{tabular} \\\cmidrule[.05em]{2-9}
 
& \begin{tabular}{@{}c@{}} 530.90\\ 476.50\end{tabular} & \begin{tabular}{@{}c@{}} 3.87\\ 1.72\end{tabular} & \begin{tabular}{@{}c@{}} 4\\ 4\end{tabular} & \begin{tabular}{@{}c@{}} 20\\ 20\end{tabular} & \begin{tabular}{@{}c@{}} 3\\ 3\end{tabular} & \begin{tabular}{@{}c@{}} 20\\ 19\end{tabular} & \begin{tabular}{@{}c@{}} 0\\ 0\end{tabular} & \begin{tabular}{@{}c@{}} 0\\ 0\end{tabular} \\\hline
 
p5 & \begin{tabular}{@{}c@{}} 194.11\\ 173.00\end{tabular} & \begin{tabular}{@{}c@{}} 1.33\\ 0.99\end{tabular} & \begin{tabular}{@{}c@{}} 9\\ 18\end{tabular} & \begin{tabular}{@{}c@{}} 15\\ 5\end{tabular} & \begin{tabular}{@{}c@{}} 8\\ 12\end{tabular} & \begin{tabular}{@{}c@{}} 14\\ 4\end{tabular} & \begin{tabular}{@{}c@{}} 0\\ 1\end{tabular} & \begin{tabular}{@{}c@{}} 0\\ 0\end{tabular} \\\cmidrule[.05em]{2-9}

& \begin{tabular}{@{}c@{}} 388.22\\ 345.99\end{tabular} & \begin{tabular}{@{}c@{}} 2.35\\ 1.77\end{tabular} & \begin{tabular}{@{}c@{}} 4\\ 0\end{tabular} & \begin{tabular}{@{}c@{}} 16\\ 18\end{tabular} & \begin{tabular}{@{}c@{}} 4\\ 0\end{tabular} & \begin{tabular}{@{}c@{}} 15\\ 18\end{tabular} & \begin{tabular}{@{}c@{}} 0\\ 1\end{tabular} & \begin{tabular}{@{}c@{}} 4\\ 5\end{tabular} \\\cmidrule[.05em]{2-9}
 
& \begin{tabular}{@{}c@{}} 582.33\\ 518.98\end{tabular} & \begin{tabular}{@{}c@{}} 2.23\\ 1.58\end{tabular} & \begin{tabular}{@{}c@{}} 0\\ 6\end{tabular} & \begin{tabular}{@{}c@{}} 8\\ 17\end{tabular} & \begin{tabular}{@{}c@{}} 0\\ 3\end{tabular} & \begin{tabular}{@{}c@{}} 8\\ 15\end{tabular} & \begin{tabular}{@{}c@{}} 0\\ 0\end{tabular} & \begin{tabular}{@{}c@{}} 14\\ 1\end{tabular} \\\hline
 
X-n162-k11 & \begin{tabular}{@{}c@{}} 2293.74\\ 1975.84\end{tabular} & \begin{tabular}{@{}c@{}} 1.89\\ 0.89\end{tabular} & \begin{tabular}{@{}c@{}} 5\\ 16\end{tabular} & \begin{tabular}{@{}c@{}} 16\\ 4\end{tabular} & \begin{tabular}{@{}c@{}} 5\\ 14\end{tabular} & \begin{tabular}{@{}c@{}} 15\\ 4\end{tabular} & \begin{tabular}{@{}c@{}} 2\\ 0\end{tabular} & \begin{tabular}{@{}c@{}} 1\\ 0\end{tabular} \\\cmidrule[.05em]{2-9}

& \begin{tabular}{@{}c@{}} 4587.48\\ 3951.68\end{tabular} & \begin{tabular}{@{}c@{}} 1.69\\ 1.53\end{tabular} & \begin{tabular}{@{}c@{}} 13\\ 4\end{tabular} & \begin{tabular}{@{}c@{}} 11\\ 20\end{tabular} & \begin{tabular}{@{}c@{}} 13\\ 3\end{tabular} & \begin{tabular}{@{}c@{}} 11\\ 18\end{tabular} & \begin{tabular}{@{}c@{}} 0\\ 0\end{tabular} & \begin{tabular}{@{}c@{}} 0\\ 0\end{tabular} \\\cmidrule[.05em]{2-9}
 
& \begin{tabular}{@{}c@{}} 6881.21\\ 5927.51\end{tabular} & \begin{tabular}{@{}c@{}} 1.88\\ 1.51\end{tabular} & \begin{tabular}{@{}c@{}} 10\\ 2\end{tabular} & \begin{tabular}{@{}c@{}} 14\\ 22\end{tabular} & \begin{tabular}{@{}c@{}} 9\\ 1\end{tabular} & \begin{tabular}{@{}c@{}} 13\\ 21\end{tabular} & \begin{tabular}{@{}c@{}} 0\\ 0\end{tabular} & \begin{tabular}{@{}c@{}} 0\\ 0\end{tabular} \\\hline
 
X-n195-k51 & \begin{tabular}{@{}c@{}} 2555.42\\ 2220.53\end{tabular} & \begin{tabular}{@{}c@{}} 2.48\\ 0.90\end{tabular} & \begin{tabular}{@{}c@{}} 7\\ 13\end{tabular} & \begin{tabular}{@{}c@{}} 15\\ 7\end{tabular} & \begin{tabular}{@{}c@{}} 6\\ 10\end{tabular} & \begin{tabular}{@{}c@{}} 15\\5 \end{tabular} & \begin{tabular}{@{}c@{}} 0\\ 3\end{tabular} & \begin{tabular}{@{}c@{}} 2\\ 0\end{tabular} \\\cmidrule[.05em]{2-9}

& \begin{tabular}{@{}c@{}} 5110.84\\ 4441.07\end{tabular} & \begin{tabular}{@{}c@{}} 1.73\\ 1.18\end{tabular} & \begin{tabular}{@{}c@{}} 3\\ 9\end{tabular} & \begin{tabular}{@{}c@{}} 5\\ 15\end{tabular} & \begin{tabular}{@{}c@{}} 3\\ 7\end{tabular} & \begin{tabular}{@{}c@{}} 5\\ 10\end{tabular} &\begin{tabular}{@{}c@{}} 2\\ 0\end{tabular} & \begin{tabular}{@{}c@{}} 13\\ 0\end{tabular} \\\cmidrule[.05em]{2-9}
 
& \begin{tabular}{@{}c@{}} 7666.26\\ 6661.61\end{tabular} & \begin{tabular}{@{}c@{}} 3.32\\ 1.66\end{tabular} & \begin{tabular}{@{}c@{}} 3\\ 3\end{tabular} & \begin{tabular}{@{}c@{}} 9\\ 21\end{tabular} & \begin{tabular}{@{}c@{}} 2\\ 3\end{tabular} & \begin{tabular}{@{}c@{}} 9\\ 20\end{tabular} & \begin{tabular}{@{}c@{}} 1\\ 0\end{tabular} & \begin{tabular}{@{}c@{}} 10\\ 0\end{tabular} \\\hline
 
ch150 & \begin{tabular}{@{}c@{}} 1632.73\\ 1470.24\end{tabular} & \begin{tabular}{@{}c@{}} 1.76\\ 0.46\end{tabular} & \begin{tabular}{@{}c@{}} 12\\ 19\end{tabular} & \begin{tabular}{@{}c@{}} 6\\ 0\end{tabular} & \begin{tabular}{@{}c@{}} 10\\ 18\end{tabular} & \begin{tabular}{@{}c@{}} 6\\0 \end{tabular} & \begin{tabular}{@{}c@{}} 0\\0 \end{tabular} & \begin{tabular}{@{}c@{}} 3\\ 1\end{tabular} \\\cmidrule[.05em]{2-9}
 
& \begin{tabular}{@{}c@{}} 3265.45\\ 2940.48\end{tabular} & \begin{tabular}{@{}c@{}} 2.71\\ 0.58\end{tabular} & \begin{tabular}{@{}c@{}} 8\\ 20\end{tabular} & \begin{tabular}{@{}c@{}} 9\\ 4\end{tabular} & \begin{tabular}{@{}c@{}} 8\\ 20\end{tabular} & \begin{tabular}{@{}c@{}} 8\\ 1\end{tabular} & \begin{tabular}{@{}c@{}} 1\\ 0\end{tabular} & \begin{tabular}{@{}c@{}} 1\\ 0\end{tabular} \\\cmidrule[.05em]{2-9}

& \begin{tabular}{@{}c@{}} 4898.18\\ 4410.72\end{tabular} & \begin{tabular}{@{}c@{}} 3.38\\ 0.56\end{tabular} & \begin{tabular}{@{}c@{}} 6\\ 23\end{tabular} & \begin{tabular}{@{}c@{}} 18\\ 1\end{tabular} & \begin{tabular}{@{}c@{}} 5\\ 21\end{tabular} & \begin{tabular}{@{}c@{}} 18\\ 1\end{tabular} & \begin{tabular}{@{}c@{}} 0\\ 0\end{tabular} & \begin{tabular}{@{}c@{}} 0\\ 0\end{tabular} \\\hline
 
kroa200 & \begin{tabular}{@{}c@{}} 7342.35\\ 6483.15\end{tabular} & \begin{tabular}{@{}c@{}} 2.14\\ 0.79\end{tabular} & \begin{tabular}{@{}c@{}} 3\\ 16\end{tabular} & \begin{tabular}{@{}c@{}} 11\\ 2\end{tabular} & \begin{tabular}{@{}c@{}} 3\\ 11\end{tabular} & \begin{tabular}{@{}c@{}} 8\\ 1\end{tabular} & \begin{tabular}{@{}c@{}} 1\\ 4\end{tabular} & \begin{tabular}{@{}c@{}} 2\\ 0\end{tabular} \\\cmidrule[.05em]{2-9}

& \begin{tabular}{@{}c@{}} 14684.70\\ 12966.30\end{tabular} & \begin{tabular}{@{}c@{}} 1.38\\ 0.97\end{tabular} & \begin{tabular}{@{}c@{}} 10\\ 4\end{tabular} & \begin{tabular}{@{}c@{}} 14\\ 3\end{tabular} & \begin{tabular}{@{}c@{}} 9\\ 3\end{tabular} & \begin{tabular}{@{}c@{}} 13\\ 2\end{tabular} & \begin{tabular}{@{}c@{}} 0\\ 2\end{tabular} & \begin{tabular}{@{}c@{}} 0\\ 7\end{tabular} \\\cmidrule[.05em]{2-9}
 
& \begin{tabular}{@{}c@{}} 22027.10\\ 19449.50\end{tabular} & \begin{tabular}{@{}c@{}} 1.92\\ n/a\end{tabular} & \begin{tabular}{@{}c@{}} 2\\ n/a\end{tabular} & \begin{tabular}{@{}c@{}} 8\\ n/a\end{tabular} & \begin{tabular}{@{}c@{}} 2\\ n/a\end{tabular} & \begin{tabular}{@{}c@{}} 8\\ n/a\end{tabular} & \begin{tabular}{@{}c@{}} 3\\ 0\end{tabular} & \begin{tabular}{@{}c@{}} 11\\ 2\end{tabular}
\end{tabular}
\caption{We show, clustered by group and budget $L$, consecutively the average ratio $\varphi$ and the number of instances where $\varphi$ satisfies the condition in the header of the column. The first line in each cell refers to the tour case and the second line to the tree case.}
\label{table_comparison}
\end{table}

\subsection{Impact of the parameters on the optimal solutions}

When $L$ is sufficiently large, the tour subgraph problem reduces to the TSP and analogously, the tree subgraph case reduces to the MST. By contrast, when $L$ is comparatively low with respect to the TSP (MST) optimal value, the most important decision is the choice of the vertices to visit, and thus the problem's solution approaches that of an OP (or analogous tree variant). Figure \ref{final_fig_11} and Figure \ref{final_fig_12} compare optimal solutions of two instances with high and low values of $L$, for both the tour and the tree subgraph cases respectively.

\begin{figure}[t]
\centering
\begin{adjustbox}{valign=t,minipage={.45\textwidth}}
\centering
\begin{tikzpicture}[scale=0.09]
\tikzstyle{every node}=[draw,circle,fill=white,minimum size=4pt,
                            inner sep=0pt]
\draw (35, 35) node (0) [overlay, color=black] {};
\draw (22, 22) node (1) [overlay] {};
\draw (36, 26) node (2) [overlay] {};
\draw (21, 45) node (3) [overlay] {};
\draw (45, 35) node (4) [overlay] {};
\draw (55, 20) node (5) [overlay] {};
\draw (33, 34) node (6) [overlay] {};
\draw (50, 50) node (7) [overlay] {};
\draw (55, 45) node (8) [overlay] {};
\draw (26, 59) node (9) [overlay] {};
\draw (40, 66) node (10) [overlay] {};
\draw (55, 65) node (11) [overlay] {};
\draw (35, 51) node (12) [overlay] {};
\draw (62, 35) node (13) [overlay] {};
\draw (62, 57) node (14) [overlay] {};
\draw (62, 24) node (15) [overlay] {};
\draw (21, 36) node (16) [overlay] {};
\draw (33, 44) node (17) [overlay] {};
\draw (9, 56) node (18) [overlay] {};
\draw (62, 48) node (19) [overlay] {};
\draw (66, 14) node (20) [overlay] {};
\draw (44, 13) node (21) [overlay] {};
\draw (26, 13) node (22) [overlay] {};
\draw (11, 28) node (23) [overlay] {};
\draw (7, 43) node (24) [overlay] {};
\draw (17, 64) node (25) [overlay] {};
\draw (41, 46) node (26) [overlay] {};
\draw (55, 34) node (27) [overlay] {};
\draw (35, 16) node (28) [overlay] {};
\draw (52, 26) node (29) [overlay] {};
\draw (43, 26) node (30) [overlay] {};
\draw (31, 76) node (31) [overlay] {};
\draw (22, 53) node (32) [overlay] {};
\draw (26, 29) node (33) [overlay] {};
\draw (50, 40) node (34) [overlay] {};
\draw (55, 50) node (35) [overlay] {};
\draw (54, 10) node (36) [overlay] {};
\draw (60, 15) node (37) [overlay] {};
\draw (47, 66) node (38) [overlay] {};
\draw (30, 60) node (39) [overlay] {};
\draw (30, 50) node (40) [overlay] {};
\draw (12, 17) node (41) [overlay] {};
\draw (15, 14) node (42) [overlay] {};
\draw (16, 19) node (43) [overlay] {};
\draw (21, 48) node (44) [overlay] {};
\draw (50, 30) node (45) [overlay] {};
\draw (51, 42) node (46) [overlay] {};
\draw (50, 15) node (47) [overlay] {};
\draw (48, 21) node (48) [overlay] {};
\draw (12, 38) node (49) [overlay] {};
\draw (37, 52) node (50) [overlay] {};
\draw (49, 49) node (51) [overlay] {};
\draw (52, 64) node (52) [overlay] {};
\draw (20, 26) node (53) [overlay] {};
\draw (40, 30) node (54) [overlay] {};
\draw (21, 47) node (55) [overlay] {};
\draw (17, 63) node (56) [overlay] {};
\draw (31, 62) node (57) [overlay] {};
\draw (52, 33) node (58) [overlay] {};
\draw (51, 21) node (59) [overlay] {};
\draw (42, 41) node (60) [overlay] {};
\draw (31, 32) node (61) [overlay] {};
\draw (5, 25) node (62) [overlay] {};
\draw (12, 42) node (63) [overlay] {};
\draw (36, 16) node (64) [overlay] {};
\draw (52, 41) node (65) [overlay] {};
\draw (27, 23) node (66) [overlay] {};
\draw (17, 33) node (67) [overlay] {};
\draw (13, 13) node (68) [overlay] {};
\draw (57, 58) node (69) [overlay] {};
\draw (62, 42) node (70) [overlay] {};
\draw (42, 57) node (71) [overlay] {};
\draw (16, 57) node (72) [overlay] {};
\draw (8, 52) node (73) [overlay] {};
\draw (7, 38) node (74) [overlay] {};
\draw (27, 68) node (75) [overlay] {};
\draw (30, 48) node (76) [overlay] {};
\draw (43, 67) node (77) [overlay] {};
\draw (58, 48) node (78) [overlay] {};
\draw (58, 27) node (79) [overlay] {};
\draw (37, 69) node (80) [overlay] {};
\draw (38, 46) node (81) [overlay] {};
\draw (46, 10) node (82) [overlay] {};
\draw (61, 33) node (83) [overlay] {};
\draw (62, 63) node (84) [overlay] {};
\draw (63, 69) node (85) [overlay] {};
\draw (32, 22) node (86) [overlay] {};
\draw (45, 35) node (87) [overlay] {};
\draw (59, 15) node (88) [overlay] {};
\draw (5, 6) node (89) [overlay] {};
\draw (10, 17) node (90) [overlay] {};
\draw (21, 10) node (91) [overlay] {};
\draw (5, 64) node (92) [overlay] {};
\draw (30, 15) node (93) [overlay] {};
\draw (39, 10) node (94) [overlay] {};
\draw (32, 39) node (95) [overlay] {};
\draw (25, 32) node (96) [overlay] {};
\draw (25, 55) node (97) [overlay] {};
\draw (48, 28) node (98) [overlay] {};
\draw (56, 37) node (99) [overlay] {};
\draw (41, 49) node (100) [overlay] {};
\draw (35, 17) node (101) [overlay] {};
\draw (55, 45) node (102) [overlay] {};
\draw (55, 20) node (103) [overlay] {};
\draw (15, 30) node (104) [overlay] {};
\draw (25, 30) node (105) [overlay] {};
\draw (20, 50) node (106) [overlay] {};
\draw (10, 43) node (107) [overlay] {};
\draw (55, 60) node (108) [overlay] {};
\draw (30, 60) node (109) [overlay] {};
\draw (20, 65) node (110) [overlay] {};
\draw (50, 35) node (111) [overlay] {};
\draw (30, 25) node (112) [overlay] {};
\draw (15, 10) node (113) [overlay] {};
\draw (30, 5) node (114) [overlay] {};
\draw (10, 20) node (115) [overlay] {};
\draw (5, 30) node (116) [overlay] {};
\draw (20, 40) node (117) [overlay] {};
\draw (15, 60) node (118) [overlay] {};
\draw (45, 65) node (119) [overlay] {};
\draw (45, 20) node (120) [overlay] {};
\draw (45, 10) node (121) [overlay] {};
\draw (55, 5) node (122) [overlay] {};
\draw (65, 35) node (123) [overlay] {};
\draw (65, 20) node (124) [overlay] {};
\draw (45, 30) node (125) [overlay] {};
\draw (35, 40) node (126) [overlay] {};
\draw (41, 37) node (127) [overlay] {};
\draw (64, 42) node (128) [overlay] {};
\draw (40, 60) node (129) [overlay] {};
\draw (31, 52) node (130) [overlay] {};
\draw (35, 69) node (131) [overlay] {};
\draw (53, 52) node (132) [overlay] {};
\draw (65, 55) node (133) [overlay] {};
\draw (63, 65) node (134) [overlay] {};
\draw (2, 60) node (135) [overlay] {};
\draw (20, 20) node (136) [overlay] {};
\draw (5, 5) node (137) [overlay] {};
\draw (60, 12) node (138) [overlay] {};
\draw (40, 25) node (139) [overlay] {};
\draw (42, 7) node (140) [overlay] {};
\draw (24, 12) node (141) [overlay] {};
\draw (23, 3) node (142) [overlay] {};
\draw (11, 14) node (143) [overlay] {};
\draw (6, 38) node (144) [overlay] {};
\draw (2, 48) node (145) [overlay] {};
\draw (8, 56) node (146) [overlay] {};
\draw (13, 52) node (147) [overlay] {};
\draw (6, 68) node (148) [overlay] {};
\draw (47, 47) node (149) [overlay] {};
\draw (49, 58) node (150) [overlay] {};
\draw (27, 43) node (151) [overlay] {};
\draw (37, 31) node (152) [overlay] {};
\draw (57, 29) node (153) [overlay] {};
\draw (63, 23) node (154) [overlay] {};
\draw (53, 12) node (155) [overlay] {};
\draw (32, 12) node (156) [overlay] {};
\draw (36, 26) node (157) [overlay] {};
\draw (21, 24) node (158) [overlay] {};
\draw (17, 34) node (159) [overlay] {};
\draw (12, 24) node (160) [overlay] {};
\draw (24, 58) node (161) [overlay] {};
\draw (27, 69) node (162) [overlay] {};
\draw (15, 77) node (163) [overlay] {};
\draw (62, 77) node (164) [overlay] {};
\draw (49, 73) node (165) [overlay] {};
\draw (67, 5) node (166) [overlay] {};
\draw (56, 39) node (167) [overlay] {};
\draw (37, 47) node (168) [overlay] {};
\draw (37, 56) node (169) [overlay] {};
\draw (57, 68) node (170) [overlay] {};
\draw (47, 16) node (171) [overlay] {};
\draw (44, 17) node (172) [overlay] {};
\draw (46, 13) node (173) [overlay] {};
\draw (49, 11) node (174) [overlay] {};
\draw (49, 42) node (175) [overlay] {};
\draw (53, 43) node (176) [overlay] {};
\draw (61, 52) node (177) [overlay] {};
\draw (57, 48) node (178) [overlay] {};
\draw (56, 37) node (179) [overlay] {};
\draw (55, 54) node (180) [overlay] {};
\draw (15, 47) node (181) [overlay] {};
\draw (14, 37) node (182) [overlay] {};
\draw (11, 31) node (183) [overlay] {};
\draw (16, 22) node (184) [overlay] {};
\draw (4, 18) node (185) [overlay] {};
\draw (28, 18) node (186) [overlay] {};
\draw (26, 52) node (187) [overlay] {};
\draw (26, 35) node (188) [overlay] {};
\draw (31, 67) node (189) [overlay] {};
\draw (15, 19) node (190) [overlay] {};
\draw (22, 22) node (191) [overlay] {};
\draw (18, 24) node (192) [overlay] {};
\draw (26, 27) node (193) [overlay] {};
\draw (25, 24) node (194) [overlay] {};
\draw (22, 27) node (195) [overlay] {};
\draw (25, 21) node (196) [overlay] {};
\draw (19, 21) node (197) [overlay] {};
\draw (20, 26) node (198) [overlay] {};
\draw (18, 18) node (199) [overlay] {};
\draw (0) -- (6) [very thick, color=red];
\draw (0) -- (152) [very thick, color=red];
\draw (1) -- (191) [very thick, color=red];
\draw (1) -- (194) [very thick, color=red];
\draw (2) -- (139) [very thick, color=red];
\draw (2) -- (157) [very thick, color=red];
\draw (3) -- (55) [very thick, color=red];
\draw (3) -- (117) [very thick, color=red];
\draw (5) -- (59) [very thick, color=red];
\draw (5) -- (103) [very thick, color=red];
\draw (6) -- (61) [very thick, color=red];
\draw (8) -- (102) [very thick, color=red];
\draw (8) -- (176) [very thick, color=red];
\draw (9) -- (39) [very thick, color=red];
\draw (9) -- (161) [very thick, color=red];
\draw (10) -- (77) [very thick, color=red];
\draw (10) -- (80) [very thick, color=red];
\draw (11) -- (52) [very thick, color=red];
\draw (11) -- (108) [very thick, color=red];
\draw (16) -- (67) [very thick, color=red];
\draw (16) -- (117) [very thick, color=red];
\draw (27) -- (153) [very thick, color=red];
\draw (27) -- (179) [very thick, color=red];
\draw (30) -- (48) [very thick, color=red];
\draw (30) -- (139) [very thick, color=red];
\draw (32) -- (97) [very thick, color=red];
\draw (32) -- (106) [very thick, color=red];
\draw (33) -- (105) [very thick, color=red];
\draw (33) -- (193) [very thick, color=red];
\draw (35) -- (132) [very thick, color=red];
\draw (35) -- (178) [very thick, color=red];
\draw (38) -- (52) [very thick, color=red];
\draw (38) -- (119) [very thick, color=red];
\draw (39) -- (109) [very thick, color=red];
\draw (44) -- (55) [very thick, color=red];
\draw (44) -- (106) [very thick, color=red];
\draw (46) -- (65) [very thick, color=red];
\draw (46) -- (176) [very thick, color=red];
\draw (48) -- (59) [very thick, color=red];
\draw (53) -- (158) [very thick, color=red];
\draw (53) -- (198) [very thick, color=red];
\draw (57) -- (109) [very thick, color=red];
\draw (57) -- (131) [very thick, color=red];
\draw (61) -- (96) [very thick, color=red];
\draw (65) -- (167) [very thick, color=red];
\draw (67) -- (104) [very thick, color=red];
\draw (69) -- (108) [very thick, color=red];
\draw (69) -- (180) [very thick, color=red];
\draw (77) -- (119) [very thick, color=red];
\draw (79) -- (103) [very thick, color=red];
\draw (79) -- (153) [very thick, color=red];
\draw (80) -- (131) [very thick, color=red];
\draw (96) -- (105) [very thick, color=red];
\draw (97) -- (161) [very thick, color=red];
\draw (99) -- (167) [very thick, color=red];
\draw (99) -- (179) [very thick, color=red];
\draw (102) -- (178) [very thick, color=red];
\draw (104) -- (198) [very thick, color=red];
\draw (132) -- (180) [very thick, color=red];
\draw (152) -- (157) [very thick, color=red];
\draw (158) -- (191) [very thick, color=red];
\draw (193) -- (194) [very thick, color=red];
\draw (4) -- (102) [dashed, very thin, color=teal];
\draw (7) -- (99) [dashed, very thin, color=teal];
\draw (12) -- (10) [dashed, very thin, color=teal];
\draw (13) -- (178) [dashed, very thin, color=teal];
\draw (14) -- (69) [dashed, very thin, color=teal];
\draw (15) -- (179) [dashed, very thin, color=teal];
\draw (17) -- (39) [dashed, very thin, color=teal];
\draw (18) -- (55) [dashed, very thin, color=teal];
\draw (19) -- (108) [dashed, very thin, color=teal];
\draw (20) -- (79) [dashed, very thin, color=teal];
\draw (21) -- (157) [dashed, very thin, color=teal];
\draw (22) -- (33) [dashed, very thin, color=teal];
\draw (23) -- (117) [dashed, very thin, color=teal];
\draw (24) -- (106) [dashed, very thin, color=teal];
\draw (25) -- (97) [dashed, very thin, color=teal];
\draw (26) -- (35) [dashed, very thin, color=teal];
\draw (28) -- (33) [dashed, very thin, color=teal];
\draw (29) -- (46) [dashed, very thin, color=teal];
\draw (31) -- (77) [dashed, very thin, color=teal];
\draw (34) -- (65) [dashed, very thin, color=teal];
\draw (36) -- (5) [dashed, very thin, color=teal];
\draw (37) -- (79) [dashed, very thin, color=teal];
\draw (40) -- (32) [dashed, very thin, color=teal];
\draw (41) -- (104) [dashed, very thin, color=teal];
\draw (42) -- (194) [dashed, very thin, color=teal];
\draw (43) -- (67) [dashed, very thin, color=teal];
\draw (45) -- (102) [dashed, very thin, color=teal];
\draw (47) -- (139) [dashed, very thin, color=teal];
\draw (49) -- (106) [dashed, very thin, color=teal];
\draw (50) -- (77) [dashed, very thin, color=teal];
\draw (51) -- (69) [dashed, very thin, color=teal];
\draw (54) -- (6) [dashed, very thin, color=teal];
\draw (56) -- (9) [dashed, very thin, color=teal];
\draw (58) -- (167) [dashed, very thin, color=teal];
\draw (60) -- (167) [dashed, very thin, color=teal];
\draw (62) -- (198) [dashed, very thin, color=teal];
\draw (63) -- (32) [dashed, very thin, color=teal];
\draw (64) -- (191) [dashed, very thin, color=teal];
\draw (68) -- (191) [dashed, very thin, color=teal];
\draw (70) -- (27) [dashed, very thin, color=teal];
\draw (71) -- (108) [dashed, very thin, color=teal];
\draw (72) -- (57) [dashed, very thin, color=teal];
\draw (73) -- (3) [dashed, very thin, color=teal];
\draw (74) -- (3) [dashed, very thin, color=teal];
\draw (75) -- (10) [dashed, very thin, color=teal];
\draw (76) -- (61) [dashed, very thin, color=teal];
\draw (78) -- (8) [dashed, very thin, color=teal];
\draw (81) -- (132) [dashed, very thin, color=teal];
\draw (82) -- (139) [dashed, very thin, color=teal];
\draw (83) -- (179) [dashed, very thin, color=teal];
\draw (84) -- (180) [dashed, very thin, color=teal];
\draw (85) -- (11) [dashed, very thin, color=teal];
\draw (86) -- (96) [dashed, very thin, color=teal];
\draw (87) -- (99) [dashed, very thin, color=teal];
\draw (88) -- (153) [dashed, very thin, color=teal];
\draw (90) -- (1) [dashed, very thin, color=teal];
\draw (91) -- (198) [dashed, very thin, color=teal];
\draw (93) -- (157) [dashed, very thin, color=teal];
\draw (94) -- (2) [dashed, very thin, color=teal];
\draw (95) -- (44) [dashed, very thin, color=teal];
\draw (98) -- (176) [dashed, very thin, color=teal];
\draw (100) -- (132) [dashed, very thin, color=teal];
\draw (101) -- (152) [dashed, very thin, color=teal];
\draw (107) -- (55) [dashed, very thin, color=teal];
\draw (110) -- (161) [dashed, very thin, color=teal];
\draw (111) -- (178) [dashed, very thin, color=teal];
\draw (112) -- (0) [dashed, very thin, color=teal];
\draw (113) -- (158) [dashed, very thin, color=teal];
\draw (115) -- (53) [dashed, very thin, color=teal];
\draw (118) -- (57) [dashed, very thin, color=teal];
\draw (120) -- (152) [dashed, very thin, color=teal];
\draw (121) -- (30) [dashed, very thin, color=teal];
\draw (122) -- (5) [dashed, very thin, color=teal];
\draw (123) -- (8) [dashed, very thin, color=teal];
\draw (124) -- (153) [dashed, very thin, color=teal];
\draw (125) -- (65) [dashed, very thin, color=teal];
\draw (126) -- (46) [dashed, very thin, color=teal];
\draw (127) -- (176) [dashed, very thin, color=teal];
\draw (128) -- (27) [dashed, very thin, color=teal];
\draw (129) -- (80) [dashed, very thin, color=teal];
\draw (130) -- (109) [dashed, very thin, color=teal];
\draw (133) -- (52) [dashed, very thin, color=teal];
\draw (134) -- (38) [dashed, very thin, color=teal];
\draw (136) -- (61) [dashed, very thin, color=teal];
\draw (138) -- (103) [dashed, very thin, color=teal];
\draw (140) -- (48) [dashed, very thin, color=teal];
\draw (141) -- (194) [dashed, very thin, color=teal];
\draw (143) -- (53) [dashed, very thin, color=teal];
\draw (144) -- (117) [dashed, very thin, color=teal];
\draw (146) -- (161) [dashed, very thin, color=teal];
\draw (147) -- (9) [dashed, very thin, color=teal];
\draw (149) -- (180) [dashed, very thin, color=teal];
\draw (150) -- (119) [dashed, very thin, color=teal];
\draw (151) -- (44) [dashed, very thin, color=teal];
\draw (155) -- (59) [dashed, very thin, color=teal];
\draw (156) -- (193) [dashed, very thin, color=teal];
\draw (159) -- (6) [dashed, very thin, color=teal];
\draw (160) -- (193) [dashed, very thin, color=teal];
\draw (162) -- (80) [dashed, very thin, color=teal];
\draw (164) -- (11) [dashed, very thin, color=teal];
\draw (165) -- (131) [dashed, very thin, color=teal];
\draw (168) -- (39) [dashed, very thin, color=teal];
\draw (169) -- (131) [dashed, very thin, color=teal];
\draw (170) -- (119) [dashed, very thin, color=teal];
\draw (171) -- (30) [dashed, very thin, color=teal];
\draw (172) -- (103) [dashed, very thin, color=teal];
\draw (173) -- (48) [dashed, very thin, color=teal];
\draw (174) -- (59) [dashed, very thin, color=teal];
\draw (175) -- (35) [dashed, very thin, color=teal];
\draw (177) -- (52) [dashed, very thin, color=teal];
\draw (181) -- (97) [dashed, very thin, color=teal];
\draw (182) -- (105) [dashed, very thin, color=teal];
\draw (183) -- (16) [dashed, very thin, color=teal];
\draw (184) -- (158) [dashed, very thin, color=teal];
\draw (185) -- (104) [dashed, very thin, color=teal];
\draw (186) -- (1) [dashed, very thin, color=teal];
\draw (187) -- (109) [dashed, very thin, color=teal];
\draw (188) -- (2) [dashed, very thin, color=teal];
\draw (189) -- (38) [dashed, very thin, color=teal];
\draw (192) -- (16) [dashed, very thin, color=teal];
\draw (195) -- (0) [dashed, very thin, color=teal];
\draw (196) -- (67) [dashed, very thin, color=teal];
\draw (197) -- (105) [dashed, very thin, color=teal];
\draw (199) -- (96) [dashed, very thin, color=teal];
\end{tikzpicture}
\begin{center}An optimal tour with $L=194.11$.\end{center}
\end{adjustbox}\hfill
\begin{adjustbox}{valign=t,minipage={.45\textwidth}}
\centering
\begin{tikzpicture}[scale=0.09]
\tikzstyle{every node}=[draw,circle,fill=white,minimum size=4pt,
                            inner sep=0pt]
\draw (35, 35) node (0) [overlay, color=black] {};
\draw (22, 22) node (1) [overlay] {};
\draw (36, 26) node (2) [overlay] {};
\draw (21, 45) node (3) [overlay] {};
\draw (45, 35) node (4) [overlay] {};
\draw (55, 20) node (5) [overlay] {};
\draw (33, 34) node (6) [overlay] {};
\draw (50, 50) node (7) [overlay] {};
\draw (55, 45) node (8) [overlay] {};
\draw (26, 59) node (9) [overlay] {};
\draw (40, 66) node (10) [overlay] {};
\draw (55, 65) node (11) [overlay] {};
\draw (35, 51) node (12) [overlay] {};
\draw (62, 35) node (13) [overlay] {};
\draw (62, 57) node (14) [overlay] {};
\draw (62, 24) node (15) [overlay] {};
\draw (21, 36) node (16) [overlay] {};
\draw (33, 44) node (17) [overlay] {};
\draw (9, 56) node (18) [overlay] {};
\draw (62, 48) node (19) [overlay] {};
\draw (66, 14) node (20) [overlay] {};
\draw (44, 13) node (21) [overlay] {};
\draw (26, 13) node (22) [overlay] {};
\draw (11, 28) node (23) [overlay] {};
\draw (7, 43) node (24) [overlay] {};
\draw (17, 64) node (25) [overlay] {};
\draw (41, 46) node (26) [overlay] {};
\draw (55, 34) node (27) [overlay] {};
\draw (35, 16) node (28) [overlay] {};
\draw (52, 26) node (29) [overlay] {};
\draw (43, 26) node (30) [overlay] {};
\draw (31, 76) node (31) [overlay] {};
\draw (22, 53) node (32) [overlay] {};
\draw (26, 29) node (33) [overlay] {};
\draw (50, 40) node (34) [overlay] {};
\draw (55, 50) node (35) [overlay] {};
\draw (54, 10) node (36) [overlay] {};
\draw (60, 15) node (37) [overlay] {};
\draw (47, 66) node (38) [overlay] {};
\draw (30, 60) node (39) [overlay] {};
\draw (30, 50) node (40) [overlay] {};
\draw (12, 17) node (41) [overlay] {};
\draw (15, 14) node (42) [overlay] {};
\draw (16, 19) node (43) [overlay] {};
\draw (21, 48) node (44) [overlay] {};
\draw (50, 30) node (45) [overlay] {};
\draw (51, 42) node (46) [overlay] {};
\draw (50, 15) node (47) [overlay] {};
\draw (48, 21) node (48) [overlay] {};
\draw (12, 38) node (49) [overlay] {};
\draw (37, 52) node (50) [overlay] {};
\draw (49, 49) node (51) [overlay] {};
\draw (52, 64) node (52) [overlay] {};
\draw (20, 26) node (53) [overlay] {};
\draw (40, 30) node (54) [overlay] {};
\draw (21, 47) node (55) [overlay] {};
\draw (17, 63) node (56) [overlay] {};
\draw (31, 62) node (57) [overlay] {};
\draw (52, 33) node (58) [overlay] {};
\draw (51, 21) node (59) [overlay] {};
\draw (42, 41) node (60) [overlay] {};
\draw (31, 32) node (61) [overlay] {};
\draw (5, 25) node (62) [overlay] {};
\draw (12, 42) node (63) [overlay] {};
\draw (36, 16) node (64) [overlay] {};
\draw (52, 41) node (65) [overlay] {};
\draw (27, 23) node (66) [overlay] {};
\draw (17, 33) node (67) [overlay] {};
\draw (13, 13) node (68) [overlay] {};
\draw (57, 58) node (69) [overlay] {};
\draw (62, 42) node (70) [overlay] {};
\draw (42, 57) node (71) [overlay] {};
\draw (16, 57) node (72) [overlay] {};
\draw (8, 52) node (73) [overlay] {};
\draw (7, 38) node (74) [overlay] {};
\draw (27, 68) node (75) [overlay] {};
\draw (30, 48) node (76) [overlay] {};
\draw (43, 67) node (77) [overlay] {};
\draw (58, 48) node (78) [overlay] {};
\draw (58, 27) node (79) [overlay] {};
\draw (37, 69) node (80) [overlay] {};
\draw (38, 46) node (81) [overlay] {};
\draw (46, 10) node (82) [overlay] {};
\draw (61, 33) node (83) [overlay] {};
\draw (62, 63) node (84) [overlay] {};
\draw (63, 69) node (85) [overlay] {};
\draw (32, 22) node (86) [overlay] {};
\draw (45, 35) node (87) [overlay] {};
\draw (59, 15) node (88) [overlay] {};
\draw (5, 6) node (89) [overlay] {};
\draw (10, 17) node (90) [overlay] {};
\draw (21, 10) node (91) [overlay] {};
\draw (5, 64) node (92) [overlay] {};
\draw (30, 15) node (93) [overlay] {};
\draw (39, 10) node (94) [overlay] {};
\draw (32, 39) node (95) [overlay] {};
\draw (25, 32) node (96) [overlay] {};
\draw (25, 55) node (97) [overlay] {};
\draw (48, 28) node (98) [overlay] {};
\draw (56, 37) node (99) [overlay] {};
\draw (41, 49) node (100) [overlay] {};
\draw (35, 17) node (101) [overlay] {};
\draw (55, 45) node (102) [overlay] {};
\draw (55, 20) node (103) [overlay] {};
\draw (15, 30) node (104) [overlay] {};
\draw (25, 30) node (105) [overlay] {};
\draw (20, 50) node (106) [overlay] {};
\draw (10, 43) node (107) [overlay] {};
\draw (55, 60) node (108) [overlay] {};
\draw (30, 60) node (109) [overlay] {};
\draw (20, 65) node (110) [overlay] {};
\draw (50, 35) node (111) [overlay] {};
\draw (30, 25) node (112) [overlay] {};
\draw (15, 10) node (113) [overlay] {};
\draw (30, 5) node (114) [overlay] {};
\draw (10, 20) node (115) [overlay] {};
\draw (5, 30) node (116) [overlay] {};
\draw (20, 40) node (117) [overlay] {};
\draw (15, 60) node (118) [overlay] {};
\draw (45, 65) node (119) [overlay] {};
\draw (45, 20) node (120) [overlay] {};
\draw (45, 10) node (121) [overlay] {};
\draw (55, 5) node (122) [overlay] {};
\draw (65, 35) node (123) [overlay] {};
\draw (65, 20) node (124) [overlay] {};
\draw (45, 30) node (125) [overlay] {};
\draw (35, 40) node (126) [overlay] {};
\draw (41, 37) node (127) [overlay] {};
\draw (64, 42) node (128) [overlay] {};
\draw (40, 60) node (129) [overlay] {};
\draw (31, 52) node (130) [overlay] {};
\draw (35, 69) node (131) [overlay] {};
\draw (53, 52) node (132) [overlay] {};
\draw (65, 55) node (133) [overlay] {};
\draw (63, 65) node (134) [overlay] {};
\draw (2, 60) node (135) [overlay] {};
\draw (20, 20) node (136) [overlay] {};
\draw (5, 5) node (137) [overlay] {};
\draw (60, 12) node (138) [overlay] {};
\draw (40, 25) node (139) [overlay] {};
\draw (42, 7) node (140) [overlay] {};
\draw (24, 12) node (141) [overlay] {};
\draw (23, 3) node (142) [overlay] {};
\draw (11, 14) node (143) [overlay] {};
\draw (6, 38) node (144) [overlay] {};
\draw (2, 48) node (145) [overlay] {};
\draw (8, 56) node (146) [overlay] {};
\draw (13, 52) node (147) [overlay] {};
\draw (6, 68) node (148) [overlay] {};
\draw (47, 47) node (149) [overlay] {};
\draw (49, 58) node (150) [overlay] {};
\draw (27, 43) node (151) [overlay] {};
\draw (37, 31) node (152) [overlay] {};
\draw (57, 29) node (153) [overlay] {};
\draw (63, 23) node (154) [overlay] {};
\draw (53, 12) node (155) [overlay] {};
\draw (32, 12) node (156) [overlay] {};
\draw (36, 26) node (157) [overlay] {};
\draw (21, 24) node (158) [overlay] {};
\draw (17, 34) node (159) [overlay] {};
\draw (12, 24) node (160) [overlay] {};
\draw (24, 58) node (161) [overlay] {};
\draw (27, 69) node (162) [overlay] {};
\draw (15, 77) node (163) [overlay] {};
\draw (62, 77) node (164) [overlay] {};
\draw (49, 73) node (165) [overlay] {};
\draw (67, 5) node (166) [overlay] {};
\draw (56, 39) node (167) [overlay] {};
\draw (37, 47) node (168) [overlay] {};
\draw (37, 56) node (169) [overlay] {};
\draw (57, 68) node (170) [overlay] {};
\draw (47, 16) node (171) [overlay] {};
\draw (44, 17) node (172) [overlay] {};
\draw (46, 13) node (173) [overlay] {};
\draw (49, 11) node (174) [overlay] {};
\draw (49, 42) node (175) [overlay] {};
\draw (53, 43) node (176) [overlay] {};
\draw (61, 52) node (177) [overlay] {};
\draw (57, 48) node (178) [overlay] {};
\draw (56, 37) node (179) [overlay] {};
\draw (55, 54) node (180) [overlay] {};
\draw (15, 47) node (181) [overlay] {};
\draw (14, 37) node (182) [overlay] {};
\draw (11, 31) node (183) [overlay] {};
\draw (16, 22) node (184) [overlay] {};
\draw (4, 18) node (185) [overlay] {};
\draw (28, 18) node (186) [overlay] {};
\draw (26, 52) node (187) [overlay] {};
\draw (26, 35) node (188) [overlay] {};
\draw (31, 67) node (189) [overlay] {};
\draw (15, 19) node (190) [overlay] {};
\draw (22, 22) node (191) [overlay] {};
\draw (18, 24) node (192) [overlay] {};
\draw (26, 27) node (193) [overlay] {};
\draw (25, 24) node (194) [overlay] {};
\draw (22, 27) node (195) [overlay] {};
\draw (25, 21) node (196) [overlay] {};
\draw (19, 21) node (197) [overlay] {};
\draw (20, 26) node (198) [overlay] {};
\draw (18, 18) node (199) [overlay] {};
\draw (0) -- (6) [very thick, color=red];
\draw (0) -- (126) [very thick, color=red];
\draw (1) -- (136) [very thick, color=red];
\draw (1) -- (191) [very thick, color=red];
\draw (2) -- (54) [very thick, color=red];
\draw (2) -- (157) [very thick, color=red];
\draw (3) -- (55) [very thick, color=red];
\draw (3) -- (117) [very thick, color=red];
\draw (4) -- (87) [very thick, color=red];
\draw (4) -- (127) [very thick, color=red];
\draw (5) -- (79) [very thick, color=red];
\draw (5) -- (103) [very thick, color=red];
\draw (6) -- (61) [very thick, color=red];
\draw (7) -- (51) [very thick, color=red];
\draw (7) -- (132) [very thick, color=red];
\draw (8) -- (102) [very thick, color=red];
\draw (8) -- (176) [very thick, color=red];
\draw (9) -- (39) [very thick, color=red];
\draw (9) -- (161) [very thick, color=red];
\draw (10) -- (77) [very thick, color=red];
\draw (10) -- (80) [very thick, color=red];
\draw (11) -- (52) [very thick, color=red];
\draw (11) -- (170) [very thick, color=red];
\draw (12) -- (50) [very thick, color=red];
\draw (12) -- (130) [very thick, color=red];
\draw (13) -- (83) [very thick, color=red];
\draw (13) -- (128) [very thick, color=red];
\draw (14) -- (69) [very thick, color=red];
\draw (14) -- (84) [very thick, color=red];
\draw (16) -- (117) [very thick, color=red];
\draw (16) -- (188) [very thick, color=red];
\draw (17) -- (76) [very thick, color=red];
\draw (17) -- (126) [very thick, color=red];
\draw (18) -- (72) [very thick, color=red];
\draw (18) -- (146) [very thick, color=red];
\draw (19) -- (70) [very thick, color=red];
\draw (19) -- (78) [very thick, color=red];
\draw (21) -- (121) [very thick, color=red];
\draw (21) -- (173) [very thick, color=red];
\draw (22) -- (141) [very thick, color=red];
\draw (22) -- (186) [very thick, color=red];
\draw (24) -- (107) [very thick, color=red];
\draw (24) -- (144) [very thick, color=red];
\draw (25) -- (56) [very thick, color=red];
\draw (25) -- (110) [very thick, color=red];
\draw (26) -- (60) [very thick, color=red];
\draw (26) -- (81) [very thick, color=red];
\draw (27) -- (58) [very thick, color=red];
\draw (27) -- (179) [very thick, color=red];
\draw (28) -- (64) [very thick, color=red];
\draw (28) -- (93) [very thick, color=red];
\draw (29) -- (45) [very thick, color=red];
\draw (29) -- (59) [very thick, color=red];
\draw (30) -- (125) [very thick, color=red];
\draw (30) -- (139) [very thick, color=red];
\draw (31) -- (131) [very thick, color=red];
\draw (31) -- (162) [very thick, color=red];
\draw (32) -- (106) [very thick, color=red];
\draw (32) -- (187) [very thick, color=red];
\draw (33) -- (105) [very thick, color=red];
\draw (33) -- (193) [very thick, color=red];
\draw (34) -- (65) [very thick, color=red];
\draw (34) -- (175) [very thick, color=red];
\draw (36) -- (122) [very thick, color=red];
\draw (36) -- (155) [very thick, color=red];
\draw (37) -- (88) [very thick, color=red];
\draw (37) -- (138) [very thick, color=red];
\draw (38) -- (52) [very thick, color=red];
\draw (38) -- (119) [very thick, color=red];
\draw (39) -- (109) [very thick, color=red];
\draw (40) -- (76) [very thick, color=red];
\draw (40) -- (130) [very thick, color=red];
\draw (41) -- (90) [very thick, color=red];
\draw (41) -- (143) [very thick, color=red];
\draw (42) -- (68) [very thick, color=red];
\draw (42) -- (113) [very thick, color=red];
\draw (43) -- (184) [very thick, color=red];
\draw (43) -- (199) [very thick, color=red];
\draw (44) -- (55) [very thick, color=red];
\draw (44) -- (106) [very thick, color=red];
\draw (45) -- (98) [very thick, color=red];
\draw (46) -- (175) [very thick, color=red];
\draw (46) -- (176) [very thick, color=red];
\draw (47) -- (155) [very thick, color=red];
\draw (47) -- (174) [very thick, color=red];
\draw (48) -- (59) [very thick, color=red];
\draw (48) -- (120) [very thick, color=red];
\draw (49) -- (74) [very thick, color=red];
\draw (49) -- (182) [very thick, color=red];
\draw (50) -- (169) [very thick, color=red];
\draw (51) -- (149) [very thick, color=red];
\draw (53) -- (195) [very thick, color=red];
\draw (53) -- (198) [very thick, color=red];
\draw (54) -- (152) [very thick, color=red];
\draw (56) -- (118) [very thick, color=red];
\draw (57) -- (109) [very thick, color=red];
\draw (57) -- (129) [very thick, color=red];
\draw (58) -- (111) [very thick, color=red];
\draw (60) -- (127) [very thick, color=red];
\draw (61) -- (152) [very thick, color=red];
\draw (62) -- (160) [very thick, color=red];
\draw (62) -- (185) [very thick, color=red];
\draw (63) -- (107) [very thick, color=red];
\draw (63) -- (181) [very thick, color=red];
\draw (64) -- (101) [very thick, color=red];
\draw (65) -- (167) [very thick, color=red];
\draw (66) -- (112) [very thick, color=red];
\draw (66) -- (194) [very thick, color=red];
\draw (67) -- (104) [very thick, color=red];
\draw (67) -- (159) [very thick, color=red];
\draw (68) -- (143) [very thick, color=red];
\draw (69) -- (108) [very thick, color=red];
\draw (70) -- (128) [very thick, color=red];
\draw (72) -- (118) [very thick, color=red];
\draw (73) -- (146) [very thick, color=red];
\draw (73) -- (147) [very thick, color=red];
\draw (74) -- (144) [very thick, color=red];
\draw (75) -- (110) [very thick, color=red];
\draw (75) -- (162) [very thick, color=red];
\draw (77) -- (119) [very thick, color=red];
\draw (78) -- (178) [very thick, color=red];
\draw (79) -- (153) [very thick, color=red];
\draw (80) -- (131) [very thick, color=red];
\draw (81) -- (168) [very thick, color=red];
\draw (82) -- (121) [very thick, color=red];
\draw (82) -- (174) [very thick, color=red];
\draw (83) -- (153) [very thick, color=red];
\draw (84) -- (134) [very thick, color=red];
\draw (86) -- (101) [very thick, color=red];
\draw (86) -- (112) [very thick, color=red];
\draw (87) -- (111) [very thick, color=red];
\draw (88) -- (103) [very thick, color=red];
\draw (90) -- (115) [very thick, color=red];
\draw (91) -- (113) [very thick, color=red];
\draw (91) -- (141) [very thick, color=red];
\draw (93) -- (186) [very thick, color=red];
\draw (96) -- (105) [very thick, color=red];
\draw (96) -- (188) [very thick, color=red];
\draw (97) -- (161) [very thick, color=red];
\draw (97) -- (187) [very thick, color=red];
\draw (98) -- (125) [very thick, color=red];
\draw (99) -- (167) [very thick, color=red];
\draw (99) -- (179) [very thick, color=red];
\draw (100) -- (149) [very thick, color=red];
\draw (100) -- (168) [very thick, color=red];
\draw (102) -- (178) [very thick, color=red];
\draw (104) -- (160) [very thick, color=red];
\draw (108) -- (180) [very thick, color=red];
\draw (115) -- (185) [very thick, color=red];
\draw (120) -- (172) [very thick, color=red];
\draw (122) -- (138) [very thick, color=red];
\draw (129) -- (169) [very thick, color=red];
\draw (132) -- (180) [very thick, color=red];
\draw (134) -- (170) [very thick, color=red];
\draw (136) -- (197) [very thick, color=red];
\draw (139) -- (157) [very thick, color=red];
\draw (147) -- (181) [very thick, color=red];
\draw (158) -- (192) [very thick, color=red];
\draw (158) -- (198) [very thick, color=red];
\draw (159) -- (182) [very thick, color=red];
\draw (171) -- (172) [very thick, color=red];
\draw (171) -- (173) [very thick, color=red];
\draw (184) -- (192) [very thick, color=red];
\draw (191) -- (196) [very thick, color=red];
\draw (193) -- (195) [very thick, color=red];
\draw (194) -- (196) [very thick, color=red];
\draw (197) -- (199) [very thick, color=red];
\draw (15) -- (98) [dashed, very thin, color=teal];
\draw (20) -- (47) [dashed, very thin, color=teal];
\draw (23) -- (16) [dashed, very thin, color=teal];
\draw (35) -- (51) [dashed, very thin, color=teal];
\draw (71) -- (57) [dashed, very thin, color=teal];
\draw (85) -- (38) [dashed, very thin, color=teal];
\draw (89) -- (42) [dashed, very thin, color=teal];
\draw (92) -- (110) [dashed, very thin, color=teal];
\draw (94) -- (174) [dashed, very thin, color=teal];
\draw (95) -- (127) [dashed, very thin, color=teal];
\draw (114) -- (93) [dashed, very thin, color=teal];
\draw (116) -- (192) [dashed, very thin, color=teal];
\draw (123) -- (29) [dashed, very thin, color=teal];
\draw (124) -- (47) [dashed, very thin, color=teal];
\draw (133) -- (70) [dashed, very thin, color=teal];
\draw (135) -- (147) [dashed, very thin, color=teal];
\draw (137) -- (143) [dashed, very thin, color=teal];
\draw (140) -- (174) [dashed, very thin, color=teal];
\draw (142) -- (22) [dashed, very thin, color=teal];
\draw (145) -- (181) [dashed, very thin, color=teal];
\draw (148) -- (110) [dashed, very thin, color=teal];
\draw (150) -- (80) [dashed, very thin, color=teal];
\draw (151) -- (0) [dashed, very thin, color=teal];
\draw (154) -- (98) [dashed, very thin, color=teal];
\draw (156) -- (86) [dashed, very thin, color=teal];
\draw (163) -- (31) [dashed, very thin, color=teal];
\draw (164) -- (84) [dashed, very thin, color=teal];
\draw (165) -- (131) [dashed, very thin, color=teal];
\draw (166) -- (122) [dashed, very thin, color=teal];
\draw (177) -- (70) [dashed, very thin, color=teal];
\draw (183) -- (33) [dashed, very thin, color=teal];
\draw (189) -- (187) [dashed, very thin, color=teal];
\draw (190) -- (1) [dashed, very thin, color=teal];

\end{tikzpicture}
\begin{center}An optimal tour with $L=582.33$.\end{center}
\end{adjustbox}
\caption{Comparison between optimal solutions of two tour instances generated from base instance p5, with $c_{v}=2$ and $r_{v}=16.46$ for all vertices $v\in V$.}
\label{final_fig_11}
\end{figure}
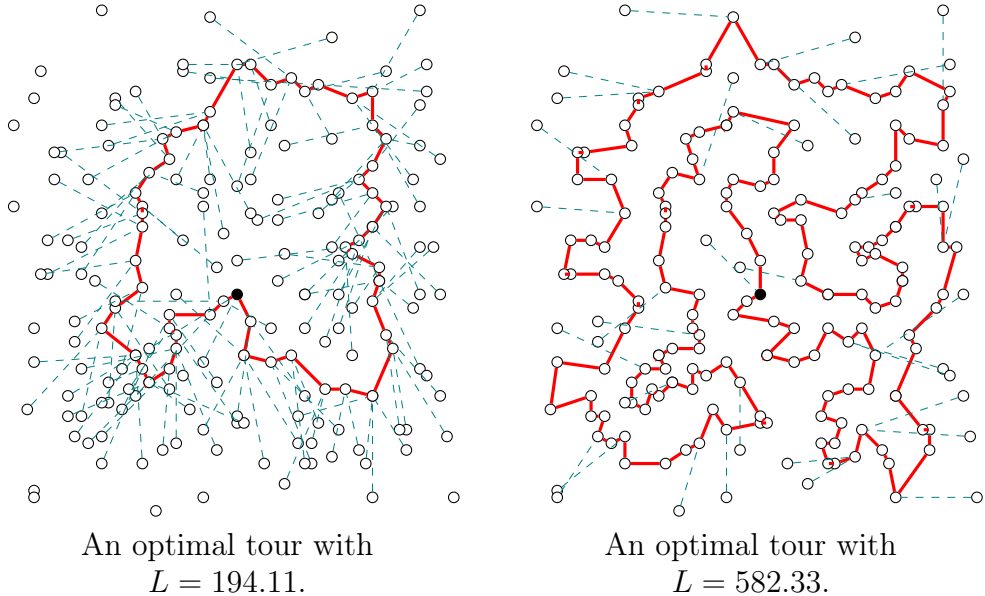
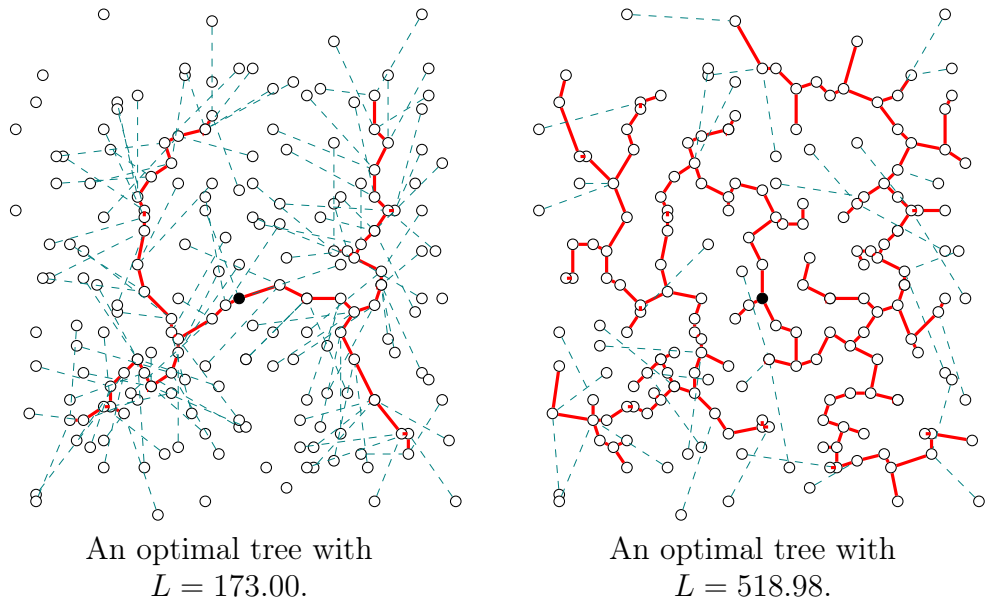
\begin{figure}[t]
\centering
\begin{adjustbox}{valign=t,minipage={.45\textwidth}}
\centering
\begin{tikzpicture}[scale=0.09]
\tikzstyle{every node}=[draw,circle,fill=white,minimum size=4pt,
                            inner sep=0pt]
\draw (35, 35) node (0) [overlay, color=black] {};
\draw (22, 22) node (1) [overlay] {};
\draw (36, 26) node (2) [overlay] {};
\draw (21, 45) node (3) [overlay] {};
\draw (45, 35) node (4) [overlay] {};
\draw (55, 20) node (5) [overlay] {};
\draw (33, 34) node (6) [overlay] {};
\draw (50, 50) node (7) [overlay] {};
\draw (55, 45) node (8) [overlay] {};
\draw (26, 59) node (9) [overlay] {};
\draw (40, 66) node (10) [overlay] {};
\draw (55, 65) node (11) [overlay] {};
\draw (35, 51) node (12) [overlay] {};
\draw (62, 35) node (13) [overlay] {};
\draw (62, 57) node (14) [overlay] {};
\draw (62, 24) node (15) [overlay] {};
\draw (21, 36) node (16) [overlay] {};
\draw (33, 44) node (17) [overlay] {};
\draw (9, 56) node (18) [overlay] {};
\draw (62, 48) node (19) [overlay] {};
\draw (66, 14) node (20) [overlay] {};
\draw (44, 13) node (21) [overlay] {};
\draw (26, 13) node (22) [overlay] {};
\draw (11, 28) node (23) [overlay] {};
\draw (7, 43) node (24) [overlay] {};
\draw (17, 64) node (25) [overlay] {};
\draw (41, 46) node (26) [overlay] {};
\draw (55, 34) node (27) [overlay] {};
\draw (35, 16) node (28) [overlay] {};
\draw (52, 26) node (29) [overlay] {};
\draw (43, 26) node (30) [overlay] {};
\draw (31, 76) node (31) [overlay] {};
\draw (22, 53) node (32) [overlay] {};
\draw (26, 29) node (33) [overlay] {};
\draw (50, 40) node (34) [overlay] {};
\draw (55, 50) node (35) [overlay] {};
\draw (54, 10) node (36) [overlay] {};
\draw (60, 15) node (37) [overlay] {};
\draw (47, 66) node (38) [overlay] {};
\draw (30, 60) node (39) [overlay] {};
\draw (30, 50) node (40) [overlay] {};
\draw (12, 17) node (41) [overlay] {};
\draw (15, 14) node (42) [overlay] {};
\draw (16, 19) node (43) [overlay] {};
\draw (21, 48) node (44) [overlay] {};
\draw (50, 30) node (45) [overlay] {};
\draw (51, 42) node (46) [overlay] {};
\draw (50, 15) node (47) [overlay] {};
\draw (48, 21) node (48) [overlay] {};
\draw (12, 38) node (49) [overlay] {};
\draw (37, 52) node (50) [overlay] {};
\draw (49, 49) node (51) [overlay] {};
\draw (52, 64) node (52) [overlay] {};
\draw (20, 26) node (53) [overlay] {};
\draw (40, 30) node (54) [overlay] {};
\draw (21, 47) node (55) [overlay] {};
\draw (17, 63) node (56) [overlay] {};
\draw (31, 62) node (57) [overlay] {};
\draw (52, 33) node (58) [overlay] {};
\draw (51, 21) node (59) [overlay] {};
\draw (42, 41) node (60) [overlay] {};
\draw (31, 32) node (61) [overlay] {};
\draw (5, 25) node (62) [overlay] {};
\draw (12, 42) node (63) [overlay] {};
\draw (36, 16) node (64) [overlay] {};
\draw (52, 41) node (65) [overlay] {};
\draw (27, 23) node (66) [overlay] {};
\draw (17, 33) node (67) [overlay] {};
\draw (13, 13) node (68) [overlay] {};
\draw (57, 58) node (69) [overlay] {};
\draw (62, 42) node (70) [overlay] {};
\draw (42, 57) node (71) [overlay] {};
\draw (16, 57) node (72) [overlay] {};
\draw (8, 52) node (73) [overlay] {};
\draw (7, 38) node (74) [overlay] {};
\draw (27, 68) node (75) [overlay] {};
\draw (30, 48) node (76) [overlay] {};
\draw (43, 67) node (77) [overlay] {};
\draw (58, 48) node (78) [overlay] {};
\draw (58, 27) node (79) [overlay] {};
\draw (37, 69) node (80) [overlay] {};
\draw (38, 46) node (81) [overlay] {};
\draw (46, 10) node (82) [overlay] {};
\draw (61, 33) node (83) [overlay] {};
\draw (62, 63) node (84) [overlay] {};
\draw (63, 69) node (85) [overlay] {};
\draw (32, 22) node (86) [overlay] {};
\draw (45, 35) node (87) [overlay] {};
\draw (59, 15) node (88) [overlay] {};
\draw (5, 6) node (89) [overlay] {};
\draw (10, 17) node (90) [overlay] {};
\draw (21, 10) node (91) [overlay] {};
\draw (5, 64) node (92) [overlay] {};
\draw (30, 15) node (93) [overlay] {};
\draw (39, 10) node (94) [overlay] {};
\draw (32, 39) node (95) [overlay] {};
\draw (25, 32) node (96) [overlay] {};
\draw (25, 55) node (97) [overlay] {};
\draw (48, 28) node (98) [overlay] {};
\draw (56, 37) node (99) [overlay] {};
\draw (41, 49) node (100) [overlay] {};
\draw (35, 17) node (101) [overlay] {};
\draw (55, 45) node (102) [overlay] {};
\draw (55, 20) node (103) [overlay] {};
\draw (15, 30) node (104) [overlay] {};
\draw (25, 30) node (105) [overlay] {};
\draw (20, 50) node (106) [overlay] {};
\draw (10, 43) node (107) [overlay] {};
\draw (55, 60) node (108) [overlay] {};
\draw (30, 60) node (109) [overlay] {};
\draw (20, 65) node (110) [overlay] {};
\draw (50, 35) node (111) [overlay] {};
\draw (30, 25) node (112) [overlay] {};
\draw (15, 10) node (113) [overlay] {};
\draw (30, 5) node (114) [overlay] {};
\draw (10, 20) node (115) [overlay] {};
\draw (5, 30) node (116) [overlay] {};
\draw (20, 40) node (117) [overlay] {};
\draw (15, 60) node (118) [overlay] {};
\draw (45, 65) node (119) [overlay] {};
\draw (45, 20) node (120) [overlay] {};
\draw (45, 10) node (121) [overlay] {};
\draw (55, 5) node (122) [overlay] {};
\draw (65, 35) node (123) [overlay] {};
\draw (65, 20) node (124) [overlay] {};
\draw (45, 30) node (125) [overlay] {};
\draw (35, 40) node (126) [overlay] {};
\draw (41, 37) node (127) [overlay] {};
\draw (64, 42) node (128) [overlay] {};
\draw (40, 60) node (129) [overlay] {};
\draw (31, 52) node (130) [overlay] {};
\draw (35, 69) node (131) [overlay] {};
\draw (53, 52) node (132) [overlay] {};
\draw (65, 55) node (133) [overlay] {};
\draw (63, 65) node (134) [overlay] {};
\draw (2, 60) node (135) [overlay] {};
\draw (20, 20) node (136) [overlay] {};
\draw (5, 5) node (137) [overlay] {};
\draw (60, 12) node (138) [overlay] {};
\draw (40, 25) node (139) [overlay] {};
\draw (42, 7) node (140) [overlay] {};
\draw (24, 12) node (141) [overlay] {};
\draw (23, 3) node (142) [overlay] {};
\draw (11, 14) node (143) [overlay] {};
\draw (6, 38) node (144) [overlay] {};
\draw (2, 48) node (145) [overlay] {};
\draw (8, 56) node (146) [overlay] {};
\draw (13, 52) node (147) [overlay] {};
\draw (6, 68) node (148) [overlay] {};
\draw (47, 47) node (149) [overlay] {};
\draw (49, 58) node (150) [overlay] {};
\draw (27, 43) node (151) [overlay] {};
\draw (37, 31) node (152) [overlay] {};
\draw (57, 29) node (153) [overlay] {};
\draw (63, 23) node (154) [overlay] {};
\draw (53, 12) node (155) [overlay] {};
\draw (32, 12) node (156) [overlay] {};
\draw (36, 26) node (157) [overlay] {};
\draw (21, 24) node (158) [overlay] {};
\draw (17, 34) node (159) [overlay] {};
\draw (12, 24) node (160) [overlay] {};
\draw (24, 58) node (161) [overlay] {};
\draw (27, 69) node (162) [overlay] {};
\draw (15, 77) node (163) [overlay] {};
\draw (62, 77) node (164) [overlay] {};
\draw (49, 73) node (165) [overlay] {};
\draw (67, 5) node (166) [overlay] {};
\draw (56, 39) node (167) [overlay] {};
\draw (37, 47) node (168) [overlay] {};
\draw (37, 56) node (169) [overlay] {};
\draw (57, 68) node (170) [overlay] {};
\draw (47, 16) node (171) [overlay] {};
\draw (44, 17) node (172) [overlay] {};
\draw (46, 13) node (173) [overlay] {};
\draw (49, 11) node (174) [overlay] {};
\draw (49, 42) node (175) [overlay] {};
\draw (53, 43) node (176) [overlay] {};
\draw (61, 52) node (177) [overlay] {};
\draw (57, 48) node (178) [overlay] {};
\draw (56, 37) node (179) [overlay] {};
\draw (55, 54) node (180) [overlay] {};
\draw (15, 47) node (181) [overlay] {};
\draw (14, 37) node (182) [overlay] {};
\draw (11, 31) node (183) [overlay] {};
\draw (16, 22) node (184) [overlay] {};
\draw (4, 18) node (185) [overlay] {};
\draw (28, 18) node (186) [overlay] {};
\draw (26, 52) node (187) [overlay] {};
\draw (26, 35) node (188) [overlay] {};
\draw (31, 67) node (189) [overlay] {};
\draw (15, 19) node (190) [overlay] {};
\draw (22, 22) node (191) [overlay] {};
\draw (18, 24) node (192) [overlay] {};
\draw (26, 27) node (193) [overlay] {};
\draw (25, 24) node (194) [overlay] {};
\draw (22, 27) node (195) [overlay] {};
\draw (25, 21) node (196) [overlay] {};
\draw (19, 21) node (197) [overlay] {};
\draw (20, 26) node (198) [overlay] {};
\draw (18, 18) node (199) [overlay] {};
\draw (0) -- (6) [very thick, color=red];
\draw (0) -- (127) [very thick, color=red];
\draw (1) -- (191) [very thick, color=red];
\draw (3) -- (55) [very thick, color=red];
\draw (3) -- (117) [very thick, color=red];
\draw (4) -- (87) [very thick, color=red];
\draw (5) -- (103) [very thick, color=red];
\draw (6) -- (61) [very thick, color=red];
\draw (8) -- (102) [very thick, color=red];
\draw (9) -- (109) [very thick, color=red];
\draw (9) -- (161) [very thick, color=red];
\draw (11) -- (108) [very thick, color=red];
\draw (16) -- (96) [very thick, color=red];
\draw (16) -- (117) [very thick, color=red];
\draw (27) -- (58) [very thick, color=red];
\draw (27) -- (179) [very thick, color=red];
\draw (29) -- (45) [very thick, color=red];
\draw (29) -- (103) [very thick, color=red];
\draw (32) -- (97) [very thick, color=red];
\draw (32) -- (106) [very thick, color=red];
\draw (33) -- (61) [very thick, color=red];
\draw (33) -- (105) [very thick, color=red];
\draw (33) -- (193) [very thick, color=red];
\draw (35) -- (178) [very thick, color=red];
\draw (35) -- (180) [very thick, color=red];
\draw (37) -- (88) [very thick, color=red];
\draw (37) -- (138) [very thick, color=red];
\draw (39) -- (109) [very thick, color=red];
\draw (41) -- (90) [very thick, color=red];
\draw (41) -- (190) [very thick, color=red];
\draw (43) -- (184) [very thick, color=red];
\draw (43) -- (190) [very thick, color=red];
\draw (43) -- (199) [very thick, color=red];
\draw (44) -- (55) [very thick, color=red];
\draw (44) -- (106) [very thick, color=red];
\draw (45) -- (58) [very thick, color=red];
\draw (46) -- (65) [very thick, color=red];
\draw (46) -- (175) [very thick, color=red];
\draw (53) -- (198) [very thick, color=red];
\draw (57) -- (109) [very thick, color=red];
\draw (58) -- (111) [very thick, color=red];
\draw (65) -- (167) [very thick, color=red];
\draw (65) -- (176) [very thick, color=red];
\draw (69) -- (108) [very thick, color=red];
\draw (69) -- (180) [very thick, color=red];
\draw (78) -- (178) [very thick, color=red];
\draw (87) -- (111) [very thick, color=red];
\draw (87) -- (127) [very thick, color=red];
\draw (88) -- (103) [very thick, color=red];
\draw (96) -- (105) [very thick, color=red];
\draw (97) -- (161) [very thick, color=red];
\draw (99) -- (179) [very thick, color=red];
\draw (102) -- (176) [very thick, color=red];
\draw (102) -- (178) [very thick, color=red];
\draw (158) -- (191) [very thick, color=red];
\draw (158) -- (198) [very thick, color=red];
\draw (167) -- (179) [very thick, color=red];
\draw (184) -- (192) [very thick, color=red];
\draw (191) -- (194) [very thick, color=red];
\draw (192) -- (198) [very thick, color=red];
\draw (193) -- (194) [very thick, color=red];
\draw (2) -- (87) [dashed, very thin, color=teal];
\draw (7) -- (102) [dashed, very thin, color=teal];
\draw (10) -- (39) [dashed, very thin, color=teal];
\draw (12) -- (57) [dashed, very thin, color=teal];
\draw (13) -- (65) [dashed, very thin, color=teal];
\draw (14) -- (78) [dashed, very thin, color=teal];
\draw (15) -- (99) [dashed, very thin, color=teal];
\draw (17) -- (105) [dashed, very thin, color=teal];
\draw (18) -- (44) [dashed, very thin, color=teal];
\draw (19) -- (167) [dashed, very thin, color=teal];
\draw (20) -- (5) [dashed, very thin, color=teal];
\draw (21) -- (88) [dashed, very thin, color=teal];
\draw (22) -- (33) [dashed, very thin, color=teal];
\draw (23) -- (194) [dashed, very thin, color=teal];
\draw (24) -- (55) [dashed, very thin, color=teal];
\draw (25) -- (106) [dashed, very thin, color=teal];
\draw (26) -- (6) [dashed, very thin, color=teal];
\draw (28) -- (158) [dashed, very thin, color=teal];
\draw (30) -- (27) [dashed, very thin, color=teal];
\draw (31) -- (57) [dashed, very thin, color=teal];
\draw (34) -- (0) [dashed, very thin, color=teal];
\draw (36) -- (29) [dashed, very thin, color=teal];
\draw (38) -- (11) [dashed, very thin, color=teal];
\draw (40) -- (6) [dashed, very thin, color=teal];
\draw (42) -- (199) [dashed, very thin, color=teal];
\draw (47) -- (45) [dashed, very thin, color=teal];
\draw (48) -- (27) [dashed, very thin, color=teal];
\draw (49) -- (198) [dashed, very thin, color=teal];
\draw (50) -- (175) [dashed, very thin, color=teal];
\draw (51) -- (46) [dashed, very thin, color=teal];
\draw (52) -- (35) [dashed, very thin, color=teal];
\draw (54) -- (4) [dashed, very thin, color=teal];
\draw (56) -- (44) [dashed, very thin, color=teal];
\draw (59) -- (4) [dashed, very thin, color=teal];
\draw (60) -- (46) [dashed, very thin, color=teal];
\draw (62) -- (184) [dashed, very thin, color=teal];
\draw (63) -- (96) [dashed, very thin, color=teal];
\draw (64) -- (191) [dashed, very thin, color=teal];
\draw (67) -- (96) [dashed, very thin, color=teal];
\draw (68) -- (184) [dashed, very thin, color=teal];
\draw (70) -- (179) [dashed, very thin, color=teal];
\draw (71) -- (35) [dashed, very thin, color=teal];
\draw (72) -- (32) [dashed, very thin, color=teal];
\draw (73) -- (32) [dashed, very thin, color=teal];
\draw (74) -- (3) [dashed, very thin, color=teal];
\draw (75) -- (97) [dashed, very thin, color=teal];
\draw (76) -- (0) [dashed, very thin, color=teal];
\draw (77) -- (109) [dashed, very thin, color=teal];
\draw (79) -- (99) [dashed, very thin, color=teal];
\draw (80) -- (39) [dashed, very thin, color=teal];
\draw (81) -- (176) [dashed, very thin, color=teal];
\draw (82) -- (103) [dashed, very thin, color=teal];
\draw (83) -- (8) [dashed, very thin, color=teal];
\draw (84) -- (78) [dashed, very thin, color=teal];
\draw (85) -- (69) [dashed, very thin, color=teal];
\draw (86) -- (198) [dashed, very thin, color=teal];
\draw (89) -- (190) [dashed, very thin, color=teal];
\draw (91) -- (53) [dashed, very thin, color=teal];
\draw (93) -- (1) [dashed, very thin, color=teal];
\draw (95) -- (193) [dashed, very thin, color=teal];
\draw (98) -- (111) [dashed, very thin, color=teal];
\draw (100) -- (178) [dashed, very thin, color=teal];
\draw (101) -- (33) [dashed, very thin, color=teal];
\draw (104) -- (43) [dashed, very thin, color=teal];
\draw (107) -- (16) [dashed, very thin, color=teal];
\draw (110) -- (106) [dashed, very thin, color=teal];
\draw (112) -- (127) [dashed, very thin, color=teal];
\draw (113) -- (192) [dashed, very thin, color=teal];
\draw (115) -- (158) [dashed, very thin, color=teal];
\draw (118) -- (55) [dashed, very thin, color=teal];
\draw (119) -- (180) [dashed, very thin, color=teal];
\draw (120) -- (111) [dashed, very thin, color=teal];
\draw (121) -- (37) [dashed, very thin, color=teal];
\draw (122) -- (88) [dashed, very thin, color=teal];
\draw (124) -- (138) [dashed, very thin, color=teal];
\draw (125) -- (167) [dashed, very thin, color=teal];
\draw (126) -- (193) [dashed, very thin, color=teal];
\draw (128) -- (178) [dashed, very thin, color=teal];
\draw (129) -- (180) [dashed, very thin, color=teal];
\draw (130) -- (117) [dashed, very thin, color=teal];
\draw (131) -- (161) [dashed, very thin, color=teal];
\draw (132) -- (176) [dashed, very thin, color=teal];
\draw (133) -- (8) [dashed, very thin, color=teal];
\draw (134) -- (108) [dashed, very thin, color=teal];
\draw (136) -- (41) [dashed, very thin, color=teal];
\draw (137) -- (90) [dashed, very thin, color=teal];
\draw (139) -- (58) [dashed, very thin, color=teal];
\draw (141) -- (190) [dashed, very thin, color=teal];
\draw (142) -- (199) [dashed, very thin, color=teal];
\draw (143) -- (192) [dashed, very thin, color=teal];
\draw (144) -- (16) [dashed, very thin, color=teal];
\draw (146) -- (161) [dashed, very thin, color=teal];
\draw (147) -- (97) [dashed, very thin, color=teal];
\draw (149) -- (102) [dashed, very thin, color=teal];
\draw (150) -- (175) [dashed, very thin, color=teal];
\draw (151) -- (127) [dashed, very thin, color=teal];
\draw (152) -- (58) [dashed, very thin, color=teal];
\draw (153) -- (179) [dashed, very thin, color=teal];
\draw (155) -- (138) [dashed, very thin, color=teal];
\draw (156) -- (194) [dashed, very thin, color=teal];
\draw (157) -- (87) [dashed, very thin, color=teal];
\draw (159) -- (3) [dashed, very thin, color=teal];
\draw (160) -- (1) [dashed, very thin, color=teal];
\draw (162) -- (109) [dashed, very thin, color=teal];
\draw (164) -- (11) [dashed, very thin, color=teal];
\draw (165) -- (108) [dashed, very thin, color=teal];
\draw (166) -- (37) [dashed, very thin, color=teal];
\draw (168) -- (65) [dashed, very thin, color=teal];
\draw (169) -- (9) [dashed, very thin, color=teal];
\draw (170) -- (69) [dashed, very thin, color=teal];
\draw (171) -- (45) [dashed, very thin, color=teal];
\draw (172) -- (5) [dashed, very thin, color=teal];
\draw (173) -- (103) [dashed, very thin, color=teal];
\draw (174) -- (29) [dashed, very thin, color=teal];
\draw (181) -- (9) [dashed, very thin, color=teal];
\draw (182) -- (53) [dashed, very thin, color=teal];
\draw (183) -- (90) [dashed, very thin, color=teal];
\draw (185) -- (41) [dashed, very thin, color=teal];
\draw (186) -- (105) [dashed, very thin, color=teal];
\draw (187) -- (117) [dashed, very thin, color=teal];
\draw (188) -- (61) [dashed, very thin, color=teal];
\draw (195) -- (191) [dashed, very thin, color=teal];
\draw (196) -- (43) [dashed, very thin, color=teal];
\draw (197) -- (61) [dashed, very thin, color=teal];
\end{tikzpicture}
\begin{center}An optimal tree with $L=173.00$.\end{center}
\end{adjustbox}
\hfill
\begin{adjustbox}{valign=t,minipage={.45\textwidth}}
\centering
\begin{tikzpicture}[scale=0.09]
\tikzstyle{every node}=[draw,circle,fill=white,minimum size=4pt,
                            inner sep=0pt]
\draw (35, 35) node (0) [overlay, color=black] {};
\draw (22, 22) node (1) [overlay] {};
\draw (36, 26) node (2) [overlay] {};
\draw (21, 45) node (3) [overlay] {};
\draw (45, 35) node (4) [overlay] {};
\draw (55, 20) node (5) [overlay] {};
\draw (33, 34) node (6) [overlay] {};
\draw (50, 50) node (7) [overlay] {};
\draw (55, 45) node (8) [overlay] {};
\draw (26, 59) node (9) [overlay] {};
\draw (40, 66) node (10) [overlay] {};
\draw (55, 65) node (11) [overlay] {};
\draw (35, 51) node (12) [overlay] {};
\draw (62, 35) node (13) [overlay] {};
\draw (62, 57) node (14) [overlay] {};
\draw (62, 24) node (15) [overlay] {};
\draw (21, 36) node (16) [overlay] {};
\draw (33, 44) node (17) [overlay] {};
\draw (9, 56) node (18) [overlay] {};
\draw (62, 48) node (19) [overlay] {};
\draw (66, 14) node (20) [overlay] {};
\draw (44, 13) node (21) [overlay] {};
\draw (26, 13) node (22) [overlay] {};
\draw (11, 28) node (23) [overlay] {};
\draw (7, 43) node (24) [overlay] {};
\draw (17, 64) node (25) [overlay] {};
\draw (41, 46) node (26) [overlay] {};
\draw (55, 34) node (27) [overlay] {};
\draw (35, 16) node (28) [overlay] {};
\draw (52, 26) node (29) [overlay] {};
\draw (43, 26) node (30) [overlay] {};
\draw (31, 76) node (31) [overlay] {};
\draw (22, 53) node (32) [overlay] {};
\draw (26, 29) node (33) [overlay] {};
\draw (50, 40) node (34) [overlay] {};
\draw (55, 50) node (35) [overlay] {};
\draw (54, 10) node (36) [overlay] {};
\draw (60, 15) node (37) [overlay] {};
\draw (47, 66) node (38) [overlay] {};
\draw (30, 60) node (39) [overlay] {};
\draw (30, 50) node (40) [overlay] {};
\draw (12, 17) node (41) [overlay] {};
\draw (15, 14) node (42) [overlay] {};
\draw (16, 19) node (43) [overlay] {};
\draw (21, 48) node (44) [overlay] {};
\draw (50, 30) node (45) [overlay] {};
\draw (51, 42) node (46) [overlay] {};
\draw (50, 15) node (47) [overlay] {};
\draw (48, 21) node (48) [overlay] {};
\draw (12, 38) node (49) [overlay] {};
\draw (37, 52) node (50) [overlay] {};
\draw (49, 49) node (51) [overlay] {};
\draw (52, 64) node (52) [overlay] {};
\draw (20, 26) node (53) [overlay] {};
\draw (40, 30) node (54) [overlay] {};
\draw (21, 47) node (55) [overlay] {};
\draw (17, 63) node (56) [overlay] {};
\draw (31, 62) node (57) [overlay] {};
\draw (52, 33) node (58) [overlay] {};
\draw (51, 21) node (59) [overlay] {};
\draw (42, 41) node (60) [overlay] {};
\draw (31, 32) node (61) [overlay] {};
\draw (5, 25) node (62) [overlay] {};
\draw (12, 42) node (63) [overlay] {};
\draw (36, 16) node (64) [overlay] {};
\draw (52, 41) node (65) [overlay] {};
\draw (27, 23) node (66) [overlay] {};
\draw (17, 33) node (67) [overlay] {};
\draw (13, 13) node (68) [overlay] {};
\draw (57, 58) node (69) [overlay] {};
\draw (62, 42) node (70) [overlay] {};
\draw (42, 57) node (71) [overlay] {};
\draw (16, 57) node (72) [overlay] {};
\draw (8, 52) node (73) [overlay] {};
\draw (7, 38) node (74) [overlay] {};
\draw (27, 68) node (75) [overlay] {};
\draw (30, 48) node (76) [overlay] {};
\draw (43, 67) node (77) [overlay] {};
\draw (58, 48) node (78) [overlay] {};
\draw (58, 27) node (79) [overlay] {};
\draw (37, 69) node (80) [overlay] {};
\draw (38, 46) node (81) [overlay] {};
\draw (46, 10) node (82) [overlay] {};
\draw (61, 33) node (83) [overlay] {};
\draw (62, 63) node (84) [overlay] {};
\draw (63, 69) node (85) [overlay] {};
\draw (32, 22) node (86) [overlay] {};
\draw (45, 35) node (87) [overlay] {};
\draw (59, 15) node (88) [overlay] {};
\draw (5, 6) node (89) [overlay] {};
\draw (10, 17) node (90) [overlay] {};
\draw (21, 10) node (91) [overlay] {};
\draw (5, 64) node (92) [overlay] {};
\draw (30, 15) node (93) [overlay] {};
\draw (39, 10) node (94) [overlay] {};
\draw (32, 39) node (95) [overlay] {};
\draw (25, 32) node (96) [overlay] {};
\draw (25, 55) node (97) [overlay] {};
\draw (48, 28) node (98) [overlay] {};
\draw (56, 37) node (99) [overlay] {};
\draw (41, 49) node (100) [overlay] {};
\draw (35, 17) node (101) [overlay] {};
\draw (55, 45) node (102) [overlay] {};
\draw (55, 20) node (103) [overlay] {};
\draw (15, 30) node (104) [overlay] {};
\draw (25, 30) node (105) [overlay] {};
\draw (20, 50) node (106) [overlay] {};
\draw (10, 43) node (107) [overlay] {};
\draw (55, 60) node (108) [overlay] {};
\draw (30, 60) node (109) [overlay] {};
\draw (20, 65) node (110) [overlay] {};
\draw (50, 35) node (111) [overlay] {};
\draw (30, 25) node (112) [overlay] {};
\draw (15, 10) node (113) [overlay] {};
\draw (30, 5) node (114) [overlay] {};
\draw (10, 20) node (115) [overlay] {};
\draw (5, 30) node (116) [overlay] {};
\draw (20, 40) node (117) [overlay] {};
\draw (15, 60) node (118) [overlay] {};
\draw (45, 65) node (119) [overlay] {};
\draw (45, 20) node (120) [overlay] {};
\draw (45, 10) node (121) [overlay] {};
\draw (55, 5) node (122) [overlay] {};
\draw (65, 35) node (123) [overlay] {};
\draw (65, 20) node (124) [overlay] {};
\draw (45, 30) node (125) [overlay] {};
\draw (35, 40) node (126) [overlay] {};
\draw (41, 37) node (127) [overlay] {};
\draw (64, 42) node (128) [overlay] {};
\draw (40, 60) node (129) [overlay] {};
\draw (31, 52) node (130) [overlay] {};
\draw (35, 69) node (131) [overlay] {};
\draw (53, 52) node (132) [overlay] {};
\draw (65, 55) node (133) [overlay] {};
\draw (63, 65) node (134) [overlay] {};
\draw (2, 60) node (135) [overlay] {};
\draw (20, 20) node (136) [overlay] {};
\draw (5, 5) node (137) [overlay] {};
\draw (60, 12) node (138) [overlay] {};
\draw (40, 25) node (139) [overlay] {};
\draw (42, 7) node (140) [overlay] {};
\draw (24, 12) node (141) [overlay] {};
\draw (23, 3) node (142) [overlay] {};
\draw (11, 14) node (143) [overlay] {};
\draw (6, 38) node (144) [overlay] {};
\draw (2, 48) node (145) [overlay] {};
\draw (8, 56) node (146) [overlay] {};
\draw (13, 52) node (147) [overlay] {};
\draw (6, 68) node (148) [overlay] {};
\draw (47, 47) node (149) [overlay] {};
\draw (49, 58) node (150) [overlay] {};
\draw (27, 43) node (151) [overlay] {};
\draw (37, 31) node (152) [overlay] {};
\draw (57, 29) node (153) [overlay] {};
\draw (63, 23) node (154) [overlay] {};
\draw (53, 12) node (155) [overlay] {};
\draw (32, 12) node (156) [overlay] {};
\draw (36, 26) node (157) [overlay] {};
\draw (21, 24) node (158) [overlay] {};
\draw (17, 34) node (159) [overlay] {};
\draw (12, 24) node (160) [overlay] {};
\draw (24, 58) node (161) [overlay] {};
\draw (27, 69) node (162) [overlay] {};
\draw (15, 77) node (163) [overlay] {};
\draw (62, 77) node (164) [overlay] {};
\draw (49, 73) node (165) [overlay] {};
\draw (67, 5) node (166) [overlay] {};
\draw (56, 39) node (167) [overlay] {};
\draw (37, 47) node (168) [overlay] {};
\draw (37, 56) node (169) [overlay] {};
\draw (57, 68) node (170) [overlay] {};
\draw (47, 16) node (171) [overlay] {};
\draw (44, 17) node (172) [overlay] {};
\draw (46, 13) node (173) [overlay] {};
\draw (49, 11) node (174) [overlay] {};
\draw (49, 42) node (175) [overlay] {};
\draw (53, 43) node (176) [overlay] {};
\draw (61, 52) node (177) [overlay] {};
\draw (57, 48) node (178) [overlay] {};
\draw (56, 37) node (179) [overlay] {};
\draw (55, 54) node (180) [overlay] {};
\draw (15, 47) node (181) [overlay] {};
\draw (14, 37) node (182) [overlay] {};
\draw (11, 31) node (183) [overlay] {};
\draw (16, 22) node (184) [overlay] {};
\draw (4, 18) node (185) [overlay] {};
\draw (28, 18) node (186) [overlay] {};
\draw (26, 52) node (187) [overlay] {};
\draw (26, 35) node (188) [overlay] {};
\draw (31, 67) node (189) [overlay] {};
\draw (15, 19) node (190) [overlay] {};
\draw (22, 22) node (191) [overlay] {};
\draw (18, 24) node (192) [overlay] {};
\draw (26, 27) node (193) [overlay] {};
\draw (25, 24) node (194) [overlay] {};
\draw (22, 27) node (195) [overlay] {};
\draw (25, 21) node (196) [overlay] {};
\draw (19, 21) node (197) [overlay] {};
\draw (20, 26) node (198) [overlay] {};
\draw (18, 18) node (199) [overlay] {};
\draw (0) -- (6) [very thick, color=red];
\draw (0) -- (126) [very thick, color=red];
\draw (0) -- (152) [very thick, color=red];
\draw (1) -- (191) [very thick, color=red];
\draw (2) -- (157) [very thick, color=red];
\draw (3) -- (55) [very thick, color=red];
\draw (3) -- (117) [very thick, color=red];
\draw (4) -- (87) [very thick, color=red];
\draw (5) -- (103) [very thick, color=red];
\draw (6) -- (61) [very thick, color=red];
\draw (7) -- (51) [very thick, color=red];
\draw (7) -- (132) [very thick, color=red];
\draw (8) -- (102) [very thick, color=red];
\draw (9) -- (109) [very thick, color=red];
\draw (9) -- (161) [very thick, color=red];
\draw (10) -- (77) [very thick, color=red];
\draw (10) -- (80) [very thick, color=red];
\draw (10) -- (129) [very thick, color=red];
\draw (11) -- (52) [very thick, color=red];
\draw (11) -- (170) [very thick, color=red];
\draw (12) -- (130) [very thick, color=red];
\draw (12) -- (168) [very thick, color=red];
\draw (13) -- (83) [very thick, color=red];
\draw (14) -- (69) [very thick, color=red];
\draw (14) -- (84) [very thick, color=red];
\draw (14) -- (133) [very thick, color=red];
\draw (16) -- (117) [very thick, color=red];
\draw (16) -- (159) [very thick, color=red];
\draw (16) -- (188) [very thick, color=red];
\draw (17) -- (126) [very thick, color=red];
\draw (17) -- (168) [very thick, color=red];
\draw (18) -- (146) [very thick, color=red];
\draw (18) -- (147) [very thick, color=red];
\draw (19) -- (78) [very thick, color=red];
\draw (20) -- (37) [very thick, color=red];
\draw (21) -- (173) [very thick, color=red];
\draw (24) -- (74) [very thick, color=red];
\draw (24) -- (107) [very thick, color=red];
\draw (25) -- (56) [very thick, color=red];
\draw (25) -- (110) [very thick, color=red];
\draw (26) -- (81) [very thick, color=red];
\draw (26) -- (100) [very thick, color=red];
\draw (27) -- (58) [very thick, color=red];
\draw (27) -- (153) [very thick, color=red];
\draw (27) -- (179) [very thick, color=red];
\draw (28) -- (64) [very thick, color=red];
\draw (28) -- (93) [very thick, color=red];
\draw (28) -- (101) [very thick, color=red];
\draw (29) -- (59) [very thick, color=red];
\draw (29) -- (98) [very thick, color=red];
\draw (30) -- (125) [very thick, color=red];
\draw (30) -- (139) [very thick, color=red];
\draw (31) -- (131) [very thick, color=red];
\draw (32) -- (97) [very thick, color=red];
\draw (32) -- (106) [very thick, color=red];
\draw (33) -- (105) [very thick, color=red];
\draw (33) -- (193) [very thick, color=red];
\draw (34) -- (175) [very thick, color=red];
\draw (35) -- (132) [very thick, color=red];
\draw (35) -- (178) [very thick, color=red];
\draw (36) -- (122) [very thick, color=red];
\draw (36) -- (138) [very thick, color=red];
\draw (36) -- (155) [very thick, color=red];
\draw (37) -- (88) [very thick, color=red];
\draw (37) -- (138) [very thick, color=red];
\draw (38) -- (52) [very thick, color=red];
\draw (38) -- (119) [very thick, color=red];
\draw (38) -- (165) [very thick, color=red];
\draw (39) -- (109) [very thick, color=red];
\draw (40) -- (76) [very thick, color=red];
\draw (40) -- (130) [very thick, color=red];
\draw (40) -- (187) [very thick, color=red];
\draw (41) -- (90) [very thick, color=red];
\draw (41) -- (190) [very thick, color=red];
\draw (42) -- (68) [very thick, color=red];
\draw (43) -- (190) [very thick, color=red];
\draw (43) -- (199) [very thick, color=red];
\draw (44) -- (55) [very thick, color=red];
\draw (44) -- (106) [very thick, color=red];
\draw (45) -- (58) [very thick, color=red];
\draw (45) -- (98) [very thick, color=red];
\draw (46) -- (65) [very thick, color=red];
\draw (46) -- (175) [very thick, color=red];
\draw (47) -- (171) [very thick, color=red];
\draw (48) -- (59) [very thick, color=red];
\draw (48) -- (120) [very thick, color=red];
\draw (49) -- (63) [very thick, color=red];
\draw (49) -- (182) [very thick, color=red];
\draw (51) -- (149) [very thick, color=red];
\draw (52) -- (108) [very thick, color=red];
\draw (53) -- (198) [very thick, color=red];
\draw (54) -- (139) [very thick, color=red];
\draw (54) -- (152) [very thick, color=red];
\draw (56) -- (118) [very thick, color=red];
\draw (57) -- (109) [very thick, color=red];
\draw (58) -- (111) [very thick, color=red];
\draw (59) -- (103) [very thick, color=red];
\draw (60) -- (127) [very thick, color=red];
\draw (62) -- (185) [very thick, color=red];
\draw (63) -- (107) [very thick, color=red];
\draw (63) -- (181) [very thick, color=red];
\draw (65) -- (167) [very thick, color=red];
\draw (65) -- (176) [very thick, color=red];
\draw (67) -- (104) [very thick, color=red];
\draw (67) -- (159) [very thick, color=red];
\draw (68) -- (113) [very thick, color=red];
\draw (68) -- (143) [very thick, color=red];
\draw (69) -- (108) [very thick, color=red];
\draw (69) -- (180) [very thick, color=red];
\draw (72) -- (118) [very thick, color=red];
\draw (72) -- (147) [very thick, color=red];
\draw (74) -- (144) [very thick, color=red];
\draw (77) -- (119) [very thick, color=red];
\draw (78) -- (178) [very thick, color=red];
\draw (79) -- (153) [very thick, color=red];
\draw (80) -- (131) [very thick, color=red];
\draw (81) -- (168) [very thick, color=red];
\draw (82) -- (121) [very thick, color=red];
\draw (82) -- (173) [very thick, color=red];
\draw (82) -- (174) [very thick, color=red];
\draw (83) -- (153) [very thick, color=red];
\draw (84) -- (134) [very thick, color=red];
\draw (87) -- (111) [very thick, color=red];
\draw (87) -- (127) [very thick, color=red];
\draw (90) -- (115) [very thick, color=red];
\draw (90) -- (143) [very thick, color=red];
\draw (90) -- (185) [very thick, color=red];
\draw (92) -- (146) [very thick, color=red];
\draw (92) -- (148) [very thick, color=red];
\draw (93) -- (186) [very thick, color=red];
\draw (96) -- (105) [very thick, color=red];
\draw (96) -- (188) [very thick, color=red];
\draw (97) -- (161) [very thick, color=red];
\draw (97) -- (187) [very thick, color=red];
\draw (98) -- (125) [very thick, color=red];
\draw (99) -- (179) [very thick, color=red];
\draw (102) -- (176) [very thick, color=red];
\draw (102) -- (178) [very thick, color=red];
\draw (112) -- (193) [very thick, color=red];
\draw (120) -- (172) [very thick, color=red];
\draw (132) -- (180) [very thick, color=red];
\draw (136) -- (191) [very thick, color=red];
\draw (136) -- (197) [very thick, color=red];
\draw (136) -- (199) [very thick, color=red];
\draw (139) -- (157) [very thick, color=red];
\draw (147) -- (181) [very thick, color=red];
\draw (155) -- (174) [very thick, color=red];
\draw (158) -- (191) [very thick, color=red];
\draw (158) -- (198) [very thick, color=red];
\draw (159) -- (182) [very thick, color=red];
\draw (167) -- (179) [very thick, color=red];
\draw (171) -- (172) [very thick, color=red];
\draw (171) -- (173) [very thick, color=red];
\draw (184) -- (192) [very thick, color=red];
\draw (186) -- (196) [very thick, color=red];
\draw (191) -- (196) [very thick, color=red];
\draw (192) -- (198) [very thick, color=red];
\draw (193) -- (194) [very thick, color=red];
\draw (194) -- (196) [very thick, color=red];
\draw (195) -- (198) [very thick, color=red];
\draw (15) -- (167) [dashed, very thin, color=teal];
\draw (22) -- (192) [dashed, very thin, color=teal];
\draw (23) -- (193) [dashed, very thin, color=teal];
\draw (50) -- (7) [dashed, very thin, color=teal];
\draw (66) -- (192) [dashed, very thin, color=teal];
\draw (70) -- (51) [dashed, very thin, color=teal];
\draw (71) -- (51) [dashed, very thin, color=teal];
\draw (73) -- (147) [dashed, very thin, color=teal];
\draw (75) -- (131) [dashed, very thin, color=teal];
\draw (85) -- (170) [dashed, very thin, color=teal];
\draw (86) -- (157) [dashed, very thin, color=teal];
\draw (89) -- (185) [dashed, very thin, color=teal];
\draw (91) -- (90) [dashed, very thin, color=teal];
\draw (94) -- (2) [dashed, very thin, color=teal];
\draw (95) -- (2) [dashed, very thin, color=teal];
\draw (114) -- (64) [dashed, very thin, color=teal];
\draw (116) -- (115) [dashed, very thin, color=teal];
\draw (123) -- (178) [dashed, very thin, color=teal];
\draw (124) -- (83) [dashed, very thin, color=teal];
\draw (128) -- (167) [dashed, very thin, color=teal];
\draw (135) -- (25) [dashed, very thin, color=teal];
\draw (137) -- (90) [dashed, very thin, color=teal];
\draw (140) -- (174) [dashed, very thin, color=teal];
\draw (141) -- (193) [dashed, very thin, color=teal];
\draw (142) -- (186) [dashed, very thin, color=teal];
\draw (145) -- (147) [dashed, very thin, color=teal];
\draw (150) -- (178) [dashed, very thin, color=teal];
\draw (151) -- (67) [dashed, very thin, color=teal];
\draw (154) -- (138) [dashed, very thin, color=teal];
\draw (156) -- (82) [dashed, very thin, color=teal];
\draw (160) -- (185) [dashed, very thin, color=teal];
\draw (162) -- (97) [dashed, very thin, color=teal];
\draw (163) -- (31) [dashed, very thin, color=teal];
\draw (164) -- (11) [dashed, very thin, color=teal];
\draw (166) -- (138) [dashed, very thin, color=teal];
\draw (169) -- (131) [dashed, very thin, color=teal];
\draw (177) -- (180) [dashed, very thin, color=teal];
\draw (183) -- (1) [dashed, very thin, color=teal];
\draw (189) -- (97) [dashed, very thin, color=teal];
\end{tikzpicture}
\begin{center}An optimal tree with $L=518.98$.\end{center}
\end{adjustbox}
\caption{Comparison between optimal solutions of two tree instances generated from base instance p5, with $c_{v}=2$ and $r_{v}=16.46$ for all vertices $v\in V$.}
\label{final_fig_12}
\end{figure}

Furthermore, the coverage capacities $c_{v}$ and the covering radia $r_{v}$ determine which and how many vertices can be covered, thus having a large impact on the optimal solutions. For the tour subgraph case, when the coverage capacities $c_v$ and the radia $r_v$ are sufficiently low, the problem reduces to the OP. On the other hand, when the $c_v$ and the $r_v$ are large enough, the resulting optimal solutions approach those of the TCMCSP. An analogous behaviour can be observed with the tree subgraph case, with respect to the tree counterparts of the OP and the TCMCSP. Figure \ref{final_fig_21} shows how the low-budget solutions (on the left) of Figures \ref{final_fig_11} and \ref{final_fig_12} change when the coverage capacity $c_{v}$ is significantly raised. Analogously, Figure \ref{final_fig_22} shows how the same solutions of Figures \ref{final_fig_11} and \ref{final_fig_12} change when instead the covering radius $r_{v}$ is increased.

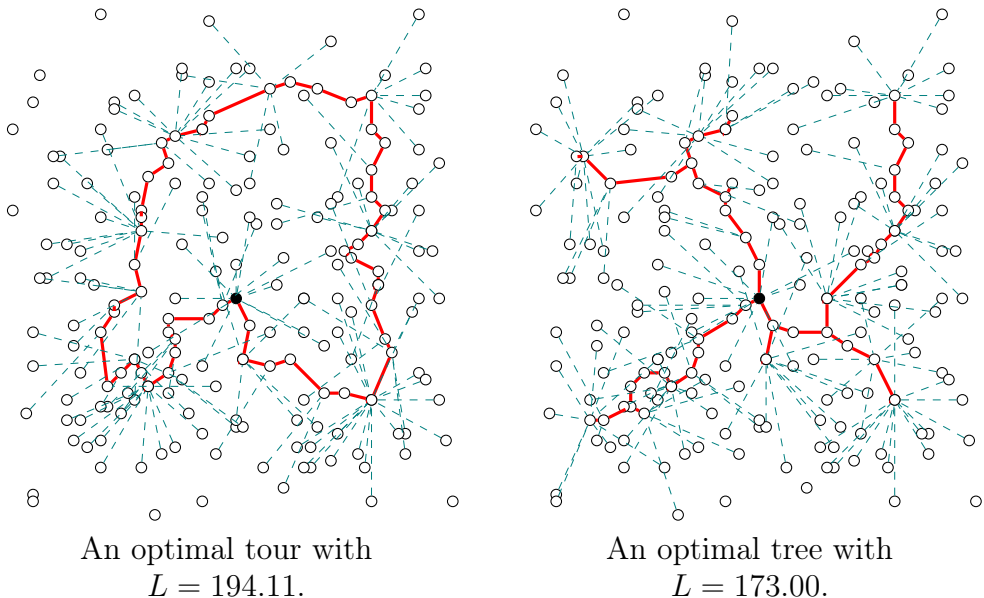
\begin{figure}[t]
\centering
\begin{adjustbox}{valign=t,minipage={.45\textwidth}}
\centering
\begin{tikzpicture}[scale=0.09]
\tikzstyle{every node}=[draw,circle,fill=white,minimum size=4pt,
                            inner sep=0pt]
\draw (35, 35) node (0) [overlay, color=black] {};
\draw (22, 22) node (1) [overlay] {};
\draw (36, 26) node (2) [overlay] {};
\draw (21, 45) node (3) [overlay] {};
\draw (45, 35) node (4) [overlay] {};
\draw (55, 20) node (5) [overlay] {};
\draw (33, 34) node (6) [overlay] {};
\draw (50, 50) node (7) [overlay] {};
\draw (55, 45) node (8) [overlay] {};
\draw (26, 59) node (9) [overlay] {};
\draw (40, 66) node (10) [overlay] {};
\draw (55, 65) node (11) [overlay] {};
\draw (35, 51) node (12) [overlay] {};
\draw (62, 35) node (13) [overlay] {};
\draw (62, 57) node (14) [overlay] {};
\draw (62, 24) node (15) [overlay] {};
\draw (21, 36) node (16) [overlay] {};
\draw (33, 44) node (17) [overlay] {};
\draw (9, 56) node (18) [overlay] {};
\draw (62, 48) node (19) [overlay] {};
\draw (66, 14) node (20) [overlay] {};
\draw (44, 13) node (21) [overlay] {};
\draw (26, 13) node (22) [overlay] {};
\draw (11, 28) node (23) [overlay] {};
\draw (7, 43) node (24) [overlay] {};
\draw (17, 64) node (25) [overlay] {};
\draw (41, 46) node (26) [overlay] {};
\draw (55, 34) node (27) [overlay] {};
\draw (35, 16) node (28) [overlay] {};
\draw (52, 26) node (29) [overlay] {};
\draw (43, 26) node (30) [overlay] {};
\draw (31, 76) node (31) [overlay] {};
\draw (22, 53) node (32) [overlay] {};
\draw (26, 29) node (33) [overlay] {};
\draw (50, 40) node (34) [overlay] {};
\draw (55, 50) node (35) [overlay] {};
\draw (54, 10) node (36) [overlay] {};
\draw (60, 15) node (37) [overlay] {};
\draw (47, 66) node (38) [overlay] {};
\draw (30, 60) node (39) [overlay] {};
\draw (30, 50) node (40) [overlay] {};
\draw (12, 17) node (41) [overlay] {};
\draw (15, 14) node (42) [overlay] {};
\draw (16, 19) node (43) [overlay] {};
\draw (21, 48) node (44) [overlay] {};
\draw (50, 30) node (45) [overlay] {};
\draw (51, 42) node (46) [overlay] {};
\draw (50, 15) node (47) [overlay] {};
\draw (48, 21) node (48) [overlay] {};
\draw (12, 38) node (49) [overlay] {};
\draw (37, 52) node (50) [overlay] {};
\draw (49, 49) node (51) [overlay] {};
\draw (52, 64) node (52) [overlay] {};
\draw (20, 26) node (53) [overlay] {};
\draw (40, 30) node (54) [overlay] {};
\draw (21, 47) node (55) [overlay] {};
\draw (17, 63) node (56) [overlay] {};
\draw (31, 62) node (57) [overlay] {};
\draw (52, 33) node (58) [overlay] {};
\draw (51, 21) node (59) [overlay] {};
\draw (42, 41) node (60) [overlay] {};
\draw (31, 32) node (61) [overlay] {};
\draw (5, 25) node (62) [overlay] {};
\draw (12, 42) node (63) [overlay] {};
\draw (36, 16) node (64) [overlay] {};
\draw (52, 41) node (65) [overlay] {};
\draw (27, 23) node (66) [overlay] {};
\draw (17, 33) node (67) [overlay] {};
\draw (13, 13) node (68) [overlay] {};
\draw (57, 58) node (69) [overlay] {};
\draw (62, 42) node (70) [overlay] {};
\draw (42, 57) node (71) [overlay] {};
\draw (16, 57) node (72) [overlay] {};
\draw (8, 52) node (73) [overlay] {};
\draw (7, 38) node (74) [overlay] {};
\draw (27, 68) node (75) [overlay] {};
\draw (30, 48) node (76) [overlay] {};
\draw (43, 67) node (77) [overlay] {};
\draw (58, 48) node (78) [overlay] {};
\draw (58, 27) node (79) [overlay] {};
\draw (37, 69) node (80) [overlay] {};
\draw (38, 46) node (81) [overlay] {};
\draw (46, 10) node (82) [overlay] {};
\draw (61, 33) node (83) [overlay] {};
\draw (62, 63) node (84) [overlay] {};
\draw (63, 69) node (85) [overlay] {};
\draw (32, 22) node (86) [overlay] {};
\draw (45, 35) node (87) [overlay] {};
\draw (59, 15) node (88) [overlay] {};
\draw (5, 6) node (89) [overlay] {};
\draw (10, 17) node (90) [overlay] {};
\draw (21, 10) node (91) [overlay] {};
\draw (5, 64) node (92) [overlay] {};
\draw (30, 15) node (93) [overlay] {};
\draw (39, 10) node (94) [overlay] {};
\draw (32, 39) node (95) [overlay] {};
\draw (25, 32) node (96) [overlay] {};
\draw (25, 55) node (97) [overlay] {};
\draw (48, 28) node (98) [overlay] {};
\draw (56, 37) node (99) [overlay] {};
\draw (41, 49) node (100) [overlay] {};
\draw (35, 17) node (101) [overlay] {};
\draw (55, 45) node (102) [overlay] {};
\draw (55, 20) node (103) [overlay] {};
\draw (15, 30) node (104) [overlay] {};
\draw (25, 30) node (105) [overlay] {};
\draw (20, 50) node (106) [overlay] {};
\draw (10, 43) node (107) [overlay] {};
\draw (55, 60) node (108) [overlay] {};
\draw (30, 60) node (109) [overlay] {};
\draw (20, 65) node (110) [overlay] {};
\draw (50, 35) node (111) [overlay] {};
\draw (30, 25) node (112) [overlay] {};
\draw (15, 10) node (113) [overlay] {};
\draw (30, 5) node (114) [overlay] {};
\draw (10, 20) node (115) [overlay] {};
\draw (5, 30) node (116) [overlay] {};
\draw (20, 40) node (117) [overlay] {};
\draw (15, 60) node (118) [overlay] {};
\draw (45, 65) node (119) [overlay] {};
\draw (45, 20) node (120) [overlay] {};
\draw (45, 10) node (121) [overlay] {};
\draw (55, 5) node (122) [overlay] {};
\draw (65, 35) node (123) [overlay] {};
\draw (65, 20) node (124) [overlay] {};
\draw (45, 30) node (125) [overlay] {};
\draw (35, 40) node (126) [overlay] {};
\draw (41, 37) node (127) [overlay] {};
\draw (64, 42) node (128) [overlay] {};
\draw (40, 60) node (129) [overlay] {};
\draw (31, 52) node (130) [overlay] {};
\draw (35, 69) node (131) [overlay] {};
\draw (53, 52) node (132) [overlay] {};
\draw (65, 55) node (133) [overlay] {};
\draw (63, 65) node (134) [overlay] {};
\draw (2, 60) node (135) [overlay] {};
\draw (20, 20) node (136) [overlay] {};
\draw (5, 5) node (137) [overlay] {};
\draw (60, 12) node (138) [overlay] {};
\draw (40, 25) node (139) [overlay] {};
\draw (42, 7) node (140) [overlay] {};
\draw (24, 12) node (141) [overlay] {};
\draw (23, 3) node (142) [overlay] {};
\draw (11, 14) node (143) [overlay] {};
\draw (6, 38) node (144) [overlay] {};
\draw (2, 48) node (145) [overlay] {};
\draw (8, 56) node (146) [overlay] {};
\draw (13, 52) node (147) [overlay] {};
\draw (6, 68) node (148) [overlay] {};
\draw (47, 47) node (149) [overlay] {};
\draw (49, 58) node (150) [overlay] {};
\draw (27, 43) node (151) [overlay] {};
\draw (37, 31) node (152) [overlay] {};
\draw (57, 29) node (153) [overlay] {};
\draw (63, 23) node (154) [overlay] {};
\draw (53, 12) node (155) [overlay] {};
\draw (32, 12) node (156) [overlay] {};
\draw (36, 26) node (157) [overlay] {};
\draw (21, 24) node (158) [overlay] {};
\draw (17, 34) node (159) [overlay] {};
\draw (12, 24) node (160) [overlay] {};
\draw (24, 58) node (161) [overlay] {};
\draw (27, 69) node (162) [overlay] {};
\draw (15, 77) node (163) [overlay] {};
\draw (62, 77) node (164) [overlay] {};
\draw (49, 73) node (165) [overlay] {};
\draw (67, 5) node (166) [overlay] {};
\draw (56, 39) node (167) [overlay] {};
\draw (37, 47) node (168) [overlay] {};
\draw (37, 56) node (169) [overlay] {};
\draw (57, 68) node (170) [overlay] {};
\draw (47, 16) node (171) [overlay] {};
\draw (44, 17) node (172) [overlay] {};
\draw (46, 13) node (173) [overlay] {};
\draw (49, 11) node (174) [overlay] {};
\draw (49, 42) node (175) [overlay] {};
\draw (53, 43) node (176) [overlay] {};
\draw (61, 52) node (177) [overlay] {};
\draw (57, 48) node (178) [overlay] {};
\draw (56, 37) node (179) [overlay] {};
\draw (55, 54) node (180) [overlay] {};
\draw (15, 47) node (181) [overlay] {};
\draw (14, 37) node (182) [overlay] {};
\draw (11, 31) node (183) [overlay] {};
\draw (16, 22) node (184) [overlay] {};
\draw (4, 18) node (185) [overlay] {};
\draw (28, 18) node (186) [overlay] {};
\draw (26, 52) node (187) [overlay] {};
\draw (26, 35) node (188) [overlay] {};
\draw (31, 67) node (189) [overlay] {};
\draw (15, 19) node (190) [overlay] {};
\draw (22, 22) node (191) [overlay] {};
\draw (18, 24) node (192) [overlay] {};
\draw (26, 27) node (193) [overlay] {};
\draw (25, 24) node (194) [overlay] {};
\draw (22, 27) node (195) [overlay] {};
\draw (25, 21) node (196) [overlay] {};
\draw (19, 21) node (197) [overlay] {};
\draw (20, 26) node (198) [overlay] {};
\draw (18, 18) node (199) [overlay] {};
\draw (0) -- (6) [very thick, color=red];
\draw (0) -- (152) [very thick, color=red];
\draw (1) -- (191) [very thick, color=red];
\draw (1) -- (194) [very thick, color=red];
\draw (2) -- (139) [very thick, color=red];
\draw (2) -- (157) [very thick, color=red];
\draw (3) -- (55) [very thick, color=red];
\draw (3) -- (117) [very thick, color=red];
\draw (5) -- (59) [very thick, color=red];
\draw (5) -- (103) [very thick, color=red];
\draw (6) -- (61) [very thick, color=red];
\draw (8) -- (102) [very thick, color=red];
\draw (8) -- (178) [very thick, color=red];
\draw (9) -- (109) [very thick, color=red];
\draw (9) -- (161) [very thick, color=red];
\draw (10) -- (57) [very thick, color=red];
\draw (10) -- (77) [very thick, color=red];
\draw (11) -- (52) [very thick, color=red];
\draw (11) -- (108) [very thick, color=red];
\draw (16) -- (117) [very thick, color=red];
\draw (16) -- (159) [very thick, color=red];
\draw (27) -- (153) [very thick, color=red];
\draw (27) -- (179) [very thick, color=red];
\draw (30) -- (48) [very thick, color=red];
\draw (30) -- (139) [very thick, color=red];
\draw (32) -- (44) [very thick, color=red];
\draw (32) -- (97) [very thick, color=red];
\draw (33) -- (105) [very thick, color=red];
\draw (33) -- (193) [very thick, color=red];
\draw (35) -- (178) [very thick, color=red];
\draw (35) -- (180) [very thick, color=red];
\draw (38) -- (52) [very thick, color=red];
\draw (38) -- (77) [very thick, color=red];
\draw (39) -- (57) [very thick, color=red];
\draw (39) -- (109) [very thick, color=red];
\draw (44) -- (55) [very thick, color=red];
\draw (46) -- (65) [very thick, color=red];
\draw (46) -- (176) [very thick, color=red];
\draw (48) -- (59) [very thick, color=red];
\draw (53) -- (158) [very thick, color=red];
\draw (53) -- (198) [very thick, color=red];
\draw (61) -- (96) [very thick, color=red];
\draw (65) -- (167) [very thick, color=red];
\draw (67) -- (104) [very thick, color=red];
\draw (67) -- (159) [very thick, color=red];
\draw (69) -- (108) [very thick, color=red];
\draw (69) -- (180) [very thick, color=red];
\draw (79) -- (103) [very thick, color=red];
\draw (79) -- (153) [very thick, color=red];
\draw (96) -- (105) [very thick, color=red];
\draw (97) -- (161) [very thick, color=red];
\draw (99) -- (167) [very thick, color=red];
\draw (99) -- (179) [very thick, color=red];
\draw (102) -- (176) [very thick, color=red];
\draw (104) -- (184) [very thick, color=red];
\draw (152) -- (157) [very thick, color=red];
\draw (158) -- (191) [very thick, color=red];
\draw (184) -- (192) [very thick, color=red];
\draw (192) -- (198) [very thick, color=red];
\draw (193) -- (194) [very thick, color=red];
\draw (4) -- (79) [dashed, very thin, color=teal];
\draw (7) -- (11) [dashed, very thin, color=teal];
\draw (12) -- (9) [dashed, very thin, color=teal];
\draw (13) -- (79) [dashed, very thin, color=teal];
\draw (14) -- (11) [dashed, very thin, color=teal];
\draw (15) -- (103) [dashed, very thin, color=teal];
\draw (17) -- (3) [dashed, very thin, color=teal];
\draw (18) -- (3) [dashed, very thin, color=teal];
\draw (19) -- (27) [dashed, very thin, color=teal];
\draw (20) -- (103) [dashed, very thin, color=teal];
\draw (21) -- (103) [dashed, very thin, color=teal];
\draw (22) -- (1) [dashed, very thin, color=teal];
\draw (23) -- (159) [dashed, very thin, color=teal];
\draw (24) -- (3) [dashed, very thin, color=teal];
\draw (25) -- (9) [dashed, very thin, color=teal];
\draw (26) -- (35) [dashed, very thin, color=teal];
\draw (28) -- (194) [dashed, very thin, color=teal];
\draw (29) -- (179) [dashed, very thin, color=teal];
\draw (31) -- (10) [dashed, very thin, color=teal];
\draw (34) -- (79) [dashed, very thin, color=teal];
\draw (36) -- (5) [dashed, very thin, color=teal];
\draw (37) -- (79) [dashed, very thin, color=teal];
\draw (40) -- (6) [dashed, very thin, color=teal];
\draw (41) -- (1) [dashed, very thin, color=teal];
\draw (42) -- (1) [dashed, very thin, color=teal];
\draw (43) -- (33) [dashed, very thin, color=teal];
\draw (45) -- (103) [dashed, very thin, color=teal];
\draw (47) -- (103) [dashed, very thin, color=teal];
\draw (49) -- (16) [dashed, very thin, color=teal];
\draw (50) -- (10) [dashed, very thin, color=teal];
\draw (51) -- (27) [dashed, very thin, color=teal];
\draw (54) -- (194) [dashed, very thin, color=teal];
\draw (56) -- (9) [dashed, very thin, color=teal];
\draw (58) -- (5) [dashed, very thin, color=teal];
\draw (60) -- (6) [dashed, very thin, color=teal];
\draw (62) -- (198) [dashed, very thin, color=teal];
\draw (63) -- (3) [dashed, very thin, color=teal];
\draw (64) -- (194) [dashed, very thin, color=teal];
\draw (66) -- (198) [dashed, very thin, color=teal];
\draw (68) -- (1) [dashed, very thin, color=teal];
\draw (70) -- (8) [dashed, very thin, color=teal];
\draw (71) -- (10) [dashed, very thin, color=teal];
\draw (72) -- (161) [dashed, very thin, color=teal];
\draw (73) -- (3) [dashed, very thin, color=teal];
\draw (74) -- (3) [dashed, very thin, color=teal];
\draw (75) -- (10) [dashed, very thin, color=teal];
\draw (76) -- (6) [dashed, very thin, color=teal];
\draw (78) -- (8) [dashed, very thin, color=teal];
\draw (80) -- (9) [dashed, very thin, color=teal];
\draw (81) -- (6) [dashed, very thin, color=teal];
\draw (82) -- (5) [dashed, very thin, color=teal];
\draw (83) -- (5) [dashed, very thin, color=teal];
\draw (84) -- (11) [dashed, very thin, color=teal];
\draw (85) -- (11) [dashed, very thin, color=teal];
\draw (86) -- (184) [dashed, very thin, color=teal];
\draw (87) -- (2) [dashed, very thin, color=teal];
\draw (88) -- (103) [dashed, very thin, color=teal];
\draw (90) -- (1) [dashed, very thin, color=teal];
\draw (91) -- (1) [dashed, very thin, color=teal];
\draw (93) -- (1) [dashed, very thin, color=teal];
\draw (94) -- (48) [dashed, very thin, color=teal];
\draw (95) -- (2) [dashed, very thin, color=teal];
\draw (98) -- (0) [dashed, very thin, color=teal];
\draw (100) -- (8) [dashed, very thin, color=teal];
\draw (101) -- (2) [dashed, very thin, color=teal];
\draw (106) -- (117) [dashed, very thin, color=teal];
\draw (107) -- (3) [dashed, very thin, color=teal];
\draw (110) -- (9) [dashed, very thin, color=teal];
\draw (111) -- (5) [dashed, very thin, color=teal];
\draw (112) -- (0) [dashed, very thin, color=teal];
\draw (113) -- (1) [dashed, very thin, color=teal];
\draw (115) -- (53) [dashed, very thin, color=teal];
\draw (116) -- (53) [dashed, very thin, color=teal];
\draw (118) -- (3) [dashed, very thin, color=teal];
\draw (119) -- (180) [dashed, very thin, color=teal];
\draw (120) -- (153) [dashed, very thin, color=teal];
\draw (121) -- (5) [dashed, very thin, color=teal];
\draw (122) -- (103) [dashed, very thin, color=teal];
\draw (123) -- (8) [dashed, very thin, color=teal];
\draw (124) -- (5) [dashed, very thin, color=teal];
\draw (125) -- (0) [dashed, very thin, color=teal];
\draw (126) -- (0) [dashed, very thin, color=teal];
\draw (127) -- (0) [dashed, very thin, color=teal];
\draw (128) -- (8) [dashed, very thin, color=teal];
\draw (129) -- (9) [dashed, very thin, color=teal];
\draw (130) -- (161) [dashed, very thin, color=teal];
\draw (131) -- (9) [dashed, very thin, color=teal];
\draw (132) -- (8) [dashed, very thin, color=teal];
\draw (133) -- (8) [dashed, very thin, color=teal];
\draw (134) -- (11) [dashed, very thin, color=teal];
\draw (136) -- (16) [dashed, very thin, color=teal];
\draw (138) -- (79) [dashed, very thin, color=teal];
\draw (140) -- (48) [dashed, very thin, color=teal];
\draw (141) -- (1) [dashed, very thin, color=teal];
\draw (143) -- (1) [dashed, very thin, color=teal];
\draw (144) -- (16) [dashed, very thin, color=teal];
\draw (146) -- (161) [dashed, very thin, color=teal];
\draw (147) -- (3) [dashed, very thin, color=teal];
\draw (149) -- (8) [dashed, very thin, color=teal];
\draw (150) -- (8) [dashed, very thin, color=teal];
\draw (151) -- (0) [dashed, very thin, color=teal];
\draw (154) -- (5) [dashed, very thin, color=teal];
\draw (155) -- (5) [dashed, very thin, color=teal];
\draw (156) -- (139) [dashed, very thin, color=teal];
\draw (160) -- (16) [dashed, very thin, color=teal];
\draw (162) -- (9) [dashed, very thin, color=teal];
\draw (164) -- (11) [dashed, very thin, color=teal];
\draw (165) -- (10) [dashed, very thin, color=teal];
\draw (168) -- (0) [dashed, very thin, color=teal];
\draw (169) -- (9) [dashed, very thin, color=teal];
\draw (170) -- (11) [dashed, very thin, color=teal];
\draw (171) -- (157) [dashed, very thin, color=teal];
\draw (172) -- (157) [dashed, very thin, color=teal];
\draw (173) -- (103) [dashed, very thin, color=teal];
\draw (174) -- (5) [dashed, very thin, color=teal];
\draw (175) -- (0) [dashed, very thin, color=teal];
\draw (177) -- (8) [dashed, very thin, color=teal];
\draw (181) -- (3) [dashed, very thin, color=teal];
\draw (182) -- (16) [dashed, very thin, color=teal];
\draw (183) -- (16) [dashed, very thin, color=teal];
\draw (185) -- (104) [dashed, very thin, color=teal];
\draw (186) -- (198) [dashed, very thin, color=teal];
\draw (187) -- (117) [dashed, very thin, color=teal];
\draw (188) -- (0) [dashed, very thin, color=teal];
\draw (189) -- (9) [dashed, very thin, color=teal];
\draw (190) -- (192) [dashed, very thin, color=teal];
\draw (195) -- (0) [dashed, very thin, color=teal];
\draw (196) -- (193) [dashed, very thin, color=teal];
\draw (197) -- (33) [dashed, very thin, color=teal];
\draw (199) -- (33) [dashed, very thin, color=teal];

\end{tikzpicture}
\begin{center}An optimal tour with $L=194.11$.\end{center}
\end{adjustbox}\hfill
\begin{adjustbox}{valign=t,minipage={.45\textwidth}}
\centering
\begin{tikzpicture}[scale=0.09]
\tikzstyle{every node}=[draw,circle,fill=white,minimum size=4pt,
                            inner sep=0pt]
\draw (35, 35) node (0) [overlay, color=black] {};
\draw (22, 22) node (1) [overlay] {};
\draw (36, 26) node (2) [overlay] {};
\draw (21, 45) node (3) [overlay] {};
\draw (45, 35) node (4) [overlay] {};
\draw (55, 20) node (5) [overlay] {};
\draw (33, 34) node (6) [overlay] {};
\draw (50, 50) node (7) [overlay] {};
\draw (55, 45) node (8) [overlay] {};
\draw (26, 59) node (9) [overlay] {};
\draw (40, 66) node (10) [overlay] {};
\draw (55, 65) node (11) [overlay] {};
\draw (35, 51) node (12) [overlay] {};
\draw (62, 35) node (13) [overlay] {};
\draw (62, 57) node (14) [overlay] {};
\draw (62, 24) node (15) [overlay] {};
\draw (21, 36) node (16) [overlay] {};
\draw (33, 44) node (17) [overlay] {};
\draw (9, 56) node (18) [overlay] {};
\draw (62, 48) node (19) [overlay] {};
\draw (66, 14) node (20) [overlay] {};
\draw (44, 13) node (21) [overlay] {};
\draw (26, 13) node (22) [overlay] {};
\draw (11, 28) node (23) [overlay] {};
\draw (7, 43) node (24) [overlay] {};
\draw (17, 64) node (25) [overlay] {};
\draw (41, 46) node (26) [overlay] {};
\draw (55, 34) node (27) [overlay] {};
\draw (35, 16) node (28) [overlay] {};
\draw (52, 26) node (29) [overlay] {};
\draw (43, 26) node (30) [overlay] {};
\draw (31, 76) node (31) [overlay] {};
\draw (22, 53) node (32) [overlay] {};
\draw (26, 29) node (33) [overlay] {};
\draw (50, 40) node (34) [overlay] {};
\draw (55, 50) node (35) [overlay] {};
\draw (54, 10) node (36) [overlay] {};
\draw (60, 15) node (37) [overlay] {};
\draw (47, 66) node (38) [overlay] {};
\draw (30, 60) node (39) [overlay] {};
\draw (30, 50) node (40) [overlay] {};
\draw (12, 17) node (41) [overlay] {};
\draw (15, 14) node (42) [overlay] {};
\draw (16, 19) node (43) [overlay] {};
\draw (21, 48) node (44) [overlay] {};
\draw (50, 30) node (45) [overlay] {};
\draw (51, 42) node (46) [overlay] {};
\draw (50, 15) node (47) [overlay] {};
\draw (48, 21) node (48) [overlay] {};
\draw (12, 38) node (49) [overlay] {};
\draw (37, 52) node (50) [overlay] {};
\draw (49, 49) node (51) [overlay] {};
\draw (52, 64) node (52) [overlay] {};
\draw (20, 26) node (53) [overlay] {};
\draw (40, 30) node (54) [overlay] {};
\draw (21, 47) node (55) [overlay] {};
\draw (17, 63) node (56) [overlay] {};
\draw (31, 62) node (57) [overlay] {};
\draw (52, 33) node (58) [overlay] {};
\draw (51, 21) node (59) [overlay] {};
\draw (42, 41) node (60) [overlay] {};
\draw (31, 32) node (61) [overlay] {};
\draw (5, 25) node (62) [overlay] {};
\draw (12, 42) node (63) [overlay] {};
\draw (36, 16) node (64) [overlay] {};
\draw (52, 41) node (65) [overlay] {};
\draw (27, 23) node (66) [overlay] {};
\draw (17, 33) node (67) [overlay] {};
\draw (13, 13) node (68) [overlay] {};
\draw (57, 58) node (69) [overlay] {};
\draw (62, 42) node (70) [overlay] {};
\draw (42, 57) node (71) [overlay] {};
\draw (16, 57) node (72) [overlay] {};
\draw (8, 52) node (73) [overlay] {};
\draw (7, 38) node (74) [overlay] {};
\draw (27, 68) node (75) [overlay] {};
\draw (30, 48) node (76) [overlay] {};
\draw (43, 67) node (77) [overlay] {};
\draw (58, 48) node (78) [overlay] {};
\draw (58, 27) node (79) [overlay] {};
\draw (37, 69) node (80) [overlay] {};
\draw (38, 46) node (81) [overlay] {};
\draw (46, 10) node (82) [overlay] {};
\draw (61, 33) node (83) [overlay] {};
\draw (62, 63) node (84) [overlay] {};
\draw (63, 69) node (85) [overlay] {};
\draw (32, 22) node (86) [overlay] {};
\draw (45, 35) node (87) [overlay] {};
\draw (59, 15) node (88) [overlay] {};
\draw (5, 6) node (89) [overlay] {};
\draw (10, 17) node (90) [overlay] {};
\draw (21, 10) node (91) [overlay] {};
\draw (5, 64) node (92) [overlay] {};
\draw (30, 15) node (93) [overlay] {};
\draw (39, 10) node (94) [overlay] {};
\draw (32, 39) node (95) [overlay] {};
\draw (25, 32) node (96) [overlay] {};
\draw (25, 55) node (97) [overlay] {};
\draw (48, 28) node (98) [overlay] {};
\draw (56, 37) node (99) [overlay] {};
\draw (41, 49) node (100) [overlay] {};
\draw (35, 17) node (101) [overlay] {};
\draw (55, 45) node (102) [overlay] {};
\draw (55, 20) node (103) [overlay] {};
\draw (15, 30) node (104) [overlay] {};
\draw (25, 30) node (105) [overlay] {};
\draw (20, 50) node (106) [overlay] {};
\draw (10, 43) node (107) [overlay] {};
\draw (55, 60) node (108) [overlay] {};
\draw (30, 60) node (109) [overlay] {};
\draw (20, 65) node (110) [overlay] {};
\draw (50, 35) node (111) [overlay] {};
\draw (30, 25) node (112) [overlay] {};
\draw (15, 10) node (113) [overlay] {};
\draw (30, 5) node (114) [overlay] {};
\draw (10, 20) node (115) [overlay] {};
\draw (5, 30) node (116) [overlay] {};
\draw (20, 40) node (117) [overlay] {};
\draw (15, 60) node (118) [overlay] {};
\draw (45, 65) node (119) [overlay] {};
\draw (45, 20) node (120) [overlay] {};
\draw (45, 10) node (121) [overlay] {};
\draw (55, 5) node (122) [overlay] {};
\draw (65, 35) node (123) [overlay] {};
\draw (65, 20) node (124) [overlay] {};
\draw (45, 30) node (125) [overlay] {};
\draw (35, 40) node (126) [overlay] {};
\draw (41, 37) node (127) [overlay] {};
\draw (64, 42) node (128) [overlay] {};
\draw (40, 60) node (129) [overlay] {};
\draw (31, 52) node (130) [overlay] {};
\draw (35, 69) node (131) [overlay] {};
\draw (53, 52) node (132) [overlay] {};
\draw (65, 55) node (133) [overlay] {};
\draw (63, 65) node (134) [overlay] {};
\draw (2, 60) node (135) [overlay] {};
\draw (20, 20) node (136) [overlay] {};
\draw (5, 5) node (137) [overlay] {};
\draw (60, 12) node (138) [overlay] {};
\draw (40, 25) node (139) [overlay] {};
\draw (42, 7) node (140) [overlay] {};
\draw (24, 12) node (141) [overlay] {};
\draw (23, 3) node (142) [overlay] {};
\draw (11, 14) node (143) [overlay] {};
\draw (6, 38) node (144) [overlay] {};
\draw (2, 48) node (145) [overlay] {};
\draw (8, 56) node (146) [overlay] {};
\draw (13, 52) node (147) [overlay] {};
\draw (6, 68) node (148) [overlay] {};
\draw (47, 47) node (149) [overlay] {};
\draw (49, 58) node (150) [overlay] {};
\draw (27, 43) node (151) [overlay] {};
\draw (37, 31) node (152) [overlay] {};
\draw (57, 29) node (153) [overlay] {};
\draw (63, 23) node (154) [overlay] {};
\draw (53, 12) node (155) [overlay] {};
\draw (32, 12) node (156) [overlay] {};
\draw (36, 26) node (157) [overlay] {};
\draw (21, 24) node (158) [overlay] {};
\draw (17, 34) node (159) [overlay] {};
\draw (12, 24) node (160) [overlay] {};
\draw (24, 58) node (161) [overlay] {};
\draw (27, 69) node (162) [overlay] {};
\draw (15, 77) node (163) [overlay] {};
\draw (62, 77) node (164) [overlay] {};
\draw (49, 73) node (165) [overlay] {};
\draw (67, 5) node (166) [overlay] {};
\draw (56, 39) node (167) [overlay] {};
\draw (37, 47) node (168) [overlay] {};
\draw (37, 56) node (169) [overlay] {};
\draw (57, 68) node (170) [overlay] {};
\draw (47, 16) node (171) [overlay] {};
\draw (44, 17) node (172) [overlay] {};
\draw (46, 13) node (173) [overlay] {};
\draw (49, 11) node (174) [overlay] {};
\draw (49, 42) node (175) [overlay] {};
\draw (53, 43) node (176) [overlay] {};
\draw (61, 52) node (177) [overlay] {};
\draw (57, 48) node (178) [overlay] {};
\draw (56, 37) node (179) [overlay] {};
\draw (55, 54) node (180) [overlay] {};
\draw (15, 47) node (181) [overlay] {};
\draw (14, 37) node (182) [overlay] {};
\draw (11, 31) node (183) [overlay] {};
\draw (16, 22) node (184) [overlay] {};
\draw (4, 18) node (185) [overlay] {};
\draw (28, 18) node (186) [overlay] {};
\draw (26, 52) node (187) [overlay] {};
\draw (26, 35) node (188) [overlay] {};
\draw (31, 67) node (189) [overlay] {};
\draw (15, 19) node (190) [overlay] {};
\draw (22, 22) node (191) [overlay] {};
\draw (18, 24) node (192) [overlay] {};
\draw (26, 27) node (193) [overlay] {};
\draw (25, 24) node (194) [overlay] {};
\draw (22, 27) node (195) [overlay] {};
\draw (25, 21) node (196) [overlay] {};
\draw (19, 21) node (197) [overlay] {};
\draw (20, 26) node (198) [overlay] {};
\draw (18, 18) node (199) [overlay] {};
\draw (0) -- (6) [very thick, color=red];
\draw (0) -- (126) [very thick, color=red];
\draw (0) -- (152) [very thick, color=red];
\draw (1) -- (191) [very thick, color=red];
\draw (2) -- (157) [very thick, color=red];
\draw (4) -- (87) [very thick, color=red];
\draw (5) -- (103) [very thick, color=red];
\draw (6) -- (61) [very thick, color=red];
\draw (8) -- (102) [very thick, color=red];
\draw (9) -- (109) [very thick, color=red];
\draw (9) -- (161) [very thick, color=red];
\draw (11) -- (108) [very thick, color=red];
\draw (17) -- (76) [very thick, color=red];
\draw (17) -- (126) [very thick, color=red];
\draw (18) -- (146) [very thick, color=red];
\draw (18) -- (147) [very thick, color=red];
\draw (29) -- (98) [very thick, color=red];
\draw (29) -- (103) [very thick, color=red];
\draw (32) -- (97) [very thick, color=red];
\draw (32) -- (147) [very thick, color=red];
\draw (33) -- (61) [very thick, color=red];
\draw (33) -- (193) [very thick, color=red];
\draw (34) -- (65) [very thick, color=red];
\draw (34) -- (87) [very thick, color=red];
\draw (35) -- (178) [very thick, color=red];
\draw (35) -- (180) [very thick, color=red];
\draw (39) -- (109) [very thick, color=red];
\draw (40) -- (76) [very thick, color=red];
\draw (40) -- (130) [very thick, color=red];
\draw (40) -- (187) [very thick, color=red];
\draw (41) -- (43) [very thick, color=red];
\draw (41) -- (90) [very thick, color=red];
\draw (43) -- (184) [very thick, color=red];
\draw (43) -- (199) [very thick, color=red];
\draw (46) -- (65) [very thick, color=red];
\draw (54) -- (125) [very thick, color=red];
\draw (54) -- (152) [very thick, color=red];
\draw (57) -- (109) [very thick, color=red];
\draw (65) -- (176) [very thick, color=red];
\draw (69) -- (108) [very thick, color=red];
\draw (69) -- (180) [very thick, color=red];
\draw (87) -- (125) [very thick, color=red];
\draw (97) -- (161) [very thick, color=red];
\draw (97) -- (187) [very thick, color=red];
\draw (98) -- (125) [very thick, color=red];
\draw (102) -- (176) [very thick, color=red];
\draw (102) -- (178) [very thick, color=red];
\draw (152) -- (157) [very thick, color=red];
\draw (158) -- (191) [very thick, color=red];
\draw (158) -- (192) [very thick, color=red];
\draw (184) -- (192) [very thick, color=red];
\draw (191) -- (194) [very thick, color=red];
\draw (193) -- (194) [very thick, color=red];
\draw (3) -- (6) [dashed, very thin, color=teal];
\draw (7) -- (4) [dashed, very thin, color=teal];
\draw (10) -- (109) [dashed, very thin, color=teal];
\draw (12) -- (187) [dashed, very thin, color=teal];
\draw (13) -- (29) [dashed, very thin, color=teal];
\draw (14) -- (8) [dashed, very thin, color=teal];
\draw (15) -- (29) [dashed, very thin, color=teal];
\draw (16) -- (194) [dashed, very thin, color=teal];
\draw (19) -- (34) [dashed, very thin, color=teal];
\draw (20) -- (5) [dashed, very thin, color=teal];
\draw (21) -- (5) [dashed, very thin, color=teal];
\draw (22) -- (199) [dashed, very thin, color=teal];
\draw (23) -- (90) [dashed, very thin, color=teal];
\draw (24) -- (146) [dashed, very thin, color=teal];
\draw (25) -- (18) [dashed, very thin, color=teal];
\draw (26) -- (187) [dashed, very thin, color=teal];
\draw (27) -- (5) [dashed, very thin, color=teal];
\draw (28) -- (194) [dashed, very thin, color=teal];
\draw (30) -- (29) [dashed, very thin, color=teal];
\draw (31) -- (39) [dashed, very thin, color=teal];
\draw (36) -- (103) [dashed, very thin, color=teal];
\draw (37) -- (103) [dashed, very thin, color=teal];
\draw (38) -- (11) [dashed, very thin, color=teal];
\draw (42) -- (1) [dashed, very thin, color=teal];
\draw (44) -- (17) [dashed, very thin, color=teal];
\draw (45) -- (2) [dashed, very thin, color=teal];
\draw (47) -- (29) [dashed, very thin, color=teal];
\draw (48) -- (157) [dashed, very thin, color=teal];
\draw (49) -- (147) [dashed, very thin, color=teal];
\draw (50) -- (9) [dashed, very thin, color=teal];
\draw (51) -- (8) [dashed, very thin, color=teal];
\draw (52) -- (11) [dashed, very thin, color=teal];
\draw (53) -- (90) [dashed, very thin, color=teal];
\draw (55) -- (17) [dashed, very thin, color=teal];
\draw (56) -- (9) [dashed, very thin, color=teal];
\draw (58) -- (4) [dashed, very thin, color=teal];
\draw (59) -- (2) [dashed, very thin, color=teal];
\draw (60) -- (2) [dashed, very thin, color=teal];
\draw (62) -- (90) [dashed, very thin, color=teal];
\draw (63) -- (18) [dashed, very thin, color=teal];
\draw (64) -- (2) [dashed, very thin, color=teal];
\draw (66) -- (199) [dashed, very thin, color=teal];
\draw (67) -- (6) [dashed, very thin, color=teal];
\draw (68) -- (90) [dashed, very thin, color=teal];
\draw (70) -- (8) [dashed, very thin, color=teal];
\draw (71) -- (180) [dashed, very thin, color=teal];
\draw (72) -- (18) [dashed, very thin, color=teal];
\draw (73) -- (18) [dashed, very thin, color=teal];
\draw (74) -- (147) [dashed, very thin, color=teal];
\draw (75) -- (9) [dashed, very thin, color=teal];
\draw (77) -- (39) [dashed, very thin, color=teal];
\draw (78) -- (8) [dashed, very thin, color=teal];
\draw (79) -- (29) [dashed, very thin, color=teal];
\draw (80) -- (9) [dashed, very thin, color=teal];
\draw (81) -- (4) [dashed, very thin, color=teal];
\draw (82) -- (5) [dashed, very thin, color=teal];
\draw (83) -- (4) [dashed, very thin, color=teal];
\draw (84) -- (11) [dashed, very thin, color=teal];
\draw (85) -- (11) [dashed, very thin, color=teal];
\draw (86) -- (199) [dashed, very thin, color=teal];
\draw (88) -- (103) [dashed, very thin, color=teal];
\draw (89) -- (90) [dashed, very thin, color=teal];
\draw (91) -- (90) [dashed, very thin, color=teal];
\draw (92) -- (18) [dashed, very thin, color=teal];
\draw (93) -- (2) [dashed, very thin, color=teal];
\draw (94) -- (2) [dashed, very thin, color=teal];
\draw (95) -- (2) [dashed, very thin, color=teal];
\draw (96) -- (199) [dashed, very thin, color=teal];
\draw (99) -- (4) [dashed, very thin, color=teal];
\draw (100) -- (4) [dashed, very thin, color=teal];
\draw (101) -- (194) [dashed, very thin, color=teal];
\draw (104) -- (90) [dashed, very thin, color=teal];
\draw (105) -- (0) [dashed, very thin, color=teal];
\draw (106) -- (9) [dashed, very thin, color=teal];
\draw (107) -- (18) [dashed, very thin, color=teal];
\draw (110) -- (18) [dashed, very thin, color=teal];
\draw (111) -- (29) [dashed, very thin, color=teal];
\draw (112) -- (0) [dashed, very thin, color=teal];
\draw (113) -- (1) [dashed, very thin, color=teal];
\draw (115) -- (1) [dashed, very thin, color=teal];
\draw (116) -- (41) [dashed, very thin, color=teal];
\draw (117) -- (0) [dashed, very thin, color=teal];
\draw (118) -- (9) [dashed, very thin, color=teal];
\draw (119) -- (11) [dashed, very thin, color=teal];
\draw (120) -- (29) [dashed, very thin, color=teal];
\draw (121) -- (5) [dashed, very thin, color=teal];
\draw (122) -- (103) [dashed, very thin, color=teal];
\draw (123) -- (8) [dashed, very thin, color=teal];
\draw (124) -- (103) [dashed, very thin, color=teal];
\draw (127) -- (0) [dashed, very thin, color=teal];
\draw (128) -- (8) [dashed, very thin, color=teal];
\draw (129) -- (11) [dashed, very thin, color=teal];
\draw (131) -- (9) [dashed, very thin, color=teal];
\draw (132) -- (8) [dashed, very thin, color=teal];
\draw (133) -- (8) [dashed, very thin, color=teal];
\draw (134) -- (11) [dashed, very thin, color=teal];
\draw (135) -- (18) [dashed, very thin, color=teal];
\draw (136) -- (1) [dashed, very thin, color=teal];
\draw (137) -- (90) [dashed, very thin, color=teal];
\draw (138) -- (5) [dashed, very thin, color=teal];
\draw (139) -- (0) [dashed, very thin, color=teal];
\draw (141) -- (1) [dashed, very thin, color=teal];
\draw (142) -- (199) [dashed, very thin, color=teal];
\draw (143) -- (41) [dashed, very thin, color=teal];
\draw (144) -- (147) [dashed, very thin, color=teal];
\draw (145) -- (18) [dashed, very thin, color=teal];
\draw (148) -- (18) [dashed, very thin, color=teal];
\draw (149) -- (4) [dashed, very thin, color=teal];
\draw (150) -- (8) [dashed, very thin, color=teal];
\draw (151) -- (0) [dashed, very thin, color=teal];
\draw (153) -- (4) [dashed, very thin, color=teal];
\draw (154) -- (5) [dashed, very thin, color=teal];
\draw (155) -- (103) [dashed, very thin, color=teal];
\draw (156) -- (199) [dashed, very thin, color=teal];
\draw (159) -- (6) [dashed, very thin, color=teal];
\draw (160) -- (33) [dashed, very thin, color=teal];
\draw (162) -- (161) [dashed, very thin, color=teal];
\draw (164) -- (11) [dashed, very thin, color=teal];
\draw (165) -- (11) [dashed, very thin, color=teal];
\draw (167) -- (4) [dashed, very thin, color=teal];
\draw (168) -- (0) [dashed, very thin, color=teal];
\draw (169) -- (9) [dashed, very thin, color=teal];
\draw (170) -- (11) [dashed, very thin, color=teal];
\draw (171) -- (2) [dashed, very thin, color=teal];
\draw (172) -- (2) [dashed, very thin, color=teal];
\draw (173) -- (2) [dashed, very thin, color=teal];
\draw (174) -- (103) [dashed, very thin, color=teal];
\draw (175) -- (0) [dashed, very thin, color=teal];
\draw (177) -- (8) [dashed, very thin, color=teal];
\draw (179) -- (4) [dashed, very thin, color=teal];
\draw (181) -- (9) [dashed, very thin, color=teal];
\draw (182) -- (33) [dashed, very thin, color=teal];
\draw (183) -- (1) [dashed, very thin, color=teal];
\draw (185) -- (90) [dashed, very thin, color=teal];
\draw (186) -- (1) [dashed, very thin, color=teal];
\draw (188) -- (0) [dashed, very thin, color=teal];
\draw (189) -- (9) [dashed, very thin, color=teal];
\draw (190) -- (194) [dashed, very thin, color=teal];
\draw (195) -- (0) [dashed, very thin, color=teal];
\draw (196) -- (1) [dashed, very thin, color=teal];
\draw (197) -- (1) [dashed, very thin, color=teal];
\draw (198) -- (1) [dashed, very thin, color=teal];

\end{tikzpicture}
\begin{center}An optimal tree with $L=173.00$.\end{center}
\end{adjustbox}
\caption{Optimal solutions (tour subgraph on the left and its tree counterpart on the right) of a further instance generated from p5, with $r_{v}=16.46$ and $c_{v}=10$ for all $v\in V$.}
\label{final_fig_21}
\end{figure}
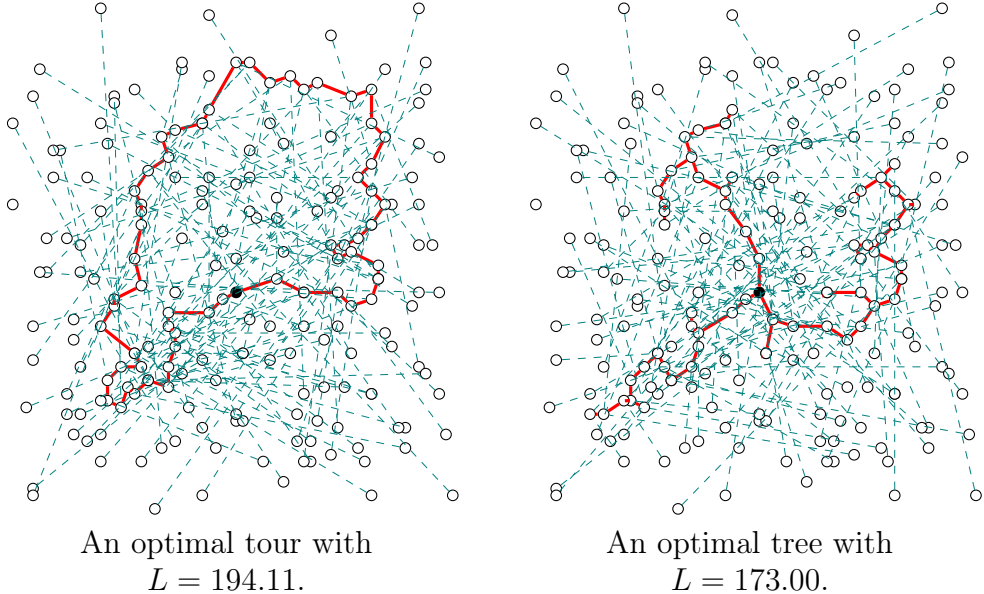
\begin{figure}[t]
\centering
\begin{adjustbox}{valign=t,minipage={.45\textwidth}}
\centering
\begin{tikzpicture}[scale=0.09]
\tikzstyle{every node}=[draw,circle,fill=white,minimum size=4pt,
                            inner sep=0pt]
\draw (35, 35) node (0) [overlay, color=black] {};
\draw (22, 22) node (1) [overlay] {};
\draw (36, 26) node (2) [overlay] {};
\draw (21, 45) node (3) [overlay] {};
\draw (45, 35) node (4) [overlay] {};
\draw (55, 20) node (5) [overlay] {};
\draw (33, 34) node (6) [overlay] {};
\draw (50, 50) node (7) [overlay] {};
\draw (55, 45) node (8) [overlay] {};
\draw (26, 59) node (9) [overlay] {};
\draw (40, 66) node (10) [overlay] {};
\draw (55, 65) node (11) [overlay] {};
\draw (35, 51) node (12) [overlay] {};
\draw (62, 35) node (13) [overlay] {};
\draw (62, 57) node (14) [overlay] {};
\draw (62, 24) node (15) [overlay] {};
\draw (21, 36) node (16) [overlay] {};
\draw (33, 44) node (17) [overlay] {};
\draw (9, 56) node (18) [overlay] {};
\draw (62, 48) node (19) [overlay] {};
\draw (66, 14) node (20) [overlay] {};
\draw (44, 13) node (21) [overlay] {};
\draw (26, 13) node (22) [overlay] {};
\draw (11, 28) node (23) [overlay] {};
\draw (7, 43) node (24) [overlay] {};
\draw (17, 64) node (25) [overlay] {};
\draw (41, 46) node (26) [overlay] {};
\draw (55, 34) node (27) [overlay] {};
\draw (35, 16) node (28) [overlay] {};
\draw (52, 26) node (29) [overlay] {};
\draw (43, 26) node (30) [overlay] {};
\draw (31, 76) node (31) [overlay] {};
\draw (22, 53) node (32) [overlay] {};
\draw (26, 29) node (33) [overlay] {};
\draw (50, 40) node (34) [overlay] {};
\draw (55, 50) node (35) [overlay] {};
\draw (54, 10) node (36) [overlay] {};
\draw (60, 15) node (37) [overlay] {};
\draw (47, 66) node (38) [overlay] {};
\draw (30, 60) node (39) [overlay] {};
\draw (30, 50) node (40) [overlay] {};
\draw (12, 17) node (41) [overlay] {};
\draw (15, 14) node (42) [overlay] {};
\draw (16, 19) node (43) [overlay] {};
\draw (21, 48) node (44) [overlay] {};
\draw (50, 30) node (45) [overlay] {};
\draw (51, 42) node (46) [overlay] {};
\draw (50, 15) node (47) [overlay] {};
\draw (48, 21) node (48) [overlay] {};
\draw (12, 38) node (49) [overlay] {};
\draw (37, 52) node (50) [overlay] {};
\draw (49, 49) node (51) [overlay] {};
\draw (52, 64) node (52) [overlay] {};
\draw (20, 26) node (53) [overlay] {};
\draw (40, 30) node (54) [overlay] {};
\draw (21, 47) node (55) [overlay] {};
\draw (17, 63) node (56) [overlay] {};
\draw (31, 62) node (57) [overlay] {};
\draw (52, 33) node (58) [overlay] {};
\draw (51, 21) node (59) [overlay] {};
\draw (42, 41) node (60) [overlay] {};
\draw (31, 32) node (61) [overlay] {};
\draw (5, 25) node (62) [overlay] {};
\draw (12, 42) node (63) [overlay] {};
\draw (36, 16) node (64) [overlay] {};
\draw (52, 41) node (65) [overlay] {};
\draw (27, 23) node (66) [overlay] {};
\draw (17, 33) node (67) [overlay] {};
\draw (13, 13) node (68) [overlay] {};
\draw (57, 58) node (69) [overlay] {};
\draw (62, 42) node (70) [overlay] {};
\draw (42, 57) node (71) [overlay] {};
\draw (16, 57) node (72) [overlay] {};
\draw (8, 52) node (73) [overlay] {};
\draw (7, 38) node (74) [overlay] {};
\draw (27, 68) node (75) [overlay] {};
\draw (30, 48) node (76) [overlay] {};
\draw (43, 67) node (77) [overlay] {};
\draw (58, 48) node (78) [overlay] {};
\draw (58, 27) node (79) [overlay] {};
\draw (37, 69) node (80) [overlay] {};
\draw (38, 46) node (81) [overlay] {};
\draw (46, 10) node (82) [overlay] {};
\draw (61, 33) node (83) [overlay] {};
\draw (62, 63) node (84) [overlay] {};
\draw (63, 69) node (85) [overlay] {};
\draw (32, 22) node (86) [overlay] {};
\draw (45, 35) node (87) [overlay] {};
\draw (59, 15) node (88) [overlay] {};
\draw (5, 6) node (89) [overlay] {};
\draw (10, 17) node (90) [overlay] {};
\draw (21, 10) node (91) [overlay] {};
\draw (5, 64) node (92) [overlay] {};
\draw (30, 15) node (93) [overlay] {};
\draw (39, 10) node (94) [overlay] {};
\draw (32, 39) node (95) [overlay] {};
\draw (25, 32) node (96) [overlay] {};
\draw (25, 55) node (97) [overlay] {};
\draw (48, 28) node (98) [overlay] {};
\draw (56, 37) node (99) [overlay] {};
\draw (41, 49) node (100) [overlay] {};
\draw (35, 17) node (101) [overlay] {};
\draw (55, 45) node (102) [overlay] {};
\draw (55, 20) node (103) [overlay] {};
\draw (15, 30) node (104) [overlay] {};
\draw (25, 30) node (105) [overlay] {};
\draw (20, 50) node (106) [overlay] {};
\draw (10, 43) node (107) [overlay] {};
\draw (55, 60) node (108) [overlay] {};
\draw (30, 60) node (109) [overlay] {};
\draw (20, 65) node (110) [overlay] {};
\draw (50, 35) node (111) [overlay] {};
\draw (30, 25) node (112) [overlay] {};
\draw (15, 10) node (113) [overlay] {};
\draw (30, 5) node (114) [overlay] {};
\draw (10, 20) node (115) [overlay] {};
\draw (5, 30) node (116) [overlay] {};
\draw (20, 40) node (117) [overlay] {};
\draw (15, 60) node (118) [overlay] {};
\draw (45, 65) node (119) [overlay] {};
\draw (45, 20) node (120) [overlay] {};
\draw (45, 10) node (121) [overlay] {};
\draw (55, 5) node (122) [overlay] {};
\draw (65, 35) node (123) [overlay] {};
\draw (65, 20) node (124) [overlay] {};
\draw (45, 30) node (125) [overlay] {};
\draw (35, 40) node (126) [overlay] {};
\draw (41, 37) node (127) [overlay] {};
\draw (64, 42) node (128) [overlay] {};
\draw (40, 60) node (129) [overlay] {};
\draw (31, 52) node (130) [overlay] {};
\draw (35, 69) node (131) [overlay] {};
\draw (53, 52) node (132) [overlay] {};
\draw (65, 55) node (133) [overlay] {};
\draw (63, 65) node (134) [overlay] {};
\draw (2, 60) node (135) [overlay] {};
\draw (20, 20) node (136) [overlay] {};
\draw (5, 5) node (137) [overlay] {};
\draw (60, 12) node (138) [overlay] {};
\draw (40, 25) node (139) [overlay] {};
\draw (42, 7) node (140) [overlay] {};
\draw (24, 12) node (141) [overlay] {};
\draw (23, 3) node (142) [overlay] {};
\draw (11, 14) node (143) [overlay] {};
\draw (6, 38) node (144) [overlay] {};
\draw (2, 48) node (145) [overlay] {};
\draw (8, 56) node (146) [overlay] {};
\draw (13, 52) node (147) [overlay] {};
\draw (6, 68) node (148) [overlay] {};
\draw (47, 47) node (149) [overlay] {};
\draw (49, 58) node (150) [overlay] {};
\draw (27, 43) node (151) [overlay] {};
\draw (37, 31) node (152) [overlay] {};
\draw (57, 29) node (153) [overlay] {};
\draw (63, 23) node (154) [overlay] {};
\draw (53, 12) node (155) [overlay] {};
\draw (32, 12) node (156) [overlay] {};
\draw (36, 26) node (157) [overlay] {};
\draw (21, 24) node (158) [overlay] {};
\draw (17, 34) node (159) [overlay] {};
\draw (12, 24) node (160) [overlay] {};
\draw (24, 58) node (161) [overlay] {};
\draw (27, 69) node (162) [overlay] {};
\draw (15, 77) node (163) [overlay] {};
\draw (62, 77) node (164) [overlay] {};
\draw (49, 73) node (165) [overlay] {};
\draw (67, 5) node (166) [overlay] {};
\draw (56, 39) node (167) [overlay] {};
\draw (37, 47) node (168) [overlay] {};
\draw (37, 56) node (169) [overlay] {};
\draw (57, 68) node (170) [overlay] {};
\draw (47, 16) node (171) [overlay] {};
\draw (44, 17) node (172) [overlay] {};
\draw (46, 13) node (173) [overlay] {};
\draw (49, 11) node (174) [overlay] {};
\draw (49, 42) node (175) [overlay] {};
\draw (53, 43) node (176) [overlay] {};
\draw (61, 52) node (177) [overlay] {};
\draw (57, 48) node (178) [overlay] {};
\draw (56, 37) node (179) [overlay] {};
\draw (55, 54) node (180) [overlay] {};
\draw (15, 47) node (181) [overlay] {};
\draw (14, 37) node (182) [overlay] {};
\draw (11, 31) node (183) [overlay] {};
\draw (16, 22) node (184) [overlay] {};
\draw (4, 18) node (185) [overlay] {};
\draw (28, 18) node (186) [overlay] {};
\draw (26, 52) node (187) [overlay] {};
\draw (26, 35) node (188) [overlay] {};
\draw (31, 67) node (189) [overlay] {};
\draw (15, 19) node (190) [overlay] {};
\draw (22, 22) node (191) [overlay] {};
\draw (18, 24) node (192) [overlay] {};
\draw (26, 27) node (193) [overlay] {};
\draw (25, 24) node (194) [overlay] {};
\draw (22, 27) node (195) [overlay] {};
\draw (25, 21) node (196) [overlay] {};
\draw (19, 21) node (197) [overlay] {};
\draw (20, 26) node (198) [overlay] {};
\draw (18, 18) node (199) [overlay] {};
\draw (0) -- (6) [very thick, color=red];
\draw (0) -- (127) [very thick, color=red];
\draw (1) -- (191) [very thick, color=red];
\draw (1) -- (196) [very thick, color=red];
\draw (3) -- (55) [very thick, color=red];
\draw (3) -- (117) [very thick, color=red];
\draw (4) -- (87) [very thick, color=red];
\draw (4) -- (111) [very thick, color=red];
\draw (6) -- (61) [very thick, color=red];
\draw (8) -- (102) [very thick, color=red];
\draw (8) -- (178) [very thick, color=red];
\draw (9) -- (39) [very thick, color=red];
\draw (9) -- (161) [very thick, color=red];
\draw (10) -- (77) [very thick, color=red];
\draw (10) -- (80) [very thick, color=red];
\draw (11) -- (52) [very thick, color=red];
\draw (11) -- (108) [very thick, color=red];
\draw (16) -- (117) [very thick, color=red];
\draw (16) -- (159) [very thick, color=red];
\draw (27) -- (58) [very thick, color=red];
\draw (27) -- (99) [very thick, color=red];
\draw (32) -- (97) [very thick, color=red];
\draw (32) -- (106) [very thick, color=red];
\draw (33) -- (105) [very thick, color=red];
\draw (33) -- (193) [very thick, color=red];
\draw (34) -- (65) [very thick, color=red];
\draw (34) -- (175) [very thick, color=red];
\draw (35) -- (132) [very thick, color=red];
\draw (35) -- (178) [very thick, color=red];
\draw (38) -- (52) [very thick, color=red];
\draw (38) -- (119) [very thick, color=red];
\draw (39) -- (109) [very thick, color=red];
\draw (43) -- (184) [very thick, color=red];
\draw (43) -- (199) [very thick, color=red];
\draw (44) -- (55) [very thick, color=red];
\draw (44) -- (106) [very thick, color=red];
\draw (46) -- (175) [very thick, color=red];
\draw (46) -- (176) [very thick, color=red];
\draw (53) -- (158) [very thick, color=red];
\draw (53) -- (198) [very thick, color=red];
\draw (57) -- (109) [very thick, color=red];
\draw (57) -- (131) [very thick, color=red];
\draw (58) -- (111) [very thick, color=red];
\draw (61) -- (96) [very thick, color=red];
\draw (65) -- (167) [very thick, color=red];
\draw (67) -- (104) [very thick, color=red];
\draw (67) -- (159) [very thick, color=red];
\draw (69) -- (108) [very thick, color=red];
\draw (69) -- (180) [very thick, color=red];
\draw (77) -- (119) [very thick, color=red];
\draw (80) -- (131) [very thick, color=red];
\draw (87) -- (127) [very thick, color=red];
\draw (96) -- (105) [very thick, color=red];
\draw (97) -- (161) [very thick, color=red];
\draw (99) -- (179) [very thick, color=red];
\draw (102) -- (176) [very thick, color=red];
\draw (104) -- (198) [very thick, color=red];
\draw (132) -- (180) [very thick, color=red];
\draw (136) -- (191) [very thick, color=red];
\draw (136) -- (197) [very thick, color=red];
\draw (158) -- (192) [very thick, color=red];
\draw (167) -- (179) [very thick, color=red];
\draw (184) -- (192) [very thick, color=red];
\draw (193) -- (194) [very thick, color=red];
\draw (194) -- (196) [very thick, color=red];
\draw (197) -- (199) [very thick, color=red];
\draw (2) -- (198) [dashed, very thin, color=teal];
\draw (5) -- (105) [dashed, very thin, color=teal];
\draw (7) -- (57) [dashed, very thin, color=teal];
\draw (12) -- (108) [dashed, very thin, color=teal];
\draw (13) -- (65) [dashed, very thin, color=teal];
\draw (14) -- (34) [dashed, very thin, color=teal];
\draw (15) -- (178) [dashed, very thin, color=teal];
\draw (17) -- (43) [dashed, very thin, color=teal];
\draw (18) -- (184) [dashed, very thin, color=teal];
\draw (19) -- (105) [dashed, very thin, color=teal];
\draw (20) -- (58) [dashed, very thin, color=teal];
\draw (21) -- (8) [dashed, very thin, color=teal];
\draw (22) -- (104) [dashed, very thin, color=teal];
\draw (23) -- (38) [dashed, very thin, color=teal];
\draw (24) -- (77) [dashed, very thin, color=teal];
\draw (25) -- (199) [dashed, very thin, color=teal];
\draw (26) -- (158) [dashed, very thin, color=teal];
\draw (28) -- (97) [dashed, very thin, color=teal];
\draw (29) -- (57) [dashed, very thin, color=teal];
\draw (30) -- (196) [dashed, very thin, color=teal];
\draw (31) -- (179) [dashed, very thin, color=teal];
\draw (36) -- (194) [dashed, very thin, color=teal];
\draw (37) -- (198) [dashed, very thin, color=teal];
\draw (40) -- (0) [dashed, very thin, color=teal];
\draw (41) -- (27) [dashed, very thin, color=teal];
\draw (42) -- (108) [dashed, very thin, color=teal];
\draw (45) -- (194) [dashed, very thin, color=teal];
\draw (47) -- (38) [dashed, very thin, color=teal];
\draw (48) -- (33) [dashed, very thin, color=teal];
\draw (49) -- (176) [dashed, very thin, color=teal];
\draw (50) -- (158) [dashed, very thin, color=teal];
\draw (51) -- (9) [dashed, very thin, color=teal];
\draw (54) -- (32) [dashed, very thin, color=teal];
\draw (56) -- (55) [dashed, very thin, color=teal];
\draw (59) -- (1) [dashed, very thin, color=teal];
\draw (60) -- (167) [dashed, very thin, color=teal];
\draw (62) -- (109) [dashed, very thin, color=teal];
\draw (63) -- (80) [dashed, very thin, color=teal];
\draw (64) -- (192) [dashed, very thin, color=teal];
\draw (66) -- (10) [dashed, very thin, color=teal];
\draw (68) -- (6) [dashed, very thin, color=teal];
\draw (70) -- (119) [dashed, very thin, color=teal];
\draw (71) -- (131) [dashed, very thin, color=teal];
\draw (72) -- (179) [dashed, very thin, color=teal];
\draw (73) -- (99) [dashed, very thin, color=teal];
\draw (74) -- (80) [dashed, very thin, color=teal];
\draw (75) -- (161) [dashed, very thin, color=teal];
\draw (76) -- (117) [dashed, very thin, color=teal];
\draw (78) -- (193) [dashed, very thin, color=teal];
\draw (79) -- (111) [dashed, very thin, color=teal];
\draw (81) -- (167) [dashed, very thin, color=teal];
\draw (82) -- (61) [dashed, very thin, color=teal];
\draw (83) -- (97) [dashed, very thin, color=teal];
\draw (84) -- (184) [dashed, very thin, color=teal];
\draw (85) -- (34) [dashed, very thin, color=teal];
\draw (86) -- (10) [dashed, very thin, color=teal];
\draw (88) -- (4) [dashed, very thin, color=teal];
\draw (89) -- (175) [dashed, very thin, color=teal];
\draw (90) -- (87) [dashed, very thin, color=teal];
\draw (91) -- (132) [dashed, very thin, color=teal];
\draw (92) -- (1) [dashed, very thin, color=teal];
\draw (93) -- (180) [dashed, very thin, color=teal];
\draw (94) -- (43) [dashed, very thin, color=teal];
\draw (95) -- (39) [dashed, very thin, color=teal];
\draw (98) -- (197) [dashed, very thin, color=teal];
\draw (100) -- (127) [dashed, very thin, color=teal];
\draw (101) -- (46) [dashed, very thin, color=teal];
\draw (103) -- (104) [dashed, very thin, color=teal];
\draw (107) -- (69) [dashed, very thin, color=teal];
\draw (110) -- (9) [dashed, very thin, color=teal];
\draw (112) -- (199) [dashed, very thin, color=teal];
\draw (113) -- (69) [dashed, very thin, color=teal];
\draw (114) -- (178) [dashed, very thin, color=teal];
\draw (115) -- (3) [dashed, very thin, color=teal];
\draw (116) -- (8) [dashed, very thin, color=teal];
\draw (118) -- (102) [dashed, very thin, color=teal];
\draw (120) -- (77) [dashed, very thin, color=teal];
\draw (121) -- (52) [dashed, very thin, color=teal];
\draw (122) -- (53) [dashed, very thin, color=teal];
\draw (123) -- (16) [dashed, very thin, color=teal];
\draw (124) -- (161) [dashed, very thin, color=teal];
\draw (125) -- (6) [dashed, very thin, color=teal];
\draw (126) -- (35) [dashed, very thin, color=teal];
\draw (128) -- (11) [dashed, very thin, color=teal];
\draw (129) -- (175) [dashed, very thin, color=teal];
\draw (130) -- (61) [dashed, very thin, color=teal];
\draw (133) -- (131) [dashed, very thin, color=teal];
\draw (134) -- (55) [dashed, very thin, color=teal];
\draw (135) -- (191) [dashed, very thin, color=teal];
\draw (137) -- (3) [dashed, very thin, color=teal];
\draw (138) -- (136) [dashed, very thin, color=teal];
\draw (139) -- (65) [dashed, very thin, color=teal];
\draw (140) -- (96) [dashed, very thin, color=teal];
\draw (141) -- (44) [dashed, very thin, color=teal];
\draw (142) -- (46) [dashed, very thin, color=teal];
\draw (143) -- (0) [dashed, very thin, color=teal];
\draw (144) -- (27) [dashed, very thin, color=teal];
\draw (146) -- (52) [dashed, very thin, color=teal];
\draw (147) -- (58) [dashed, very thin, color=teal];
\draw (148) -- (99) [dashed, very thin, color=teal];
\draw (149) -- (39) [dashed, very thin, color=teal];
\draw (150) -- (11) [dashed, very thin, color=teal];
\draw (151) -- (197) [dashed, very thin, color=teal];
\draw (152) -- (96) [dashed, very thin, color=teal];
\draw (153) -- (33) [dashed, very thin, color=teal];
\draw (154) -- (191) [dashed, very thin, color=teal];
\draw (155) -- (16) [dashed, very thin, color=teal];
\draw (156) -- (67) [dashed, very thin, color=teal];
\draw (157) -- (119) [dashed, very thin, color=teal];
\draw (160) -- (109) [dashed, very thin, color=teal];
\draw (162) -- (53) [dashed, very thin, color=teal];
\draw (163) -- (192) [dashed, very thin, color=teal];
\draw (164) -- (176) [dashed, very thin, color=teal];
\draw (165) -- (180) [dashed, very thin, color=teal];
\draw (166) -- (32) [dashed, very thin, color=teal];
\draw (168) -- (136) [dashed, very thin, color=teal];
\draw (169) -- (102) [dashed, very thin, color=teal];
\draw (170) -- (196) [dashed, very thin, color=teal];
\draw (171) -- (111) [dashed, very thin, color=teal];
\draw (172) -- (193) [dashed, very thin, color=teal];
\draw (173) -- (159) [dashed, very thin, color=teal];
\draw (174) -- (35) [dashed, very thin, color=teal];
\draw (177) -- (117) [dashed, very thin, color=teal];
\draw (181) -- (67) [dashed, very thin, color=teal];
\draw (182) -- (106) [dashed, very thin, color=teal];
\draw (183) -- (127) [dashed, very thin, color=teal];
\draw (185) -- (87) [dashed, very thin, color=teal];
\draw (186) -- (106) [dashed, very thin, color=teal];
\draw (187) -- (132) [dashed, very thin, color=teal];
\draw (188) -- (44) [dashed, very thin, color=teal];
\draw (189) -- (159) [dashed, very thin, color=teal];
\draw (195) -- (4) [dashed, very thin, color=teal];

\end{tikzpicture}
\begin{center}An optimal tour with $L=194.11$.\end{center}
\end{adjustbox}\hfill
\begin{adjustbox}{valign=t,minipage={.45\textwidth}}
\centering
\begin{tikzpicture}[scale=0.09]
\tikzstyle{every node}=[draw,circle,fill=white,minimum size=4pt,
                            inner sep=0pt]
\draw (35, 35) node (0) [overlay, color=black] {};
\draw (22, 22) node (1) [overlay] {};
\draw (36, 26) node (2) [overlay] {};
\draw (21, 45) node (3) [overlay] {};
\draw (45, 35) node (4) [overlay] {};
\draw (55, 20) node (5) [overlay] {};
\draw (33, 34) node (6) [overlay] {};
\draw (50, 50) node (7) [overlay] {};
\draw (55, 45) node (8) [overlay] {};
\draw (26, 59) node (9) [overlay] {};
\draw (40, 66) node (10) [overlay] {};
\draw (55, 65) node (11) [overlay] {};
\draw (35, 51) node (12) [overlay] {};
\draw (62, 35) node (13) [overlay] {};
\draw (62, 57) node (14) [overlay] {};
\draw (62, 24) node (15) [overlay] {};
\draw (21, 36) node (16) [overlay] {};
\draw (33, 44) node (17) [overlay] {};
\draw (9, 56) node (18) [overlay] {};
\draw (62, 48) node (19) [overlay] {};
\draw (66, 14) node (20) [overlay] {};
\draw (44, 13) node (21) [overlay] {};
\draw (26, 13) node (22) [overlay] {};
\draw (11, 28) node (23) [overlay] {};
\draw (7, 43) node (24) [overlay] {};
\draw (17, 64) node (25) [overlay] {};
\draw (41, 46) node (26) [overlay] {};
\draw (55, 34) node (27) [overlay] {};
\draw (35, 16) node (28) [overlay] {};
\draw (52, 26) node (29) [overlay] {};
\draw (43, 26) node (30) [overlay] {};
\draw (31, 76) node (31) [overlay] {};
\draw (22, 53) node (32) [overlay] {};
\draw (26, 29) node (33) [overlay] {};
\draw (50, 40) node (34) [overlay] {};
\draw (55, 50) node (35) [overlay] {};
\draw (54, 10) node (36) [overlay] {};
\draw (60, 15) node (37) [overlay] {};
\draw (47, 66) node (38) [overlay] {};
\draw (30, 60) node (39) [overlay] {};
\draw (30, 50) node (40) [overlay] {};
\draw (12, 17) node (41) [overlay] {};
\draw (15, 14) node (42) [overlay] {};
\draw (16, 19) node (43) [overlay] {};
\draw (21, 48) node (44) [overlay] {};
\draw (50, 30) node (45) [overlay] {};
\draw (51, 42) node (46) [overlay] {};
\draw (50, 15) node (47) [overlay] {};
\draw (48, 21) node (48) [overlay] {};
\draw (12, 38) node (49) [overlay] {};
\draw (37, 52) node (50) [overlay] {};
\draw (49, 49) node (51) [overlay] {};
\draw (52, 64) node (52) [overlay] {};
\draw (20, 26) node (53) [overlay] {};
\draw (40, 30) node (54) [overlay] {};
\draw (21, 47) node (55) [overlay] {};
\draw (17, 63) node (56) [overlay] {};
\draw (31, 62) node (57) [overlay] {};
\draw (52, 33) node (58) [overlay] {};
\draw (51, 21) node (59) [overlay] {};
\draw (42, 41) node (60) [overlay] {};
\draw (31, 32) node (61) [overlay] {};
\draw (5, 25) node (62) [overlay] {};
\draw (12, 42) node (63) [overlay] {};
\draw (36, 16) node (64) [overlay] {};
\draw (52, 41) node (65) [overlay] {};
\draw (27, 23) node (66) [overlay] {};
\draw (17, 33) node (67) [overlay] {};
\draw (13, 13) node (68) [overlay] {};
\draw (57, 58) node (69) [overlay] {};
\draw (62, 42) node (70) [overlay] {};
\draw (42, 57) node (71) [overlay] {};
\draw (16, 57) node (72) [overlay] {};
\draw (8, 52) node (73) [overlay] {};
\draw (7, 38) node (74) [overlay] {};
\draw (27, 68) node (75) [overlay] {};
\draw (30, 48) node (76) [overlay] {};
\draw (43, 67) node (77) [overlay] {};
\draw (58, 48) node (78) [overlay] {};
\draw (58, 27) node (79) [overlay] {};
\draw (37, 69) node (80) [overlay] {};
\draw (38, 46) node (81) [overlay] {};
\draw (46, 10) node (82) [overlay] {};
\draw (61, 33) node (83) [overlay] {};
\draw (62, 63) node (84) [overlay] {};
\draw (63, 69) node (85) [overlay] {};
\draw (32, 22) node (86) [overlay] {};
\draw (45, 35) node (87) [overlay] {};
\draw (59, 15) node (88) [overlay] {};
\draw (5, 6) node (89) [overlay] {};
\draw (10, 17) node (90) [overlay] {};
\draw (21, 10) node (91) [overlay] {};
\draw (5, 64) node (92) [overlay] {};
\draw (30, 15) node (93) [overlay] {};
\draw (39, 10) node (94) [overlay] {};
\draw (32, 39) node (95) [overlay] {};
\draw (25, 32) node (96) [overlay] {};
\draw (25, 55) node (97) [overlay] {};
\draw (48, 28) node (98) [overlay] {};
\draw (56, 37) node (99) [overlay] {};
\draw (41, 49) node (100) [overlay] {};
\draw (35, 17) node (101) [overlay] {};
\draw (55, 45) node (102) [overlay] {};
\draw (55, 20) node (103) [overlay] {};
\draw (15, 30) node (104) [overlay] {};
\draw (25, 30) node (105) [overlay] {};
\draw (20, 50) node (106) [overlay] {};
\draw (10, 43) node (107) [overlay] {};
\draw (55, 60) node (108) [overlay] {};
\draw (30, 60) node (109) [overlay] {};
\draw (20, 65) node (110) [overlay] {};
\draw (50, 35) node (111) [overlay] {};
\draw (30, 25) node (112) [overlay] {};
\draw (15, 10) node (113) [overlay] {};
\draw (30, 5) node (114) [overlay] {};
\draw (10, 20) node (115) [overlay] {};
\draw (5, 30) node (116) [overlay] {};
\draw (20, 40) node (117) [overlay] {};
\draw (15, 60) node (118) [overlay] {};
\draw (45, 65) node (119) [overlay] {};
\draw (45, 20) node (120) [overlay] {};
\draw (45, 10) node (121) [overlay] {};
\draw (55, 5) node (122) [overlay] {};
\draw (65, 35) node (123) [overlay] {};
\draw (65, 20) node (124) [overlay] {};
\draw (45, 30) node (125) [overlay] {};
\draw (35, 40) node (126) [overlay] {};
\draw (41, 37) node (127) [overlay] {};
\draw (64, 42) node (128) [overlay] {};
\draw (40, 60) node (129) [overlay] {};
\draw (31, 52) node (130) [overlay] {};
\draw (35, 69) node (131) [overlay] {};
\draw (53, 52) node (132) [overlay] {};
\draw (65, 55) node (133) [overlay] {};
\draw (63, 65) node (134) [overlay] {};
\draw (2, 60) node (135) [overlay] {};
\draw (20, 20) node (136) [overlay] {};
\draw (5, 5) node (137) [overlay] {};
\draw (60, 12) node (138) [overlay] {};
\draw (40, 25) node (139) [overlay] {};
\draw (42, 7) node (140) [overlay] {};
\draw (24, 12) node (141) [overlay] {};
\draw (23, 3) node (142) [overlay] {};
\draw (11, 14) node (143) [overlay] {};
\draw (6, 38) node (144) [overlay] {};
\draw (2, 48) node (145) [overlay] {};
\draw (8, 56) node (146) [overlay] {};
\draw (13, 52) node (147) [overlay] {};
\draw (6, 68) node (148) [overlay] {};
\draw (47, 47) node (149) [overlay] {};
\draw (49, 58) node (150) [overlay] {};
\draw (27, 43) node (151) [overlay] {};
\draw (37, 31) node (152) [overlay] {};
\draw (57, 29) node (153) [overlay] {};
\draw (63, 23) node (154) [overlay] {};
\draw (53, 12) node (155) [overlay] {};
\draw (32, 12) node (156) [overlay] {};
\draw (36, 26) node (157) [overlay] {};
\draw (21, 24) node (158) [overlay] {};
\draw (17, 34) node (159) [overlay] {};
\draw (12, 24) node (160) [overlay] {};
\draw (24, 58) node (161) [overlay] {};
\draw (27, 69) node (162) [overlay] {};
\draw (15, 77) node (163) [overlay] {};
\draw (62, 77) node (164) [overlay] {};
\draw (49, 73) node (165) [overlay] {};
\draw (67, 5) node (166) [overlay] {};
\draw (56, 39) node (167) [overlay] {};
\draw (37, 47) node (168) [overlay] {};
\draw (37, 56) node (169) [overlay] {};
\draw (57, 68) node (170) [overlay] {};
\draw (47, 16) node (171) [overlay] {};
\draw (44, 17) node (172) [overlay] {};
\draw (46, 13) node (173) [overlay] {};
\draw (49, 11) node (174) [overlay] {};
\draw (49, 42) node (175) [overlay] {};
\draw (53, 43) node (176) [overlay] {};
\draw (61, 52) node (177) [overlay] {};
\draw (57, 48) node (178) [overlay] {};
\draw (56, 37) node (179) [overlay] {};
\draw (55, 54) node (180) [overlay] {};
\draw (15, 47) node (181) [overlay] {};
\draw (14, 37) node (182) [overlay] {};
\draw (11, 31) node (183) [overlay] {};
\draw (16, 22) node (184) [overlay] {};
\draw (4, 18) node (185) [overlay] {};
\draw (28, 18) node (186) [overlay] {};
\draw (26, 52) node (187) [overlay] {};
\draw (26, 35) node (188) [overlay] {};
\draw (31, 67) node (189) [overlay] {};
\draw (15, 19) node (190) [overlay] {};
\draw (22, 22) node (191) [overlay] {};
\draw (18, 24) node (192) [overlay] {};
\draw (26, 27) node (193) [overlay] {};
\draw (25, 24) node (194) [overlay] {};
\draw (22, 27) node (195) [overlay] {};
\draw (25, 21) node (196) [overlay] {};
\draw (19, 21) node (197) [overlay] {};
\draw (20, 26) node (198) [overlay] {};
\draw (18, 18) node (199) [overlay] {};
\draw (0) -- (6) [very thick, color=red];
\draw (0) -- (126) [very thick, color=red];
\draw (0) -- (152) [very thick, color=red];
\draw (1) -- (191) [very thick, color=red];
\draw (2) -- (157) [very thick, color=red];
\draw (3) -- (55) [very thick, color=red];
\draw (4) -- (87) [very thick, color=red];
\draw (6) -- (61) [very thick, color=red];
\draw (7) -- (51) [very thick, color=red];
\draw (7) -- (132) [very thick, color=red];
\draw (8) -- (102) [very thick, color=red];
\draw (9) -- (109) [very thick, color=red];
\draw (9) -- (161) [very thick, color=red];
\draw (17) -- (76) [very thick, color=red];
\draw (17) -- (126) [very thick, color=red];
\draw (27) -- (58) [very thick, color=red];
\draw (27) -- (179) [very thick, color=red];
\draw (32) -- (97) [very thick, color=red];
\draw (32) -- (106) [very thick, color=red];
\draw (33) -- (61) [very thick, color=red];
\draw (33) -- (105) [very thick, color=red];
\draw (33) -- (193) [very thick, color=red];
\draw (34) -- (175) [very thick, color=red];
\draw (35) -- (132) [very thick, color=red];
\draw (35) -- (178) [very thick, color=red];
\draw (39) -- (109) [very thick, color=red];
\draw (40) -- (76) [very thick, color=red];
\draw (40) -- (130) [very thick, color=red];
\draw (40) -- (187) [very thick, color=red];
\draw (41) -- (90) [very thick, color=red];
\draw (41) -- (190) [very thick, color=red];
\draw (43) -- (184) [very thick, color=red];
\draw (43) -- (190) [very thick, color=red];
\draw (43) -- (199) [very thick, color=red];
\draw (44) -- (55) [very thick, color=red];
\draw (44) -- (106) [very thick, color=red];
\draw (45) -- (58) [very thick, color=red];
\draw (45) -- (98) [very thick, color=red];
\draw (46) -- (65) [very thick, color=red];
\draw (46) -- (175) [very thick, color=red];
\draw (53) -- (198) [very thick, color=red];
\draw (54) -- (125) [very thick, color=red];
\draw (54) -- (152) [very thick, color=red];
\draw (57) -- (109) [very thick, color=red];
\draw (58) -- (111) [very thick, color=red];
\draw (65) -- (167) [very thick, color=red];
\draw (65) -- (176) [very thick, color=red];
\draw (78) -- (178) [very thick, color=red];
\draw (87) -- (111) [very thick, color=red];
\draw (96) -- (105) [very thick, color=red];
\draw (97) -- (161) [very thick, color=red];
\draw (97) -- (187) [very thick, color=red];
\draw (98) -- (125) [very thick, color=red];
\draw (99) -- (179) [very thick, color=red];
\draw (102) -- (176) [very thick, color=red];
\draw (102) -- (178) [very thick, color=red];
\draw (132) -- (180) [very thick, color=red];
\draw (152) -- (157) [very thick, color=red];
\draw (158) -- (191) [very thick, color=red];
\draw (158) -- (198) [very thick, color=red];
\draw (167) -- (179) [very thick, color=red];
\draw (184) -- (192) [very thick, color=red];
\draw (191) -- (194) [very thick, color=red];
\draw (192) -- (198) [very thick, color=red];
\draw (193) -- (194) [very thick, color=red];
\draw (5) -- (6) [dashed, very thin, color=teal];
\draw (10) -- (43) [dashed, very thin, color=teal];
\draw (11) -- (58) [dashed, very thin, color=teal];
\draw (12) -- (152) [dashed, very thin, color=teal];
\draw (13) -- (193) [dashed, very thin, color=teal];
\draw (14) -- (161) [dashed, very thin, color=teal];
\draw (15) -- (192) [dashed, very thin, color=teal];
\draw (16) -- (4) [dashed, very thin, color=teal];
\draw (18) -- (102) [dashed, very thin, color=teal];
\draw (19) -- (1) [dashed, very thin, color=teal];
\draw (20) -- (1) [dashed, very thin, color=teal];
\draw (21) -- (46) [dashed, very thin, color=teal];
\draw (22) -- (33) [dashed, very thin, color=teal];
\draw (23) -- (175) [dashed, very thin, color=teal];
\draw (24) -- (184) [dashed, very thin, color=teal];
\draw (25) -- (106) [dashed, very thin, color=teal];
\draw (26) -- (0) [dashed, very thin, color=teal];
\draw (28) -- (35) [dashed, very thin, color=teal];
\draw (29) -- (132) [dashed, very thin, color=teal];
\draw (30) -- (109) [dashed, very thin, color=teal];
\draw (31) -- (65) [dashed, very thin, color=teal];
\draw (36) -- (65) [dashed, very thin, color=teal];
\draw (37) -- (58) [dashed, very thin, color=teal];
\draw (38) -- (9) [dashed, very thin, color=teal];
\draw (42) -- (126) [dashed, very thin, color=teal];
\draw (47) -- (44) [dashed, very thin, color=teal];
\draw (48) -- (96) [dashed, very thin, color=teal];
\draw (49) -- (187) [dashed, very thin, color=teal];
\draw (50) -- (54) [dashed, very thin, color=teal];
\draw (52) -- (194) [dashed, very thin, color=teal];
\draw (56) -- (111) [dashed, very thin, color=teal];
\draw (59) -- (4) [dashed, very thin, color=teal];
\draw (60) -- (105) [dashed, very thin, color=teal];
\draw (62) -- (53) [dashed, very thin, color=teal];
\draw (63) -- (98) [dashed, very thin, color=teal];
\draw (64) -- (125) [dashed, very thin, color=teal];
\draw (66) -- (51) [dashed, very thin, color=teal];
\draw (67) -- (51) [dashed, very thin, color=teal];
\draw (68) -- (193) [dashed, very thin, color=teal];
\draw (69) -- (180) [dashed, very thin, color=teal];
\draw (70) -- (184) [dashed, very thin, color=teal];
\draw (71) -- (9) [dashed, very thin, color=teal];
\draw (72) -- (41) [dashed, very thin, color=teal];
\draw (73) -- (132) [dashed, very thin, color=teal];
\draw (74) -- (175) [dashed, very thin, color=teal];
\draw (75) -- (55) [dashed, very thin, color=teal];
\draw (77) -- (99) [dashed, very thin, color=teal];
\draw (79) -- (99) [dashed, very thin, color=teal];
\draw (80) -- (78) [dashed, very thin, color=teal];
\draw (81) -- (90) [dashed, very thin, color=teal];
\draw (82) -- (158) [dashed, very thin, color=teal];
\draw (83) -- (157) [dashed, very thin, color=teal];
\draw (84) -- (190) [dashed, very thin, color=teal];
\draw (85) -- (167) [dashed, very thin, color=teal];
\draw (86) -- (130) [dashed, very thin, color=teal];
\draw (88) -- (106) [dashed, very thin, color=teal];
\draw (89) -- (46) [dashed, very thin, color=teal];
\draw (91) -- (190) [dashed, very thin, color=teal];
\draw (92) -- (126) [dashed, very thin, color=teal];
\draw (93) -- (102) [dashed, very thin, color=teal];
\draw (94) -- (57) [dashed, very thin, color=teal];
\draw (95) -- (87) [dashed, very thin, color=teal];
\draw (100) -- (34) [dashed, very thin, color=teal];
\draw (101) -- (43) [dashed, very thin, color=teal];
\draw (103) -- (40) [dashed, very thin, color=teal];
\draw (104) -- (167) [dashed, very thin, color=teal];
\draw (107) -- (2) [dashed, very thin, color=teal];
\draw (108) -- (32) [dashed, very thin, color=teal];
\draw (110) -- (125) [dashed, very thin, color=teal];
\draw (112) -- (32) [dashed, very thin, color=teal];
\draw (113) -- (90) [dashed, very thin, color=teal];
\draw (114) -- (0) [dashed, very thin, color=teal];
\draw (115) -- (78) [dashed, very thin, color=teal];
\draw (116) -- (6) [dashed, very thin, color=teal];
\draw (117) -- (179) [dashed, very thin, color=teal];
\draw (118) -- (176) [dashed, very thin, color=teal];
\draw (119) -- (53) [dashed, very thin, color=teal];
\draw (120) -- (3) [dashed, very thin, color=teal];
\draw (121) -- (27) [dashed, very thin, color=teal];
\draw (122) -- (180) [dashed, very thin, color=teal];
\draw (123) -- (35) [dashed, very thin, color=teal];
\draw (124) -- (34) [dashed, very thin, color=teal];
\draw (127) -- (109) [dashed, very thin, color=teal];
\draw (128) -- (130) [dashed, very thin, color=teal];
\draw (129) -- (157) [dashed, very thin, color=teal];
\draw (131) -- (87) [dashed, very thin, color=teal];
\draw (133) -- (45) [dashed, very thin, color=teal];
\draw (134) -- (44) [dashed, very thin, color=teal];
\draw (135) -- (45) [dashed, very thin, color=teal];
\draw (136) -- (8) [dashed, very thin, color=teal];
\draw (137) -- (158) [dashed, very thin, color=teal];
\draw (138) -- (17) [dashed, very thin, color=teal];
\draw (139) -- (191) [dashed, very thin, color=teal];
\draw (140) -- (2) [dashed, very thin, color=teal];
\draw (141) -- (194) [dashed, very thin, color=teal];
\draw (142) -- (179) [dashed, very thin, color=teal];
\draw (143) -- (187) [dashed, very thin, color=teal];
\draw (144) -- (8) [dashed, very thin, color=teal];
\draw (146) -- (199) [dashed, very thin, color=teal];
\draw (147) -- (198) [dashed, very thin, color=teal];
\draw (148) -- (105) [dashed, very thin, color=teal];
\draw (149) -- (57) [dashed, very thin, color=teal];
\draw (150) -- (97) [dashed, very thin, color=teal];
\draw (151) -- (198) [dashed, very thin, color=teal];
\draw (153) -- (61) [dashed, very thin, color=teal];
\draw (155) -- (199) [dashed, very thin, color=teal];
\draw (156) -- (7) [dashed, very thin, color=teal];
\draw (159) -- (61) [dashed, very thin, color=teal];
\draw (160) -- (161) [dashed, very thin, color=teal];
\draw (162) -- (178) [dashed, very thin, color=teal];
\draw (163) -- (191) [dashed, very thin, color=teal];
\draw (164) -- (39) [dashed, very thin, color=teal];
\draw (165) -- (111) [dashed, very thin, color=teal];
\draw (166) -- (178) [dashed, very thin, color=teal];
\draw (168) -- (55) [dashed, very thin, color=teal];
\draw (169) -- (96) [dashed, very thin, color=teal];
\draw (170) -- (3) [dashed, very thin, color=teal];
\draw (171) -- (192) [dashed, very thin, color=teal];
\draw (172) -- (40) [dashed, very thin, color=teal];
\draw (173) -- (33) [dashed, very thin, color=teal];
\draw (174) -- (41) [dashed, very thin, color=teal];
\draw (177) -- (176) [dashed, very thin, color=teal];
\draw (181) -- (76) [dashed, very thin, color=teal];
\draw (182) -- (27) [dashed, very thin, color=teal];
\draw (183) -- (17) [dashed, very thin, color=teal];
\draw (185) -- (54) [dashed, very thin, color=teal];
\draw (186) -- (76) [dashed, very thin, color=teal];
\draw (188) -- (98) [dashed, very thin, color=teal];
\draw (189) -- (152) [dashed, very thin, color=teal];
\draw (195) -- (7) [dashed, very thin, color=teal];
\draw (196) -- (39) [dashed, very thin, color=teal];
\draw (197) -- (97) [dashed, very thin, color=teal];

\end{tikzpicture}
\begin{center}An optimal tree with $L=173.00$.\end{center}
\end{adjustbox}
\caption{Optimal solutions (tour subgraph on the left and its tree counterpart on the right) of a further instance generated from p5, with $r_{v}=65.82$ and $c_{v}=2$ for all $v\in V$.}
\label{final_fig_22}
\end{figure}

\subsection{Impact of the symmetry-breaking inequalities}
\label{comp_sym}

The shape of the subgraphs also plays a role in the difficulty of the solution process. The numerical values of the TSP and the MST are quite similar for our instances, but there are usually (far) fewer feasible (symmetric) tours than trees. This fact leads to the tree subgraph case containing many more symmetric solutions on average than its tour counterpart, which in turn leads to higher computational times for the tree case.

Adding the symmetry-breaking inequalities \eqref{triangle_ineq} to the model results in a larger number of instances being solved within our time limit; however on average, there is an increase in the computational time for those instances that could be solved without inequalities \eqref{triangle_ineq} already, due to the larger size of the resulting model. We opted to prioritize the number of instances solved over the average computational times, and hence the choice to keep these inequalities in our model. Table \ref{table_triangles} shows how both the average computational time and the number of solved instances increase.

We observe an analogous behaviour for the tour subgraph case as well; however, in this case, the (slightly) higher number of solved instances does not seem to outweigh the considerable increase in the average computational times, as shown in Table \ref{table_2opt}. Hence for the tour subgraph case, the symmetry-breaking inequalities \eqref{tour_quad} are far less effective than their tree counterpart and should only be used when it is critical to maximize the number of instances solved to optimality.

\begin{table}[t]
\centering
\begin{tabular}{ccccc}
& without \eqref{triangle_ineq} && \multicolumn{2}{c}{with \eqref{triangle_ineq}}\\\cmidrule{2-2}\cmidrule{4-5}
& solved && solved & CPU\\\cmidrule[.1em]{1-5}
p5 & $22$ && $\textbf{28}$ & $+32\%$\\
X-n195-k11 & $20$ && $\textbf{30}$ & $+74\%$\\
kroa200 & $18$ && $\textbf{22}$ & $+4\%$
\end{tabular}
\caption{Impact of the symmetry-breaking inequalities on the tree BPCCSP\@. For the three largest base instances, we solve all the resulting instances such that $\nicefrac{q_{vw}}{p_{v}}=0.5$ via branch-and-cut. We report the number of solved instances (out of $36$) and the average increase in the computational time.}
\label{table_triangles}
\end{table}

\begin{table}[t]
\centering
\begin{tabular}{ccccc}
& without \eqref{tour_quad}&& \multicolumn{2}{c}{with \eqref{tour_quad}}\\\cmidrule{2-2}\cmidrule{4-5}
& solved && solved & CPU\\\cmidrule[.1em]{1-5}
p5 & $35$ && $\textbf{36}$ & $+284\%$\\
X-n195-k11 & $33$ && $\textbf{35}$ & $+103\%$\\
kroa200 & $\textbf{33}$ && $31$ & $+332\%$
\end{tabular}
\caption{Impact of the symmetry breaking inequalities on the tour BPCCSP\@. For the three largest base instances, we solve all the resulting instances such that $\nicefrac{q_{vw}}{p_{v}}=0.5$ via branch-and-cut. We report the number of solved instances (out of $36$) and the average increase in the computational time.}
\label{table_2opt}
\end{table}

\section{Conclusion}
We have introduced a new class of budgeted prize-collecting problems and developed a branch-and-cut framework for their exact solution. We have furthermore proposed a Benders decomposition for the case with independent prizes, as an alternative to the branch-and-cut for low values of the budget $L$. We have validated our algorithms on realistic-sized instances of the the tree and the tour case. Finally, we have identified novel symmetry-breaking inequalities for both tree and tour subgraphs.

As presented in this work, the Benders decomposition only solves instances with independent covering prizes. We believe that new interesting results could be obtained by incorporating \emph{dependent} covering prizes into the Benders decomposition, the challenge being not to slow down the convergence of the resulting Benders cuts.

Finally, for both the tour and the tree subgraph cases, the formulation can be extended by using a larger number of variables to account for an arc and a coverage lower-layer star at the same time. This idea opens the door to column generation approaches, which we believe promising for future investigation.

\section*{Acknowledgements}
The authors wish to thank Ivana Ljubi\'c (ESSEC Business School, Paris) and Jannik Matuschke (KU Leuven) for their valuable comments and suggestions on a preliminary draft of this work.


\end{document}